\RequirePackage{fix-cm}

\documentclass[11pt, reqno,sumlimits]{amsart}
\usepackage{amssymb,amscd, epsfig, mathrsfs, xypic, amsmath,amsthm, dsfont,  color} 
\usepackage{hyperref}
\usepackage{blindtext}

\xyoption{all}
\newtheorem{theorem}{Theorem}[section]  
\newtheorem{corollary}[theorem]{Corollary} 
\newtheorem{lemma}[theorem]{Lemma}  
\newtheorem{proposition}[theorem]{Proposition} 
\newtheorem{df-pr}[theorem]{Definition-Proposition}

\newtheorem{remark}[theorem]{Remark}

\theoremstyle{definition} 
\newtheorem{definition}[theorem]{Definition}

\newtheorem{example}[theorem]{Example}

\newcommand{\lmd}{\lambda}

\newcommand{\CC}{{\mathbb C}}

\newcommand{\ZZ}{{\mathbb Z}}

\newcommand{\sfp }{{pr_k}}

\newcommand{\sfu }{{u}}
\newcommand{\sfv }{{v}}

\newcommand{\frakt}{{\mathfrak  t}}

\newcommand{\calA}{{\mathcal A}}

\newcommand{\calI}{{\mathcal I}}

\newcommand{\calP}{{\mathcal P}}

\newcommand{\calR}{{\mathcal R}}

\newcommand{\ee }{\pmb{e}}

\newcommand{\Span}{\operatorname{Span}}

\newcommand{\Map}{\operatorname{Map}}

\newcommand{\inv}{{\operatorname{inv}}}

\newcommand{\Loc}{{\Phi}}

\newcommand{\frakC}{{\mathfrak{C}}}

\newcommand{\Det}{{\operatorname{Det}}}
\newcommand{\Pf}{{\operatorname{Pf}}}
\newcommand{\vt}{{\vartheta}}

\newsavebox{\savepar}

\numberwithin{equation}{section}

\newcounter{labelflag} 
\newcommand{\labelon}{\setcounter{labelflag}{1}}
\newcommand{\Label}[1]{\ifnum\thelabelflag=1\ifmmode
\makebox[0in][l]{\qquad\fbox{\rm#1}} \else
\marginpar{\vspace{0.7\baselineskip} \hspace{-1.1\textwidth}
\fbox{\rm#1}} \fi \fi \label{#1} } \labelon

\newcommand{\Maya}
{\put(0,0){\line(1,0){40}}
  \put(0,10){\line(1,0){40}}
  \multiput(0,0)(10,0){5}{\line(0,1){10}}
}

\begin{document}

\title{
Pfaffian sum formula for the symplectic Grassmannians
}
\author{Takeshi Ikeda}
\address{Department of Applied Mathematics\\
Okayama University of Science,
Okayama 700-0005, Japan}
\email{ike@xmath.ous.ac.jp}     
\author{Tomoo Matsumura}
\address{Department of Mathematical Sciences\\
KAIST\\
Daejeon, South Korea}
\email{tomoomatsumura@kaist.ac.kr}


%


\maketitle 

\begin{abstract}
We study the torus equivariant Schubert classes of the Grassmannian of non-maximal isotropic subspaces in a symplectic vector space. We prove a formula that expresses each of those classes as a \emph{sum} of multi Schur-Pfaffians, whose entries are equivariantly modified special Schubert classes. Our result gives a proof to Wilson's conjectural formula, which generalizes the Giambelli formula for the ordinary cohomology proved by Buch-Kresch-Tamvakis, given in terms of Young's raising operators. Furthermore we show that the formula extends to a certain family of Schubert classes of the symplectic partial isotropic flag varieties. 
\end{abstract}

\section{Introduction}
The classical Giambelli formula \cite{G} expresses a general Schubert class of the Grassmannian as the determinant of a matrix whose entries are the so-called \emph{special Schubert classes}. A special Schubert class is defined by the locus of of subspaces having excess intersection with a fixed subspace. These classes also coincide with the Chern classes of the universal quotient bundle over the Grassmannian. Various extensions of the formula have been obtained (see for example \cite{FP}, \cite{T} and the references therein). The \emph{Giambelli problem} consists in finding a ``closed formula'' for a Schubert class in terms of those special classes.
Note that, in the torus equivariant setting, the problem is closely related to the theory of degeneracy loci of vector bundles (\textit{cf.} \cite{And}, \cite{FP}, \cite{T}). 

For the symplectic or orthogonal Grassmannians, there is a natural notion of special Schubert classes, which takes into account the isotropic condition. For the Grassmannian of maximal isotropic subspaces, the Giambelli formula, first found by Pragacz \cite{Pr}, expresses  a general Schubert class as a Pfaffian whose entries are appropriate quadratic polynomials in the special Schubert classes. Its natural equivariant version was obtained by Kazarian \cite{Ka} in the context of degeneracy loci (see \cite{KT}, \cite{PR} for other versions, and also the survey articles \cite{FP}, \cite{T} and references therein).
The Kazarian's formula was rediscovered by Ikeda in \cite{Ik} and Ikeda-Naruse in \cite{IN} and was interpreted in the context of the torus equivariant cohomology.
For the non-maximal isotropic Grassmannians, an answer to the (non-equivariant) Giambelli problem was given by  Buch, Kresch, and Tamvakis (\cite{BKT}, \cite{BKT2}). Their formula expresses an arbitrary Schubert class 
by means of \emph{Young's raising operators}. We can regard this polynomial expression as a certain ``combinatorial  interpolation'' between the Jacobi-Trudi determinant and the Schur Pfaffian. 

This paper is concerned with the \emph{equivariant} Giambelli problem for the non-maximal isotropic Grassmannians in the symplectic case. In \cite{W} , Wilson employed the raising operators to define {\em double theta polynomials\/} $\Theta_\lambda$, and proved that these polynomials satisfy the equivariant Chevalley formula for the non-maximal symplectic Grassmannian. In \cite{W}, it was further conjectured that $\Theta_\lambda$ would equal to the \emph{double Schubert polynomial} introduced in [10] (see \S \ref{ssec:DSP}).

The main result of this paper provides a formula expressing each double Schubert polynomial associated to the isotropic Grassmannians as a sum of Pfaffians whose entries are Wilson's double theta polynomials corresponding to the special Schubert classes. This immediately leads to a proof of Wilson's conjecture, because the raising operator formula can be rewritten as a Pfaffian sum by a formal computation (see \S \ref{sec: IM-TW} for details). Note that the equivalence of the two formulas in the non-equivariant situation was already included in the proof of Proposition 2 of \cite{BKT3}. In this sense, the non-equivariant version of the Pfaffian sum formula for the symplectic Grassmannian was first obtained in \cite{BKT3}.

Our method for the proof of the main result is to use
the \emph{left divided difference operators}. 
These operators are essential in the theory of (double) Schubert polynomials, and exist only in the equivariant setup. This technique allows us to completely avoid using the raising operators, and more importantly, the technique is applicable to more general contexts. In particular, we extend the Pfaffian sum formula beyond Grassmannians, namely to some part of the Schubert classes of the partial isotropic flag varieties. In this extended Pfaffian sum formula, each entry can be regarded as equivariantly modified special Schubert classes arising from the symplectic Grassmannians of various dimensional subspaces.

Below we summarize our results in more details.

\subsection{Symplectic Grassmannian and its Schubert varieties}
Throughout the paper, we fix a non-negative integer $k.$ For any positive integer $n\geq k$, let  ${SG}_n^k$ denote the Grassmannian of $(n-k)$-dimensional isotropic subspaces in $\CC^{2n}$ equipped with a symplectic form. There is  a maximal parabolic subgroup $P_k$ of the symplectic group $G:={Sp}_{2n}(\CC)$ such that ${SG}_n^k$ can be realized as the homogeneous space $G/P_k$.

A partition $\lambda$ is $k$-{\it strict\/} if no part greater than $k$ is repeated. The Schubert varieties of ${SG}_n^k$ are indexed by the $k$-strict partitions whose Young diagrams fit in the $(n-k)\times(n+k)$ rectangle. We denote the set of such partitions  by $\mathcal{P}_n^{(k)}.$ Given $\lambda\in \mathcal{P}_n^{(k)}$ and a complete flag of subspaces $0=F_0\subset F_1\subset \cdots \subset F_{2n}=\CC^{2n}$ such that $F_{n+i}=(F_{n-i})^{\perp}$ for $0\leq i\leq n$, the corresponding Schubert variety is defined as 
\begin{equation}\label{def:Omega_lambda1}
\Omega_\lambda=\{V\in {SG}_n^k\;|\;
\dim(V\cap F_{p_j(\lambda)})\geq j, \quad 1\leq \forall j\leq 
\ell(\lambda)\},
\end{equation}
where $\ell(\lambda)$ denotes the number of non-zero parts of $\lambda$, and 
\[
p_j(\lambda)=n+k+j-\lambda_j -\#\{i\;|\;i<j,\;
\lambda_i+\lambda_j>2k+j-i\}.
\]
See \cite{BKT}, for example. The codimension of $\Omega_\lambda$ is 
$|\lambda|=\sum_{i}\lambda_i$.
The corresponding class $[\Omega_\lambda]\in H^*({SG}_n^k)$ is the Schubert
class. 
The {\it special Schubert varieties\/} are the ones associated to one-line $k$-strict partitions:
\[
\Omega_r=\{V\in {SG}_n^k\;|\;
\dim(V\cap F_{n+k+1-r})\geq 1
\},
\]
for $1\leq r\leq n+k$. Their classes $[\Omega_r]$ are called the {\it special Schubert classes\/.} They are equal to the $r$-th Chern classes $c_r(\mathcal{Q})$ of the universal  quotient bundle $\mathcal{Q}$ over ${SG}_n^k.$
\subsection{Double Schubert polynomials of type $C$}\label{ssec:DSP}
We recall the results in \cite{DSP}, where the double Schubert polynomials of type $C$ were introduced. Let $B$ be a Borel subgroup of $G={Sp}_{2n}(\CC),$ and $T$ the maximal torus  contained in $B$. The Weyl group $N_G(T)/T$ is denoted by $W_n$ and identified with the  group of the  signed permutations of  $\{\pm 1,\ldots,\pm n\}.$ We often denote $-i$ by $\overline{i}.$ The flag variety $\mathcal{F}l_n$ is  defined as the quotient space $G/B.$ For each $w\in W_n$, the Schubert variety $X_w$ is defined as the Zariski closure of $B^{-}$-orbit of  the corresponding point $e_w \in \mathcal{F}l_n$, where $B^{-}$ is the Borel subgroup such that $B\cap B^{-}=T$. The codimension of $X_w$ is precisely the length $\ell(w)$ of $w$ as a Weyl group element of $W_n$. We denote by $[X_w]_T$ the corresponding $T$-equivariant  Schubert class in $H_T^*(\mathcal{F}l_n)$.

Let $\Gamma$ be the  ring generated over $\ZZ$ by the Schur $Q$-functions $Q_r(x)\;(r\geq 1),$ where $x$ is an infinite sequence of variables  $x_1,x_2,\ldots$.  Let $\mathcal{R}_\infty$ be the  polynomial ring $\Gamma[z,t]$ in  the variables $z_i,t_i\;(i\geq 1)$ with coefficients  in $\Gamma.$  Let $W_\infty$ be the Weyl group of type $C_\infty$, where we can regard it as the union of $W_n\;(n\geq 1).$ There are two commuting actions of $W_\infty$ on  the ring $\mathcal{R}_\infty$ (see \S \ref{sec:DSP}).  The {\it double Schubert polynomials\/} are defined as elements of a distinguished $\ZZ[t]$-basis $\{\mathfrak{C}_w(z,t;x)\;|\;w\in W_\infty\}$ of $\mathcal{R}_\infty$. They are characterized by two series of operators $\{\delta_i\;|\;i\geq 0\}$ and $\{\partial_i\;|\;i\geq 0\}$ called the left and the right divided difference operators respectively (See Theorem \ref{DSPTheorem} in Section \ref{sec:DSP}).   

The integral $T$-equivariant cohomology ring $H_T^*(\mathcal{F}l_n)$ of $\mathcal{F}l_n$ has an $H_T^*(pt)$-algebra structure given by the pullback of $\mathcal{F}l_n\rightarrow pt$. Together with an appropriate identification $H_{T}^*(pt)=\ZZ[t_1,\ldots,t_n]$, there is a canonical homomorphism
\[
\pi_n: 
\mathcal{R}_\infty
\longrightarrow H_T^*(\mathcal{F}l_n)
\]
of $H_{T}^*(pt)$-algebras such that $\pi_n$ sends $\mathfrak{C}_w(z,t;x)$ to $[X_w]_T $ if $w\in W_n$ and to zero  if $w\not\in W_n.$
\subsection{Equivariant Schubert classes  of $H_T^*({SG}_n^k)$} 
Let $s_i\;(i\geq 0)$ be the standard (Coxeter) 
generators of $W_\infty$, usually referred to as the simple reflections (see \S  \ref{sec:DSP}).
Let $W_{(k)}$ be the subgroup of $W_\infty$ generated by  $s_i\;(i\geq 0,\;i\neq k).$ Let $\mathcal{R}_\infty^{(k)}$ denote the  invariant subring of $\mathcal{R}_\infty$ with respect to  the 2nd (``right'') action of $W_{(k)}.$ There is the following commutative diagram 
\[ 
\xymatrix{ \mathcal{R}_\infty^{(k)}
\ar[d]^{{\pi}_n^{(k)}} \ar[rr]^{} &&  \mathcal{R}_\infty
\ar[d]^{{\pi}_n} \\ 
H_T^*({SG}_n^k) \ar[rr]_{\ \ \ \ \ \ \sfp^*} && 
H_T^*(\mathcal{F}l_n), }
\]
where the horizontal arrow $\sfp^*$ in the  second row 
is 
the pullback of  the natural projection $\sfp: \mathcal{F}l_n\rightarrow {SG}_n^k$, and $\pi_n^{(k)}$ is obtained by restricting $\pi_n$ to $\mathcal{R}_\infty^{(k)}$. 

Let $W^{(k)}_\infty$ be the set of  the minimum-length coset representatives for $W_\infty/W_{(k)}.$ 
We denote the set of all $k$-strict partitions by $\mathcal{P}^{(k)}_\infty=\bigcup_{n\geq k}
\mathcal{P}_n^{(k)}$.
There is a natural 
bijection
$$
\mathcal{P}^{(k)}_\infty\longrightarrow 
W^{(k)}_\infty,\quad \lambda\mapsto w_\lambda^{(k)},
$$
such that $|\lambda|=\ell(w_\lambda^{(k)})$ and the image of $\mathcal{P}_n^{(k)}$ is $W_n^{(k)}:=W_n\cap W^{(k)}_\infty$. If $\lambda\in  \mathcal{P}_n^{(k)}$, we have
${\sfp}^*[\Omega_\lambda]_T=[X_{w_\lambda^{(k)}}]_T$ and 
$$
\pi_n^{(k)}(
\mathfrak{C}_{w_{\lambda}^{(k)}}(z,t;x))=
[\Omega_\lambda]_T. 
$$
In particular, the special Schubert class $[\Omega_r]_T$ of degree $r$ is the image of $\mathfrak{C}_{w_r^{(k)}}(z,t;x),$ where $w_r^{(k)}$ is the element of $W^{(k)}_\infty$ corresponding to the partition with $r$ boxes in one row. The set of the functions $\mathfrak{C}_{w}(z,t;x),\;w\in W^{(k)}_\infty$ forms a $\ZZ[t]$-basis of  $\mathcal{R}_\infty^{(k)}$
(Proposition \ref{prop:basisR^k}).  
\subsection{Main results} 
Our goal is to  give an explicit closed formula describing $\mathfrak{C}_{w}(z,t;x)\;(w\in W^{(k)}_\infty)$ as a polynomial in terms of the  double Schubert polynomials 
corresponding to the special classes $[\Omega_r]_T\;(r\geq 1)$. 
\begin{definition}\label{def-theta}
Define ${}_k\vt_r^{(l)}(x,z|t)$ for $l, r\geq 0$ by
\begin{eqnarray*}
\sum_{r=0}^{\infty} \ _k\vt_r^{(l)}(x,z|t) \cdot u^r &=& \prod_{i=1}^{\infty}  \frac{1+x_i u }{1-x_i u }  \prod_{i=1}^k (1 + z_i u ) \prod_{i=1}^l (1 - t_i u), \\
\sum_{r=0}^{\infty}\ _k\vt_r^{(-l)}(x,z|t) \cdot u^r &=& \prod_{i=1}^{\infty}  \frac{1+x_i u }{1-x_i u }  \prod_{i=1}^k (1 + z_i u ) \prod_{i=1}^l \frac{1}{1 + t_i u}.\\
\end{eqnarray*} 
For $r < 0$, we set ${}_k\vt_r^{(l)}(x,z|t)=0$. We omit $k$ when it is made clear by the context. Under $\pi_n$, the variables $z_i$ correspond to the Chern roots of the tautological bundles and the theta polynomials ${}_k\vt_r^{(l)}(x,z|t)$ map to the Chern classes of certain virtual bundles (see Proposition \ref{theta geometry}). Although the above definition of the double theta polynomials appears slightly different from the one in Wilson's thesis \cite{W} (also mentioned  in \cite{T}), one recovers Wilson's definition after applying appropriate changes of indices. See Remark \ref{theta of TW and IM} for details.
\end{definition}
We show that 
$\mathcal{R}_\infty^{(k)}$ contains ${}_k\vt_r^{(l)}(x,z|t)$ (see \S \ref{ssec:InvariantRing}) and can be described as follows (Corollary \ref{rem:generated by theta}):
\begin{equation}
\mathcal{R}_\infty^{(k)}=
\ZZ[t][\vt_1,\vt_2,\ldots],\quad
\vt_r:={}_k\vt_r^{(0)}.
\end{equation}
Wilson proved the following fact.
\begin{theorem}[\cite{W}]We have
\begin{equation} 
\mathfrak{C}_{w_r^{(k)}}(z,t;x)={}_k\vartheta_{r}^{(r-k-1)}(x,z|t).
\label{eq:vtr}
\end{equation}
\end{theorem}
Let $\lambda$ be a $k$-strict partition in $\mathcal{P}_n^{(k)}$.  In the one-line notation of signed permutations (see \S \ref{sec:DSP}), we can write the corresponding $k$-Grassmannian element $w_{\lambda}^{(k)}$ of $W_n^{(k)}$ as  
\begin{eqnarray}
w_\lambda^{(k)}=\sfv_1\sfv_2\ldots \sfv_k | \overline{\zeta_1} \cdots \overline{\zeta_s}
\sfu_{1}\cdots \sfu_{n-k-s} ,\nonumber\\
0<\sfv_1<\cdots<\sfv_k,\quad
\overline{\zeta_1}<\cdots<\overline{\zeta_s} <0<\sfu_{1}<\cdots<\sfu_{n-k-s},\label{eq:Wseq}
\end{eqnarray}
where $s$ is a non-negative integer. Let $\chi_\lambda=(\chi_1,\ldots,\chi_{n-k})$ be the following sequence
\begin{equation}
\chi_\lambda:=(\zeta_1-1,\ldots,\zeta_s-1,-\sfu_{1},\ldots, -\sfu_{n-k-s})\in \ZZ^{n-k}.\label{eq:upperindex}
\end{equation}
We call $\chi_\lambda$ the {\it characteristic index\/} of $\lambda$. For a positive integer $m$, let $\Delta_m:=\{ (i,j) \ |\ 1\leq i < j \leq m\}.$
Define a subset $D(\lambda)$ of $\Delta_{n-k}$ by
\begin{equation}
D(\lambda) :=\{(i,j)\in \Delta_{n-k}\;|\;\chi_i+\chi_j < 0\}.
\end{equation}
We use the {\it multi Schur-Pfaffian\/} due to Kazarian \cite{Ka}, which is a natural variation of  the Schur Pfaffian appearing in \cite{Schur}. Let $(c^{(1)}, c^{(2)}, \dots ,c^{(m)})$ be an $m$-tuple such that each $c^{(i)}$
is an infinite sequence of variables $c_r^{(i)}\;(r\in \ZZ)$. The multi Schur-Pfaffian $\Pf[c_{r_1}^{(1)}c_{r_2}^{(2)}\cdots c_{r_m}^{(m)}]$ is defined in \S \ref{sec:Pf}. 
This is a finite $\ZZ$-linear combination of $c_{s_1}^{(1)}\cdots c_{s_m}^{(m)}, \; (s_1,\cdots,s_m) \in \ZZ^m$. For each $l=(l_1,\ldots,l_m) \in \ZZ^m$, the substitution of $\vt_{s_i}^{(l_i)}$ to $c_{s_i}^{(i)}$ in this linear combination is denoted by
\[
\Pf[\vt_{r_1}^{(l_1)}\vt_{r_2}^{(l_2)}\cdots \vt_{r_m}^{(l_m)}].
\]
The main result of this paper is the following. 
\begin{theorem}[Pfaffian sum formula, Theorem 6.6 below]\label{thm:Theta=C} 
Let $\lambda$ be a $k$-strict partition in $\mathcal{P}_n^{(k)}$, and $\chi$ the corresponding characteristic index.
We have 
\begin{equation}\label{eq:C=Theta}
\mathfrak{C}_{w_\lambda^{(k)}}= \sum_{I\subset {D}(\lambda)} \Pf\left[\vt_{\lambda_1+a^I_1}^{(\chi_1)} \cdots  \vt_{\lambda_{n-k}+a^I_{n-k}}^{(\chi_{n-k})}\right],%
\end{equation}
where  $I$ runs over all subsets  of $D(\lambda)$ and  $a_s^I= \#\{j\;|\;(s,j)\in I\} -\#\{i\;|\;(i,s)\in I\}.$
\end{theorem}
Note that the right hand side does not depend on $n$, \textit{i.e.}, it depends only  on $\lambda\in \mathcal{P}_\infty^{(k)}$. See Remark \ref{rem:n-indep} for a more precise  statement.
\begin{example} Let $k=1,n=5$. Let 
$\lambda=(5,3,2,1)$ be a $k$-strict partition. Then 
$w_\lambda^{(1)}=5|\bar{4}\bar{2}\bar{1}3$
and 
${D}(\lambda)=\{(2,4),(3,4)\}$.
We have 
\[
\frakC_{w_{(5,3,2,1)}^{(1)}}=
\Pf[\vt^{(3)}_5\vt^{(1)}_3\vt^{(0)}_2\vt_1^{(-3)}]
+
\Pf[\vt^{(3)}_5\vt^{(1)}_4\vt^{(0)}_2\vt_0^{(-3)}]
+
\Pf[\vt^{(3)}_5\vt^{(1)}_3\vt^{(0)}_3\vt_0^{(-3)}].
\]
\end{example}
 
Note that, even if $D(\lambda)\not=\emptyset$, it is possible that the double Schubert polynomial is a single Pfaffian in the formula. For example, we have 
$
\frakC_{13|\bar5\bar42} = \Pf[\vt_7^{(4)}\vt_5^{(2)}\vt_0^{(-4)}] + \Pf[\vt_7^{(4)}\vt_4^{(2)}\vt_{-1}^{(-4)}] $
according to the formula, but the second term is zero (\textit{cf.} Remark \ref{rem: n-ind theta}), and therefore
$\frakC_{13|\bar5\bar42}=\Pf[\vt_7^{(4)}\vt_5^{(2)}\vt_0^{(-4)}].$

\bigskip

Once we read the formula in terms of raising operators, the following corollary is immediate.

\begin{corollary}\label{Pf-Det}
If $D(\lambda)=\Delta_{n-k}$,  in particular, if $\lambda$ is contained in the $(n-k)\times k$ rectangle, then $\frakC_{w_{\lambda}^{(k)}}$ is a single determinant 
\begin{equation}
\mathrm{Det}[\vt_{\lambda_1}^{(\chi_1)}\cdots
\vt_{\lambda_{n-k}}^{(\chi_{n-k})}]:=
\det(\vt_{\lambda_i+j-i}^{(\chi_i)})_{1\leq i,j\leq n-k}.
\label{eq:small}
\end{equation}
If $D(\lambda)=\emptyset$, in particular, if $\lambda$ is a strict partition containing the $(n-k)\times k$ rectangle, then $\frakC_{w_{\lambda}^{(k)}}$ is a single Pfaffian
\[
\mathrm{Pf}[\vt_{\lambda_1}^{(\chi_1)}\cdots
\vt_{\lambda_{n-k}}^{(\chi_{n-k})}].\]
\end{corollary}

The case when $\lambda$ is contained in the $(n-k)\times k$ rectangle
was considered by Wilson. She calls such partition a ``small'' partition
and proved in \cite[Section 5.8]{W} that the corresponding double 
theta polynomial, written as the determinant (\ref{eq:small}), 
satisfies an appropriate vanishing property.

It is straightforward to apply our results to the problem of degeneracy loci formulas of vector bundles (\textit{cf.} \cite{And}, \cite{AF}, \cite{T}).  As the simplest example, we provide a Chern class interpretation of ${}_k\vt_r^{(l)}(z,x|t)$ in \S \ref{sec:Chern}. It is worth mentioning that the special cases appearing in Corollary \ref{Pf-Det} look precisely the same as the classical Kempf-Laksov determinantal formula for type A degeneracy loci \cite{KeLa} and the Pfaffian formula for Lagrangian degeneracy loci \cite{Ik}, \cite{Ka} (see also Remark \ref{rem: rel to Ik and Ka}), although the functions are associated to the isotropic Grassmannians.  

In the proof of Theorem \ref{thm:Theta=C}, the left  divided difference operators $\delta_i$ play an essential role.  By a direct calculation, we show that  the right hand side of (\ref{eq:C=Theta}) satisfies the part of the defining properties  of the double Schubert polynomials corresponding to the left divided difference operators. To finish the proof, we then make use of a uniqueness lemma (Lemma \ref{uniqueness}) which is available for the parabolic case.

\subsection{Beyond Grassmannians}
The method described above to derive the Pfaffian sum formula for double Schubert polynomials works beyond the $k$-Grassmannian elements. Namely, we first derive a single Pfaffian formula for the top class for symplectic partial flag variety (Theorem \ref{top block k-Pf}). Then we introduce a certain family of signed permutations (\emph{pseudo $k$-Grassmannian elements}, see Definition \ref{def: pseudo Gr}) and derive the Pfaffian sum formula for them (Theorem \ref{NEW THEOREM INTRO}). Those signed permutations form a subset of the minimum-length coset representatives for non-maximal parabolic cases, \textit{i.e.} they correspond to certain Schubert varieties of the isotropic  partial flag varieties. In this case, one can regard the entries of the Pfaffians  as special classes arising from various isotropic Grassmannians. 

Furthermore, we can go beyond the pseudo $k$-Grassmannian elements. 
For example, 
among all 48 signed permutations in $W_3$, 16 of them are not pseudo $k$-Grassmannian elements. It turns out that those 16 are also written as (sums of) Pfaffians, \emph{except} $\bar32\bar1,  \bar321,  \bar 231$ and $\bar132$. 
It would be interesting to study in general to what extent the double Schubert polynomials are written as sums of Pfaffians, and also the geometric or combinatorial conditions for the polynomials to be Pfaffian sums.

\subsection{Related results}
Anderson and Fulton \cite{AF} defined a notion of vexillary signed permutation in type B,C, and D. They showed that the double Schubert polynomials associated to vexillary signed permutations are given by explicit Pfaffian formulas. 
Naruse \cite{N} also independently proved a formula that expresses the corresponding double Schubert polynomials as a specialization of the factorial $Q$- and $P$-functions. Since our formula also express some Schubert classes as single Pfaffians, there is an overlap between our results and the results of \cite{AF}, \cite{N}. However, not all (pseudo) $k$-Grassmannian permutations are vexillary and there are non-vexillary $k$-Grassmannian permutations whose corresponding classes are written as single Pfaffians, \textit{e.g.} $13|\bar5\bar42$ is not vexillary but $\frakC_{13|\bar5\bar42}$ is a single Pfaffian as above. 

Tamvakis \cite{T2} proved a combinatorial formula which expresses an arbitrary (equivariant) Schubert class of any classical $G/P$  space  as a polynomial in the special Schubert classes (see also \cite{T}).  The formula involves some combinatorial data related to the reduced decomposition of Weyl group elements as well as theta polynomials and Schur $S$-functions.

Beside the possibility of extending our methods and results to type D,  it is also natural to ask if our formula can be derived by using Kazarian's pushforward formula. If it is possible, there will arise a new perspective hopefully applicable to K-theory case. We hope to address these problems elsewhere.  

Note added. -- After the submission of this manuscript, 
we have received from Harry Tamvakis a preprint \cite{TW} written by him and Wilson, where they provided a presentation of the equivariant cohomology ring of the symplectic Grassmannian.  
Some of their arguments use our main result in an essential way.
\subsection{Organization}
This paper is organized as follows.
In Section \ref{sec:DSP}, we review the double Schubert polynomials (DSP) following \cite{DSP}.
In Section \ref{PrelimOnGr}, we give 
some preliminary discussions on the symplectic Grassmannian. 
In Section \ref{sec:Pf}, we introduce the multi Schur-Pfaffian used by Kazarian in a slightly generalized form.  
In Section \ref{sec:theta}, we introduce 
the double theta polynomials and establish 
some of their basic properties. 
Section \ref{sec:proof} is devoted to the proof
of Theorem \ref{thm:Theta=C}. 
In Section \ref{sec: IM-TW}, we introduce the raising operators and their action on formal power series to prove the equivalence of our main theorem and the conjecture appearing in Wilson's thesis \cite{W}. 
In Section \ref{sec: beyond grassmannian}, we prove the Pfaffian sum formula for pseudo $k$-Grassmannian permutations.
In Section \ref{sec:exm} for the cases $(n,k)=(5,2), (5,3)$, we provide the expressions for the $k$-Grassmannian elements.

\

\noindent{\bf Acknowledgements:} We are especially grateful to Hiroshi Naruse for explaining his results, and also to Harry Tamvakis for valuable comments to an earlier version of this manuscript. We thank Dave Anderson, Anders Skovsted Buch, Andrew Kresch, Changzheng Li, Leonardo Mihalcea, Masaki Nakagawa for the helpful conversations and their comments. We thank the anonymous referee and Harry Tamvakis for independently pointing out an error of an argument in proving Theorem 4 in a previous version. 
We also thank Thomas Hudson for carefully reading the manuscript. This paper was written for the most part during the first named author's stay at KAIST in 2013. The hospitality and perfect working conditions there are gratefully acknowledged.

\section{Double Schubert polynomials of type $C$}\label{sec:DSP}
In this section, we review the construction of the double Schubert polynomials, following \cite{DSP}.
The expository article \cite{T} by Tamvakis will be also helpful to grasp more geometric backgrounds of this construction. 

\bigskip

Let $W_\infty$ be the group defined by the generators  
$\{s_i\;|\;i=0,1,\ldots\}$
and the relations $$
s_i^2=e\;(i\geq 0) ,\quad
s_0s_1s_0s_1=s_1s_0s_1s_0,\quad
s_is_{i+1}s_i=s_{i+1}s_is_{i+1}\;(i\geq 1),
$$
$$
s_is_j=s_js_i\;(|i-j|\geq 2).
$$
The corresponding Dynkin diagram is the following. 
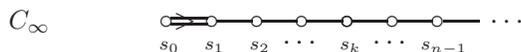
\setlength{\unitlength}{0.4mm}
\begin{center}
  \begin{picture}(200,35)
  \thicklines
    \put(0,15){$\circ$}
  \put(60,15){$\circ$}
  \put(4,16.5){\line(1,0){12}}
  \put(4,18.5){\line(1,0){12}}
    \put(4.9,15){$>$}
  \put(-50,15){\small{$C_\infty$}}
\multiput(15,15)(15,0){4}{
  \put(0,0){$\circ$}
  \put(4,2.4){\line(1,0){12}}}
  \put(75,15){$\circ$}
  \put(90,15){$\circ$}
  \put(0,8){\tiny{$s_0$}}
  \put(15,8){\tiny{$s_1$}}
  \put(30,8){\tiny{$s_2$}}
  \put(60,8){\tiny{$s_{k}$}}
  \put(41,8){{$\cdots$}}
  \put(71,8){{$\cdots$}}
  \put(87,8){\tiny{$s_{n-1}$}}
    \put(79.5,17.4){\line(1,0){11.5}}
    \put(94.5,17.4){\line(1,0){12}}
    \put(110,15){$\cdots$}
  \end{picture}
\end{center}
The group $W_\infty$  is identified with the set of all permutations $w$ of the set $\{1,2,\ldots\}\cup \{\bar{1},\bar{2},\ldots\}$ such that $w(i)\ne i$ for only finitely many $i$, and $\overline{w(i)}=w(\bar{i})$ for all $i$. The simple reflections are identified with the transpositions $s_{0}=(1,\bar{1})$ and $s_{i}=(i+1,i)(\overline{i},\overline{i+1})$ for $i\geq 1.$ The Weyl group $W_n$ is identified with the subgroup of $W_{\infty}$ consisting of $w$ such that $w(i)=i$ for all $i>n$. In one-line notation, we often denote an element $w\in W_{n}$ by the finite sequence $(w(1),\ldots,w(n)).$

The function $Q_r(x)$ is defined by the generating function
\begin{equation}\label{Q_r}
\sum_{r=0}^\infty Q_r(x)u^r=
\prod_{i=1}^{\infty}\frac{1+x_iu}{1-x_iu}
\end{equation}
\textit{i.e.}, $Q_r(x)={}_0\vt_r^{(0)}$. Let $\Gamma$
be the ring
\footnote{
It is well-known that $\Gamma$ can be defined as the quotient of the polynomial ring $\ZZ[Q_1,Q_2,Q_3,\ldots]$ of the variables $Q_1,Q_2,\ldots$  by the ideal generated by the following elements
\begin{equation}\label{def of Q_r}
Q_r^2+2\sum_{i=1}^{r}(-1)^i Q_{r+i}Q_{r-i}\quad (r\geq 1),
\end{equation}
with $Q_0=1.$ This fact follows from \cite[(8.6) (ii), $(8.2')$ and Proof of (8.4) in Chap. III]{Mac}. 
} generated by 
$Q_r(x)\;(r\geq 1).$
Let $\mathcal{R}_\infty$ 
 be the polynomial ring 
 $\Gamma[t,z]$ in the variables
 $t=(t_1,t_2,\ldots),$ and $z=(z_1,z_2,\ldots)$
 with coefficients in $\Gamma.$
There are two actions of $W_\infty$ on  the ring $\mathcal{R}_\infty$ defined below. 
We denote the corresponding 
operators on $\mathcal{R}_\infty$ by 
$s_i^z$ (\emph{right} action) and $s_i^t$ (\emph{left} action). 

For $i\geq 1,$ let  $s_i^z(z_i)=z_{i+1},\;s_i^z(z_{i+1})=z_i,\;s_i^z(z_j)=z_j \;(j\neq i,i+1),$ and $s_i^z(Q_r(x))=Q_r(x).$ There is an automorphism $s_0^z$ of $\ZZ[t]$-algebra  on $\mathcal{R}_\infty$ characterized by the following:
\[
s_0^z(z_1)=-z_1,\quad s_0^z(z_i)=z_i\;(i\geq 1), \quad\sum_{r=0}^\infty s_0^z(Q_r(x))u^r=\frac{1+z_1u}{1-z_1u}\prod_{i=1}^\infty\frac{1+x_iu}{1-x_iu}.
\]
The last equation is equivalently written as 
\[
s_0^z Q_r(x_1,x_2,\ldots)=Q_r(z_1,x_1,x_2,\ldots).
\]
Clearly we can extend $s_i^z$ to $\mathcal{R}_\infty$ as an automorphism of $\ZZ[t]$-algebra and show that $s_i\mapsto s_i^z\;(i\geq 0)$ gives a right action of $W_\infty$ on $\mathcal{R}_\infty.$ Similarly, there are operators $s_i^t\;(i\geq 0)$ on $\mathcal{R}_\infty$ such that $s_i\mapsto s_i^t\;(i\geq 0)$ gives a left  action of $W_\infty$ on $\mathcal{R}_\infty$ as $\ZZ[z]$-algebra automorphisms. In order to define this action, we can  use the following ring automorphism $\omega$:
$$
\omega(t_i)=-z_i,\quad
\omega(z_i)=-t_i,\quad
\omega(Q_r(x))=Q_r(x).
$$
We define $s_i^t=\omega s_i^z \omega$ ($i\geq 0$).
In particular, we have
$$
s_0^tQ_r(x_1,x_2,\ldots)=
Q_r(-t_1,x_1,x_2,\ldots).
$$

Define the {\it simple roots\/} by 
$$
\alpha_0=2t_1,\quad 
\alpha_i=t_{i+1}-t_i\;(i\geq 1).
$$
The right and left divided difference operators
are defined by 
$$
\partial_i f=\frac{f-s_i^z f}{\omega(\alpha_i)},\quad
\delta_i f=\frac{f-s_i^t f}{\alpha_i}\quad (i\geq 0,\;f\in \mathcal{R}_\infty).
$$
\begin{theorem}[\cite{DSP}]\label{DSPTheorem}
\label{existC}
There exists a unique $\ZZ[t]$-free basis  $\{\frakC_w(z,t;x)\;|\;w\in W_{\infty}\} $ of $\calR_\infty$ satisfying the equations
\begin{equation}\label{E:dd} 
\partial_i \frakC_w=
\begin{cases}
\frakC_{ws_i}	&\mbox{if }\  \ell(ws_i)<\ell(w),\\
0			&\mbox{otherwise},
\end{cases} \ \ \ \ \ \ \ 
\delta_i \frakC_{w}=
\begin{cases}
\frakC_{s_iw}	&\mbox{if }\ \ell(s_iw)<\ell(w),\\
0 			&\mbox{otherwise},
\end{cases}
\end{equation}
for all $i\geq 0$, and such that $\frakC_w$ has no constant term except for $\frakC_{e}=1$. 
\end{theorem}
\section{Preliminaries on the symplectic Grassmannian}\label{PrelimOnGr}
\subsection{$k$-strict partitions}\label{k-strict partitions}
We develop some combinatorics related to the Schubert classes  of ${SG}_n^k.$ The set of the minimum-length coset representatives  for $W_\infty/W_{(k)}$ is given by  
\[
W^{(k)}_\infty=\{w\in W_\infty\;|\; \ell(w)>\ell(ws_i)\;(\forall i\geq 0,\;i\neq k)\}.
\]
We will review the bijection $\mathcal{P}^{(k)}_\infty\rightarrow W^{(k)}_\infty $ $(\lambda\mapsto w_\lambda^{(k)})$ such that $|\lambda|=\ell(w_\lambda^{(k)})$, which is due to \cite{BKT} (\textit{cf.} \cite[Section 4.2]{T}). Each $w \in W^{(k)}_\infty$ is called a $k$-Grassmannian permutation, and if $w \in W_n$, we can write $w$ in the one-line notation 
\[
w=\sfv_1\cdots \sfv_k | \overline{\zeta_1}\cdots \overline{\zeta_s} \sfu_1\cdots \sfu_{n-k-s},
\] 
as in (\ref{eq:Wseq}).
For each $i$ with $1\leq i\leq k$, let $\mu_i$ be the number of the elements of $\{\sfu_1, \dots ,\sfu_{n-k-s}\}$ less than $\sfv_{k+1-i}$. Then  $\mu=(\mu_1,\ldots,\mu_k)$ is a partition whose Young diagram fits inside the $k\times (n-s-k)$ rectangle. Let $\nu$ be the conjugate of $\mu$ (the transpose of the Young diagram).
It is worth noting that
\begin{equation}\label{lambda_i+s}
\nu_i=k - \sharp\{ p \ |\ v_p < u_i\}=\sharp\{p \ |\ v_p>u_i\} \ \ (i=1,\cdots,n-k-s).
\end{equation}
Consider the strict partition $\zeta:=(\zeta_1,\dots, \zeta_s)$ defined by the entries with bars in the one-line notation. The $k$-strict partition $\lambda$ corresponding to the $k$-Grassmannian permutation $w$ is given by  
\begin{equation}\label{lambda_i+s 2}
\lambda_i=\begin{cases}
\zeta_i+k&\mbox{if}\; 1\leq i\leq s,\\ 
\nu_{i-s} &\mbox{if}\; s+1\leq i\leq n-k.
\end{cases}
\end{equation}
Note that the $k$-strict partition $\lambda \in \calP_{\infty}^{(k)}$ defined from $w \in W_{\infty}^{(k)}$ as above is independent of the choice of $n$ such that $w \in W_n$.
\begin{example}\label{MAYA}
The $2$-Grassmannian permutation 
$w=58|\bar{4}\bar{3}\bar{1}267$
corresponds to the 
$2$-strict partition 
$\lambda=(6,5,3,2,1,1).$
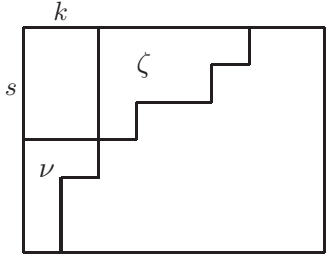
\setlength{\unitlength}{0.5mm}
\begin{center}
  \begin{picture}(80,90)
    \thicklines
  \multiput(0,80)(0,-60){2}{\line(1,0){80}}
  \multiput(0,80)(80,0){2}{\line(0,-1){60}}
  \put(20,80){\line(0,-1){30}}
 \multiput(0,50)(0,-10){1}{\line(1,0){20}}
  \put(60,80){\line(0,-1){10}}
  \put(50,70){\line(0,-1){10}}
  \put(30,60){\line(0,-1){10}}
  \put(20,50){\line(1,0){10}}
  \put(30,60){\line(1,0){20}}
    \put(50,70){\line(1,0){10}}
    \put(30,68){$\zeta$}
    \put(4,40){$\nu$}
    \put(8,82){\small{$k$}}
    \put(-5,62){\small{$s$}}
\put(20,50){\line(0,-1){10}}
   \put(10,40){\line(0,-1){20}}
   \put(10,40){\line(1,0){10}}
          \end{picture}
\end{center}
\end{example}
On the other hand, we can reconstruct $w_{\lambda}^{(k)}$ from a $k$-strict partition $\lambda=(\lambda_1,\dots,\lambda_{r}>0)$ in $\calP_{\infty}^{(k)}$. Let $\lambda_1,\dots, \lambda_s>k$ and $\lambda_{s+1}\leq k$. Let $\mu=(\mu_1,\dots, \mu_k)$ be the conjugate of the partition $(\lambda_{s+1},\dots,\lambda_r)$. Define $\zeta_i:=\lambda_i-k$ for $i=1,\ldots,s.$ Define $(v_1,\dots,v_k)$ by
\[
v_{k+1-i}=\mu_i+s + k+1-i - \sharp\{ \zeta_j \ |\ \zeta_j\geq \mu_i+s+ k+1-i  \}.
\]
The signed permutation $w_{\lambda}^{(k)}$ is given by
\[
w_{\lambda}^{(k)}=(v_1\dots v_k|\overline{\zeta_1},\dots\overline{\zeta_s}u_1u_2,\dots) \in W_{\infty}^{(k)}
\]
where $u_1,u_2,\dots$ form an increasing sequence of positive integers, which is determined uniquely by the integers $v_i$ and $\zeta_i$. Note that 
\begin{equation}\label{u_j}
u_{j} = j  + \sharp\{p \ |\ v_p< u_{j}\} + \sharp\{ \zeta_p \ |\ \zeta_p <  \sfu_{j} \}.
\end{equation}
Since $\zeta_i> u_j$ if and only if $\sharp\{\zeta_p \ |\ \zeta_p > u_j\}>i$, (\ref{u_j}) implies that
\begin{equation}\label{u_j 2}
\zeta_i > u_j \ \ \mbox{if and only if } \ \ \ \zeta _ i > j + \sharp\{p \ |\ \sfv_p < \sfu_{j}\} + s -i.
\end{equation}

\begin{remark} If $k>0$, the partial order on 
$\mathcal{P}^{(k)}_\infty$ 
given by the inclusion of the Young diagrams
is not compatible with the one on $W^{(k)}_\infty$
induced from the Bruhat order on $W_\infty$.
For example, if we let let $k=2$, $\lambda=(3,2),\mu=(5,1)$, then we have $w_\lambda^{(2)}=34\bar{1}2\cdots=s_2s_0s_1s_3s_2$ and $w_\mu^{(2)}=14\bar{3}2\cdots=s_2s_1s_0s_1s_3s_2$. So we have $w_\lambda^{(2)}\leq w_\mu^{(2)}$ in the Bruhat order,
but $\lambda\not\subset \mu.$

It would be an interesting problem to give
a good combinatorial model for $W^{(k)}_\infty$ 
which enable us to see the Bruhat order
manifestly. One candidate is the {\rm Maya
diagram} introduced below. 
\end{remark}
\begin{remark}
We can depict the permutation $w$ as 
the following {\rm ``Maya diagram''}.

\bigskip

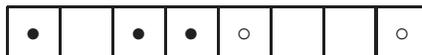
\setlength{\unitlength}{0.7mm}
\begin{center}
  \begin{picture}(70,13)
      \thicklines
  \put(0,0){\line(1,0){80}}
  \put(0,10){\line(1,0){80}}
  \multiput(0,0)(10,0){9}{\line(0,1){10}}
  \put(3.5,3.5){$\bullet$}
  \put(23.5,3.5){$\bullet$}
  \put(33.5,3.5){$\bullet$}
    \put(43.5,3.5){$\circ$}
        \put(73.5,3.5){$\circ$}
\end{picture}
\end{center}
The integers $v_1,\dots,v_k$ are the positions of the boxes with $\circ$, while $\zeta_1,\dots, \zeta_s$ are the positions of the boxes with $\bullet$.
Then $\mu_i$ is the number of 
the vacant boxes to the left of the $i$th box 
with $\circ$.
In the above diagram, we have $\zeta=(4,3,1)$ and 
$\mu=(3,1)$, so $\nu=(2,1,1).$
\end{remark}
We record the following lemmas without proofs  and will use them later in the proof of the main theorem. (\textit{cf.} \cite[Proposition 8.1.1]{BjBr}).
\begin{lemma}\label{LGLEMMA} Let $w=\sfv_1\cdots \sfv_k | \overline{\zeta_1}\cdots \overline{\zeta_s} \sfu_1\cdots \sfu_{n-k-s} \in W^{(k)}_n.$
\
Suppose $i \geq 1$.  $\ell(s_iw)<\ell(w)$ if and only if  one of the following holds:
\begin{enumerate}
\item[$(L1)$] $w=( \cdots  | \cdots \overline{i+1} \cdots i\cdots  )$, \textit{i.e.}, $\zeta_p=i+1$ and $\sfu_q=i$ for some $p$ and $q$;
\item[$(L2)$] $w=( \cdots i\cdots  | \cdots \overline{i+1} \cdots  )$, \textit{i.e.}, $\zeta_p=i+1$ and $\sfv_q=i$ for some $p$ and $q$;
\item[$(L3)$] $w=( \cdots i+1 \cdots   | \cdots  i\cdots  )$, \textit{i.e.}, $\sfu_p=i$ and $\sfv_q=i+1$ for some $p$ and $q$;
\end{enumerate}
Note that, in this case, $s_iw \in W^{(k)}_\infty$. 
\end{lemma}
\begin{lemma}\label{LGLEMMA0} Let $w=\sfv_1\cdots \sfv_k | \overline{\zeta_1}\cdots \overline{\zeta_s} \sfu_1\cdots \sfu_{n-k-s}\in W^{(k)}_n$.
\begin{itemize}
\item[$(L0)$] $\ell(s_0w)<\ell(w)$ if and only if $w=( \cdots   | \cdots \bar 1 \cdots  )$, \textit{i.e.}, $\zeta_s=1$. 
\end{itemize}
\end{lemma}

\subsection{Remarks on the Schubert conditions}

In this section, we review the definition of Schubert classes of ${SG}_n^k$ for the sake of the precise comparison of our conventions and those in \cite{DSP}. See also \cite[Section 6]{FP} and \cite[Section 0]{BKT}.  It is worth noting that the characteristic index $\chi$ appears in the Schubert conditions in an apparent manner.

Let $\ee_1,\dots, \ee_n, \ee_1^*,\dots, \ee_n^*$ be a standard symplectic basis of $\mathbb{C}^{2n}$. 
Define a symplectic form 
by \[
\langle \ee_i,\ee_j\rangle=
\langle \ee_i^*,\ee_j^*\rangle=0,\quad
\langle \ee_i,\ee_j^*\rangle=\delta_{ij}.
\]
For $1\leq i\leq n$, define a complete flag $F^{\bullet}: F^{n} \subset \cdots \subset F^1 \subset F^{\bar 1} \subset \cdots \subset F^{\bar n}$ by
\[F^i=\langle
\ee_i,\ldots,\ee_n\rangle,\quad
F^{\overline{i}}=\langle
\ee_i^*,\ldots,\ee_1^*\rangle+F^1.
\]
Let $\lambda\in \mathcal{P}_n^{(k)}.$
Then the Schubert variety $\Omega_\lambda$ with respect to $F^{\bullet}$ can be also defined as
\begin{equation}
\Omega_\lambda:=\{V\in {SG}_n^k\;|\;
\dim (V\cap F^{\overline{w_\lambda^{(k)}(k+j)}})\geq j\quad (1\leq j\leq n-k)\}.
\end{equation}
Indeed, if we relabel the above flag by $F_1 \subset \cdots \subset F_{2n}$, \textit{i.e.}, $F^i = F_{n+1 - i}$ and $F^{\bar i} = F_{n+i}$ for $1\leq j \leq n$,  then
\[
F^{\overline{w_\lambda^{(k)}(k+j)}}
=
\begin{cases}
F^{\zeta_j} = F_{n+1 - \zeta_j} =F_{n - \chi_j} &\mbox{if}\;  1\leq j \leq s,\\
F^{\overline{\sfu_{j-s}}} = F_{n+\sfu_{j-s}} = F_{n-\chi_j} &\mbox{if}\;  s+1 \leq j \leq n-k,
\end{cases}
\]
where $\sfv_1\dots\sfv_k|\overline{\zeta_1} \dots \overline{\zeta_s} \sfu_1\dots\sfu_{n-k-s}$ is the one-line notation of $w_{\lambda}^{(k)}$. Therefore the equivalence of the definitions of $\Omega_{\lambda}$ at (\ref{def:Omega_lambda1}) and here follows from 
\begin{equation}\label{eq:p=n-chi}
p_j(\lambda) = n - \chi_j \ \ (1\leq j \leq \ell(\lambda))
\end{equation}
and the fact that the condition is redundant for $j > \ell(\lambda)$. We can prove the equation (\ref{eq:p=n-chi}) as follows. 
If $1 \leq j \leq s$, then the RHS is $n-p_j(\lambda) = \zeta_j - 1=\chi_j$. Suppose that $s+1\leq j \leq \ell(\lambda)$. Then from the correspondence in \S \ref{k-strict partitions}, it is clear that
\begin{eqnarray*}
-\chi_j=\sfu_{j-s} 
&=& j - s + \sharp\{p \ |\ v_p< u_{j-s}\} + \sharp\{ \zeta_i \ |\ \zeta_i < \sfu_{j-s} \}\\
&=& k - \lambda_j + 1 + j - s - 1 + \sharp\{ \zeta_i \ |\ \zeta_i <  \sfu_{j-s} \}\\
&=& k - \lambda_j + 1 + \sharp\{ i \;|\; i<j, \ \chi_i + \chi_j \leq  -1 \}\\
&=& k - \lambda_j + 1 + j-1 - \sharp\{ i \;|\; i<j, \ \chi_i + \chi_j \geq  0 \}\\
&=& k - \lambda_j + 1 + j-1 - \sharp\{ i \;|\; i<j, \ \lambda_i + \lambda_j > 2k + j-i \}\\
&=& p_j(\lambda) - n.
\end{eqnarray*}
The first equality follows from (\ref{u_j}), the second equality follows from (\ref{lambda_i+s}) and (\ref{lambda_i+s 2}), and the second last equality follows from the following lemma. 
\begin{lemma}\label{lem:chi-lambda}
Let $\chi$ be the characteristic index of $\lambda$. Suppose $1\leq i<j \leq n-k$. Then $\chi_i + \chi_j \geq 0$ if and only if $\lambda_i + \lambda_j > 2k + j - i$.
\end{lemma}
\begin{proof} 
Let $\sfv_1\dots \sfv_k | \overline{\zeta_1}\dots \overline{\zeta_s} \sfu_1\dots \sfu_{n-k-s}$ be the one-line notation for $w_{\lambda}^{(k)}$. The only non-trivial case is when $i \leq s$ and $j \geq s+1$. The equivalence (\ref{u_j 2}) implies that $\chi_i + \chi_j = \zeta_i -1 - \sfu_{j-s} \geq 0$ if and only if
\[
\lambda_i + \lambda_j = 2k + \zeta _ i  - \sharp\{p \ |\ \sfv_p < \sfu_{j-s}\} > 2k + j - i
\]
where the first equality follows from (\ref{lambda_i+s}) and (\ref{lambda_i+s 2}).
\end{proof}
The $T$-fixed point 
of ${SG}_n^k$ corresponding to $\lambda$ is 
$\langle\ee_{w_\lambda^{(k)}(k+1)}^*,\ldots,\ee_{w_\lambda^{(k)}(n)}^*
\rangle,$ which is the image of 
$e_w\in \mathcal{F}l_n$ under the projection $\sfp$
onto ${SG}_n^k.$

\subsection{Invariant subring $\mathcal{R}_\infty^{(k)}$}\label{ssec:InvariantRing}
Let $\mathcal{R}_\infty^{(k)}$ be the subring of
elements in  
$\mathcal{R}_\infty$ which are 
fixed by the right action of $W_{(k)}$:
\[\mathcal{R}_\infty^{(k)}:=
\{f\in \mathcal{R}_\infty\;|\;
s_i^z(f)=f\quad(\forall i\neq k)
\}.
\]
Since the right action of $W_\infty$ is 
$\ZZ[t]$-linear, $\mathcal{R}_\infty^{(k)}$ is 
a $\ZZ[t]$-subalgebra of $\mathcal{R}_\infty.$
\begin{proposition}\label{prop:basisR^k} We have
$$
\mathcal{R}_\infty^{(k)}=
\bigoplus_{w\in W^{(k)}_\infty}
\ZZ[t]\mathfrak{C}_w.$$
\end{proposition}
\begin{proof} 
In order to prove the inclusion
``$\supset$'',
it is enough to show $\mathfrak{C}_{w}\in \mathcal{R}_\infty^{(k)}$
for all $w\in W^{(k)}_\infty.$
Let $w\in W^{(k)}_\infty.$ Then for any  
$j\neq k$, we have $\ell(ws_j)=\ell(w)+1,$
and hence $\partial_j\mathfrak{C}_{w}=0.$
This is equivalent to $s_j^z\mathfrak{C}_{w}=\mathfrak{C}_{w}$ for $j\neq k.$ Thus we have $\mathfrak{C}_{w}
\in \mathcal{R}_\infty^{(k)}.$
To prove the reverse inclusion ``$\subset$'', we write an arbitrary element $f$ of $\mathcal{R}_\infty^{(k)}$ as $f=\sum_{w\in W_\infty}c_w \mathfrak{C}_w\; (c_w\in \ZZ[t])$. If $v\notin W^{(k)}_\infty$, then there is $i$ such that $i\neq k$ and  $\ell(vs_i)<\ell(v)$. We have $0=\partial_if=\sum_{v\in W_\infty,\; \ell(vs_i)=\ell(v)-1}c_v \mathfrak{C}_{vs_i}.$ It follows that $c_v=0$.
\end{proof}
\begin{definition}\label{loc def}
For each $\mu \in \calP_{\infty}^{(k)}$, let $w_{\mu}^{(k)}=v_1v_2\cdots v_k|\overline{\zeta_1}\overline{\zeta_2}\cdots\overline{\zeta_s} u_1u_2\cdots \in W_{\infty}^{(k)}$ be the corresponding signed permutation. Define 
\[
\Loc:  \mathcal{R}_\infty^{(k)} \to \Map(\calP_{\infty}^{(k)}, \ZZ[t])
;  \ \ f \mapsto (\mu \mapsto f|_{\mu}),
\]
where $f|_{\mu}$ is defined by the substitution
\begin{eqnarray*}
&&(z_1,\dots, z_k) \mapsto (t_{v_1},\dots,t_{v_k});\\
&&(x_1,x_2,\dots ) \mapsto (t_{\zeta_1},\dots,t_{\zeta_s},0,0,\dots).
\end{eqnarray*}
\end{definition}
\begin{remark}
Note that this is the restriction of the universal localization map defined in \cite[Section 6.1]{DSP}. In particular, we have the following vanishing property:
\begin{equation}\label{van prop}
\frakC_{w_{\lambda}^{(k)}}|_{\mu}=0 \ \ \ \mbox{ unless $w_{\lambda}^{(k)} \leq  w_{\mu}^{(k)}$}.
\end{equation}
Here $\leq$ is the Bruhat-Chevalley order.
\end{remark}
\begin{lemma}\label{injectivity} 
The homomorphism 
$\Loc: \mathcal{R}_\infty^{(k)} \to \Map(\calP_{\infty}^{(k)}, \ZZ[t])$ is injective.
\end{lemma}
\begin{proof}
The proof is identical to the one for \cite[Lemma 6.5]{DSP}.
\end{proof}
\begin{lemma}\label{uniqueness}
If a family ${F}_w, w \in W_{\infty}^{(k)}$ of elements of $\calR_{\infty}^{(k)}$ satisfies the following conditions
\begin{eqnarray}\label{lddeq}
\delta_{i} {F}_{w}&=&\begin{cases}
 {F}_{s_{i}w}& \mbox{ if }\ \ell(s_{i}w)<\ell(w),\\
0 & \mbox{ if } \ \ell(s_{i}w)>\ell(w),
\end{cases}
\end{eqnarray}
\begin{eqnarray}\label{vanishingempty}
 {F}_{w}|_{\emptyset} &=& \delta_{w,e}\ ,
\end{eqnarray}
then $ {F}_{w}=\frakC_{w}$ for all $w \in W_{\infty}^{(k)}$.
\end{lemma}
\begin{proof}
First note that the family $\frakC_w, w \in W_{\infty}^{(k)}$ satisfies those relations: (\ref{lddeq}) by Theorem \ref{DSPTheorem}, and (\ref{vanishingempty}) by the vanishing property (\ref{van prop}) and the fact that $\frakC_{e}=1$ (Theorem 2).  By Lemma \ref{injectivity}, it suffices to show that, for each $w \in W_{\infty}^{(k)}$, the localization $\left.( {F}_w-\frakC_w)\right|_{\mu}$ is zero for all $\mu \in \calP_{\infty}^{k}$. We use the induction on $|\mu|$. By (\ref{vanishingempty}), we have $\left.( {F}_w-\frakC_w)\right|_{\emptyset}=0$ for every $w\in W_{\infty}^{(k)}$. Now assume that $\mu\not=\emptyset$ and that for each $w \in W_{\infty}^{(k)}$, the localization $\left.( {F}_w-\frakC_w)\right|_{\mu'}$ is zero for all $\mu' \in \calP_{\infty}^{(k)}$ such that $|\mu'|<|\mu|$. Since $\mu\not=\emptyset$, there is $i\geq 0$ such that 
$s_iw_\mu^{(k)}<w_\mu^{(k)}$. This implies that $s_iw_{\mu}^{(k)}$ is a minimum-length coset representative, \textit{i.e.} an element of $W_{\infty}^{(k)}$. Therefore $s_iw_\mu^{(k)}=w_{\nu}^{(k)}$ for some $\nu\in \calP_\infty^{(k)}$ with $|\nu|=|\mu|-1$. By the definition of $\delta_i$ and the localization at $\mu$, the equations (\ref{lddeq}) implies the following recurrence relation
\[
 {F}_w|_\mu=
\begin{cases}
s_i( {F}_w|_{\nu})+\alpha_i\cdot s_i( {F}_{s_iw}|_{\nu}) & \mbox{ if }\ \ell(s_iw)<\ell(w) ,\\
s_i( {F}_w|_{\nu})& \mbox{ if }\  \ell(s_iw)> \ell(w).
\end{cases}
\]
Since $\frakC_w$ also satisfies the same recurrence relation, the difference $( {F}_w - \frakC_w)|_\mu$ vanishes by the induction hypothesis. Thus $ {F}_w=\frakC_w$ for every $w\in W_{\infty}^{(k)}$.
\end{proof}
\subsection{Duality of ${SG}_n^k$}\label{sec:dual}
There is a unique longest element in $W_n^{(k)},$ which we denote by $w_{max}.$ In the one-line notation, it is given by $12\cdots k | \overline{n} \cdots \overline{n-k}.$ For $w\in W_n^{(k)}$, define $w^\vee=ww_{max}.$
We have $w^\vee\in W_n^{(k)}$ and $\ell(w^\vee)=\ell(w_{max})-\ell(w).$ Moreover, we have $w_{max}^2=e,$ and so the operation $w\mapsto w^\vee$ is an involution on the set $W_n^{(k)}.$ Note that this involution {\em does\/} depend on $n.$

\begin{remark} 
If $w=v_1\cdots v_k|\overline{\zeta_1}\cdots \overline{\zeta_s}u_1\cdots u_{n-k-s}$, then 
\[
w^\vee=v_1\cdots v_k|\overline{u_{n-k-s}}\cdots
\overline{u_1}\zeta_s\cdots \zeta_{1}.
\]
In other words,  the involution in terms of Maya diagram is  given by exchanging the vacant boxes and the boxes occupied by ``\;$\bullet$''.
\end{remark}

Let $w,v\in W_n^{(k)}$ and  $i\in \calI:=\{0,1,\ldots,n-1\}.$  We write $w\overset{i}{\rightarrow} v$ if $s_iw=v$ and $\ell(v)=\ell(w)+1.$  The relation is called the covering relation for the weak left Bruhat order (\cite{BjBr}). The {\em weak Bruhat graph\/} is the graph such that the set of vertices as $W_n^{(k)}$ and the (oriented) arrows are the covering relation for the weak left Bruhat order.
\begin{example}Let $n=4,\; k=2.$
We can draw the weak Bruhat graph as follows.
The involution is given by reflection 
with respect to the dashed horizontal line.
\setlength{\unitlength}{0.3mm}
\begin{center}
  \begin{picture}(400,350)
    \thicklines
\multiput(0,170)(2,0){180}{\line(1,0){1}}
\put(50,330){\Maya  
\put(2,2){$\circ$}\put(12,2){$\circ$}
}
\put(50,300){\Maya
\put(2,2){$\circ$}\put(22,2){$\circ$}}
\put(10,270){\Maya
\put(2,2){$\circ$}\put(32,2){$\circ$}}
\put(90,270){\Maya
\put(12,2){$\circ$}\put(22,2){$\circ$}}
\put(50,240){\Maya
\put(12,2){$\circ$}\put(32,2){$\circ$}
}
\put(160,240){\Maya
\put(12,2){$\circ$}\put(22,2){$\circ$}\put(2,2){$\bullet$}} 
\put(50,210){\Maya
\put(22,2){$\circ$}\put(32,2){$\circ$}}
\put(120,210){\Maya
\put(2,2){$\bullet$}\put(12,2){$\circ$}\put(32,2){$\circ$}} 
\put(230,210){\Maya
\put(2,2){$\circ$}\put(12,2){$\bullet$}\put(22,2){$\circ$}} 
\put(80,180){\Maya
\put(2,2){$\bullet$}\put(22,2){$\circ$}\put(32,2){$\circ$}}
\put(180,180){\Maya
\put(2,2){$\circ$}\put(12,2){$\bullet$}\put(32,2){$\circ$}} 
\put(290,180){\Maya
\put(2,2){$\circ$}\put(12,2){$\circ$}\put(22,2){$\bullet$}} 
\put(80,150){\Maya
\put(12,2){$\bullet$}\put(22,2){$\circ$}\put(32,2){$\circ$}}
\put(180,150){\Maya
\put(2,2){$\circ$}\put(22,2){$\bullet$}\put(32,2){$\circ$}} 
\put(290,150){\Maya
\put(2,2){$\circ$}\put(12,2){$\circ$}\put(32,2){$\bullet$}
} 
\put(50,120){\Maya
\put(2,2){$\bullet$}\put(12,2){$\bullet$}
\put(22,2){$\circ$}\put(32,2){$\circ$}}
\put(120,120){\Maya
\put(12,2){$\circ$}
\put(22,2){$\bullet$}\put(32,2){$\circ$}} 
\put(230,120){\Maya
\put(2,2){$\circ$}\put(22,2){$\circ$}
\put(32,2){$\bullet$}}
\put(50,90){\Maya
\put(2,2){$\bullet$}\put(12,2){$\circ$}
\put(22,2){$\bullet$}\put(32,2){$\circ$}}
\put(160,90){\Maya
\put(12,2){$\circ$}
\put(22,2){$\circ$}\put(32,2){$\bullet$}} 
\put(10,60){\Maya
\put(2,2){$\circ$}\put(12,2){$\bullet$}
\put(22,2){$\bullet$}\put(32,2){$\circ$}}
\put(90,60){\Maya
\put(2,2){$\bullet$}\put(12,2){$\circ$}
\put(22,2){$\circ$}\put(32,2){$\bullet$}}
\put(50,30){\Maya
\put(2,2){$\circ$}\put(12,2){$\bullet$}
\put(22,2){$\circ$}\put(32,2){$\bullet$}}
\put(50,0){\Maya
\put(2,2){$\circ$}\put(12,2){$\circ$}
\put(22,2){$\bullet$}\put(32,2){$\bullet$}}
\put(65,315){$\downarrow$}
\put(75,315){$\tiny{2}$}
\put(45,285){$\swarrow$}
\put(80,285){$\searrow$}
\put(37,285){$\tiny{3}$}
\put(93,285){$\tiny{1}$}
\put(45,255){$\searrow$}
\put(37,255){$\tiny{1}$}
\put(80,255){$\swarrow$}
\put(93,255){$\tiny{3}$}
\put(140,255){$\searrow$}
\put(131,255){$\tiny{0}$}
\put(65,225){$\downarrow$}
\put(75,225){$\tiny{2}$}
\put(105,225){$\searrow$}
\put(95,225){$\tiny{0}$}
\put(145,225){$\swarrow$}
\put(140,225){$\tiny{3}$}
\put(215,225){$\searrow$}
\put(210,225){$\tiny{1}$}
\put(275,195){$\searrow$}
\put(270,195){$\tiny{2}$}
\put(85,195){$\searrow$}
\put(80,195){$\tiny{0}$}
\put(115,195){$\swarrow$}
\put(110,195){$\tiny{2}$}
\put(165,195){$\searrow$}
\put(160,195){$\tiny{1}$}
\put(215,195){$\swarrow$}
\put(210,195){$\tiny{3}$}
\put(100,165){$\downarrow$}
\put(95,165){$\tiny{1}$}
\put(195,165){$\downarrow$}
\put(190,165){$\tiny{2}$}
\put(305,165){$\downarrow$}
\put(300,165){$\tiny{3}$}
\put(115,135){$\searrow$}
\put(110,135){$\tiny{2}$}
\put(80,135){$\tiny{0}$}
\put(85,135){$\swarrow$}
\put(165,135){$\swarrow$}
\put(160,135){$\tiny{1}$}
\put(215,135){$\searrow$}
\put(210,135){$\tiny{3}$}
\put(275,135){$\swarrow$}
\put(270,135){$\tiny{2}$}%
\put(65,105){$\downarrow$}
\put(75,105){$\tiny{2}$}
\put(95,105){$\swarrow$}
\put(110,105){$\tiny{0}$}
\put(145,105){$\searrow$}
\put(137,105){$\tiny{3}$}
\put(215,105){$\swarrow$}
\put(210,105){$\tiny{1}$}
\put(45,75){$\swarrow$}
\put(80,75){$\searrow$}
\put(140,75){$\swarrow$}
\put(37,75){$\tiny{1}$}
\put(93,75){$\tiny{3}$}
\put(155,75){$\tiny{0}$}
\put(45,45){$\searrow$}
\put(37,42){$\tiny{3}$}
\put(80,45){$\swarrow$}
\put(93,42){$\tiny{1}$}
\put(65,15){$\downarrow$}
\put(75,15){$\tiny{2}$}
      \end{picture}
\end{center}
\end{example}

Let $w\in W_n^{(k)}.$
Define the following sets:
\begin{equation}
\begin{split}
\calI_{-}(w)&:=\{i\in \calI\;|\;\ell(s_iw)=\ell(w)-1\},\\
\calI_{+}(w)&:=\{i\in \calI\;|\;\ell(s_iw)=\ell(w)+1\;\mbox{and} \;
s_iw\in W_n^{(k)}\},\\
\calI_{0}(w)&:=\{i\in \calI\;|\;\ell(s_iw)=\ell(w)+1\;
\mbox{and}\;s_iw\not\in W_n^{(k)}\}.
\end{split}
\end{equation}
Note that if $i\in \calI_{-}(w)$ then $s_iw\in W^{(k)}_n.$

\begin{example}Let $n=4,k=2$. If 
$w=23|\overline{4}1.$ Then $\calI_{-}(w)=\{1,3\},\;
\calI_{+}(w)=\{0\},$ and $\calI_0(w)=\{2\}.$
\end{example}

\begin{lemma}\label{lem:Poincare} 
Let $w\in W_n^{(k)}.$ Then the following hold.

(1) $\calI_-(w)=\calI_{+}(w^\vee),\;\calI_+(w)=\calI_{-}(w^\vee),$

(2) $\calI_0(w)=\calI_0(w^\vee).$
\end{lemma}
\begin{proof}
(1) We will show $\calI_{-}(w) \subset \calI_{+}(w^{\vee}).$ Let $i\in \calI_{-}(w).$  Then $s_iw\in W_n^{(k)}$ as noted above. Since  $(s_iw)^\vee=s_iww_{max}=s_iw^\vee\in W_n^{(k)}$, we have 
\[
\ell(s_iw^\vee)=\ell((s_iw)^\vee)
=\ell(w_{max})-\ell(s_iw)=\ell(w_{max})-(\ell(w)-1)
=\ell(w^\vee)+1.
\] 
Thus $i\in \calI_{+}(w^{\vee}).$ The proof of the opposite inclusion is similar. The second statement follows from the fact $(w^\vee)^\vee=w.$

(2) $\calI_0(w)$ is the complement of $\calI_{-}(w)\cup \calI_{+}(w)$ in $\calI.$ Hence the result follows from (1).
\end{proof}
\begin{lemma}\label{lem:I0} 
Let $w\in W_n^{(k)}$. If $i\in \calI_0(w)$, then $s_iw=ws_j$ for some $j\;(\neq k).$
\end{lemma}
\begin{proof} Because $s_iw\notin W_n^{(k)}$,
there exists $j(\neq k)$ such that 
$\ell(s_iw s_j)=\ell(s_iw)-1.$
This means that $s_i w$ has a reduced
expression of a form 
$s_{i_1}\cdots s_{i_l}s_j$ with $l=\ell(w)$ (\textit{cf.} \cite[Cor. 1.4.6.]{BjBr})
Since $\ell(s_iw)=\ell(w)+1$, $w$ is obtained from the
reduced expression of $s_iw$ by deleting one {\em unique\/}
simple reflection (``exchange condition'' \textit{cf.} \cite[p.117]{Hum}).
Now the right end $s_j$ is the unique one to be deleted since, otherwise, it contradicts with the assumption $w \in W_n^{(k)}$. Thus $w=s_{i_1}\cdots s_{i_l},$
and the lemma follows.
\end{proof}
\subsection{Lemma on left divided difference operators}
For $w\in W_\infty,$ we choose  a reduced expression $s_{i_1}\cdots s_{i_l}$ for $w$. Then $\delta_w:=\delta_{i_1}\cdots\delta_{i_l}$ does not depend on the reduced expression. The following fact is well-known (see for example \cite[\S 2]{Man}).
\begin{lemma}\label{lem:delta_w} 
Let $u,v\in W.$ Then 
\[\delta_u\delta_v=
\begin{cases}
\delta_{uv} &\mbox{if}\;\ell(u)+\ell(v)=
\ell(uv),\\
0 &\mbox{if}\;\ell(u)+\ell(v)>
\ell(uv).
\end{cases}\]
\end{lemma}
The following proposition will be used in the proof of the main theorem.
\begin{proposition}\label{prop:delta_w2}
Let $w\in W_n^{(k)}.$ We have the following.

(1) If $i\in \calI_{-}(w)$ then $\delta_i \delta_{w^\vee}=\delta_{(s_iw)^\vee}.$ 

(2) If $i\in \calI_{+}(w)$ then $\delta_i \delta_{w^\vee}=0.$ 

(3) If $i\in \calI_{0}(w)$ then there exists $j\not=k$
such that $\delta_i \delta_{w^\vee}
=\delta_{w^\vee}\delta_j.$
\end{proposition}
\begin{proof}
(1) If $i\in \calI_{-}(w)$, then from Lemma \ref{lem:Poincare},
we have $i\in \calI_{+}(w^\vee)$, \textit{i.e.},
$s_iw^\vee\in W_n^{(k)}$ and 
$\ell(s_iw^\vee)=\ell(w^\vee)+1.$  
Recall that  
$s_iw^\vee=(s_iw)^\vee.$
Then the result follows from Lemma \ref{lem:delta_w}.

(2) If $i\in \calI_{+}(w)$, then from Lemma \ref{lem:Poincare},
we have $i\in \calI_{-}(w^\vee).$ This means that 
$\ell(s_iw^\vee)=\ell(w^\vee)-1.$ Hence 
$\delta_{s_iw^\vee}=0$ by Lemma \ref{lem:delta_w}.

(3) If $i\in \calI_{0}(w)$, then from Lemma \ref{lem:Poincare},
we have $i\in \calI_{0}(w^\vee).$ 
Then from Lemma \ref{lem:I0} there exists some $j\neq k$ such that 
$s_iw^\vee 
=w^\vee s_j,$ where the products in both hand sides 
are length-additive.   
Then the result follows from Lemma \ref{lem:delta_w}.
\end{proof}
\section{Multi Schur-Pfaffian}\label{sec:Pf}
In this section, we recall the multi Schur-Pfaffian due to Kazarian, but in a slightly more general form.

Let $(c^{(1)}, c^{(2)},\dots, c^{(m)}) $ be 
an $m$-tuple such that each $c^{(i)}$ is an infinite sequence of variables $c_r^{(i)}(r\in \ZZ)$. For any $(r_1,\ldots,r_m)\in \ZZ^m$,
the multi Schur-Pfaffian
\[
\Pf[c_{r_1}^{(1)}\ldots c_{r_m}^{(m)}] \in \ZZ[ c_{r}^{(1)},c_{r}^{(2)},\dots,c_r^{(m)}   (r \in \ZZ)]
\]
is defined as follows:
\begin{itemize}
\item for $m=1$, we set 
$
\Pf[c_r ^{(1)}]=c_r^{(1)}.
$
\item for $m=2,$ we set 
$
\Pf[c_{r_1}^{(1)}c_{r_2}^{(2)}]=
c_{r_1}^{(1)}c_{r_2}^{(2)}+2\sum_{s=1}^{r_2}(-1)^s
c_{r_1+s}^{(1)}c_{r_2-s}^{(2)}.
$
\item for any even $m\geq 4$, we set 
$$
\Pf[c_{r_1}^{(1)}\ldots c_{r_m}^{(m)}]
=
\sum_{s=2}^m(-1)^{s}
\Pf[c_{r_1}^{(1)}c_{r_s}^{(s)}]
\cdot
 \Pf[c_{r_2}^{(2)}\ldots\widehat{c_{r_s}^{(s)}} 
 \cdots c_{r_m}^{(m)}].
$$
\item for any odd $m\geq 3$, we set 
$$
\Pf[c_{r_1}^{(1)}\cdots c_{r_m}^{(m)}]
=
\sum_{s=1}^m(-1)^{s-1}
c_{r_s}^{(s)}\cdot
\Pf[c_{r_1}^{(1)}\cdots \widehat{c_{r_s}^{(s)}}
\cdots c_{r_m}^{(m)}].
$$
\end{itemize}

\begin{remark}[{Pfaffian in classical literature}]
If we assume $r_i \geq 0$, $c_0^{(i)}=1$ and $\Pf[c_{r_i}^{(i)}c_{r_j}^{(j)}] +\Pf[c_{r_j}^{(j)}c_{r_i}^{(i)}]=0$ hold for all $1\leq i,j \leq m$, then $\Pf[c_{r_1}^{(1)}\ldots c_{r_m}^{(m)}]$ coincides with Kazarian's Pfaffian. In particular, in the case when $m$ is even, then $\Pf[c_{r_1}^{(1)}\cdots c_{r_m}^{(m)}]$ is the classical Pfaffian of the skew symmetric matrix whose $(i,j)$ entry is given by  $\Pf[c_{r_i}^{(i)}c_{r_j}^{(j)}].$  If we further assume that  $c_{r}^{(i)}= c_{r}^{(j)}$ for all $i,j$, then it is due to Schur.
\end{remark}
\begin{remark} \label{rem: rel to Ik and Ka}
Let $\lambda$ be a strict partition of length $\ell(\lambda)$, then 
\begin{equation}\label{eq:Pf_facQ}
\Pf[{}_0\vt_{\lambda_1}^{(\lambda_1-1)}\cdots {}_0\vt_{\lambda_{\ell(\lambda)}}^{(\lambda_{\ell(\lambda)}-1)}] :=  \left.\Pf[c_{\lambda_1}^{(1)}\cdots c_{\lambda_{\ell(\lambda)}}^{(\ell(\lambda))}]\right|_{c_m^{(i)}={}_0\vt_m^{(\lambda_i-1)}}
\end{equation}
is equal to the factorial $Q$-function $Q_\lambda(x|t)$ defined by Ivanov \cite{Iv}. This expression is obtained in \cite[\S 11]{DSP}, which is also equivalent to  Kazarian's Lagrangian degeneracy loci formula \cite{Ka}. Note that (\ref{eq:Pf_facQ}) is a variant of Ivanov's original Pfaffian formula \cite[Thm 9.1]{Iv}. In particular, $\Pf[Q_{\lambda_1}\cdots Q_{\lambda_{\ell(\lambda)}}]$ is the classical Schur $Q$-function \cite{Schur} where ${}_0\vt^{(0)}_r$ is denoted by $Q_r$ in (\ref{def of Q_r}).
\end{remark}
\begin{remark}
The multi Schur-Pfaffian can be defined in terms of the raising operators (\textit{cf.} \cite{BKT}). This aspect will be postponed until Section \ref{sec: IM-TW}, since we will not use it in the proof of our main theorem.
\end{remark}
Kazarian stated the following properties of $\Pf$ in \cite[\S 1]{Ka}. They follow from the above definition of Pfaffian by the induction on $m$.
\begin{proposition}\label{SQinPf}\
\begin{itemize}
\item[(1)] If $\Pf[c_{r}^{(l)}c_r^{(l)}]=0$, then we have $\Pf[c_{r_1}^{(l_1)}\cdots c_r^{(l)}c_r^{(l)}\cdots c_{r_m}^{(l_m)}]=0$.
\item[(2)] If $\Pf[c^{(l)}_{r}c^{(l)}_{s}]+\Pf[c^{(l)}_{s} c^{(l)}_{r}]=0
$, then we have
\[
\Pf[c_{r_1}^{(l_1)}\ldots c_{r}^{(l)}c_{s}^{(l)} \cdots c_{r_m}^{(l_m)}]+\Pf[c_{r_1}^{(l_1)}\ldots 
c_{s}^{(l)}c_{r}^{(l)} \cdots c_{r_m}^{(l_m)}]=0.
\] 
\end{itemize}
\end{proposition}
\section{Double theta polynomials}\label{sec:theta}
In this section, first we list basic formulas for the double theta polynomials. In particular,  Proposition \ref{delta-i-on-Pf} is essential for computing the double Schubert polynomials via the divided difference operators in Section \ref{sec:proof}. In Section \ref{sec:Chern}, we give the geometric interpretation of those polynomials in terms of the Chern classes of vector bundles, although we will not use these facts in the proof of the main theorem.
\subsection{${}_k\vt_r^{(l)}(x,z\, |\, t)$}
Recall Definition \ref{def-theta} of the double theta polynomial ${}_k\vt_r^{(l)}(x,z\, |\, t)$. We denote the generating function by 
\begin{equation}\label{gen fcn f}
{}_kf_l(u)=\sum_{r=0}^\infty{}_k \vt_r^{(l)}(x,z\, |\, t)\cdot u^r.
\end{equation}
\begin{proposition}\label{prop:theta_is_invariant} We have 
${}_k\vt_r^{(l)}(x,z\, |\, t)\in \mathcal{R}_\infty^{(k)}.$
\end{proposition}
\begin{proof}
We check the invariance of ${}_k\vt_r^{(l)}(x,z\, |\, t)$ under the action of $s_i^z\;(i\geq 0,\;i\neq k)$. Since ${}_k\vt_r^{(l)}$  is a symmetric polynomial in $z_1,\ldots,z_k$, it is obvious when $i \geq 1, i\not=k$. To show $s_0^z ({}_k\vt_r^{(l)}(x,z\, |\, t))={}_k\vt_r^{(l)}(x,z\, |\, t)$, it suffices to consider the case $l=0$, since  $s_0^z$ is $\ZZ[t]$-linear. The action of $s_0^z$ is given by substitutions $(x_1,x_2,\ldots) \mapsto (z_1,x_1,x_2,\ldots)$ and $z_1\mapsto-z_1.$
Thus $s_0^z({}_kf_0(u))$ is 
\[
\frac{1+z_1u}{1-z_1u} \prod_{i=1}^\infty \frac{1+x_iu}{1-x_iu} \cdot  (1-z_1u) \prod_{i=2}^k(1+z_iu).
\]
Clearly this is equal to ${}_kf_0(u).$
\end{proof}
We fix $k\geq 0$ throughout this section and denote $ {}_k\vt_{r}^{(l)}(x,z\, |\, t)$ by $\vt_r^{(l)}$ and  ${}_kf_l(u)$ by $f_l(u)$.
\begin{lemma}\label{square-relation0}
Suppose $l \geq 0$. Then we have
\[
 \vt_r^{(l)}\cdot  \vt_r^{(l)} +    2\sum_{j=1}^r (-1)^j   \vt_{r+j}^{(l)}  \cdot  \vt_{r-j}^{(l)} =
\begin{cases}
\sum_{s= 0}^{r} e_{s}(z_1^2,\ldots,z_k^2)e_{r-s}(t_1^2,\ldots,t_l^2) & \mbox{ if  $r \leq k+l,$}\\
0 & \mbox{ if  $r > k+l$}.
\end{cases}
\]
\end{lemma}
\begin{proof}
The claims follow from
\begin{eqnarray*}
&& \prod_{j=1}^k (1 - z_j^2 u^2 ) \prod_{s=1}^{l} (1 - t_s^2 u^2 )  \ \ \ \mbox{ an even degree polynomial in $u$ of degree $\leq 2(k+l)$}\\
&=&   f_l(u)\cdot   f_l(-u)\\
&=&\sum_{r\geq 0} \left(\sum_{s=0}^{2r+1}  (-1)^s  \vt_{2r+1-s}^{(l)} \cdot  \vt_{s}^{(l)} \right)u^{2r+1}     +  \sum_{r\geq 0}   (-1)^r \left(  \vt_r^{(l)} \cdot \vt_r^{(l)} +    2\sum_{j=1}^r (-1)^j   \vt_{r+j}^{(l)} \cdot \vt_{r-j}^{(l)}  \right)u^{2r}.\\
\end{eqnarray*}
\end{proof}
\begin{lemma}\label{theta-relation}For all $l > 1$, we have
\[
 \vt_r^{(l)} =  \vt_r^{(l-1)} - t_l  \cdot  \vt_{r-1}^{(l-1)}.
\]
For $l\geq 0$, we have 
\begin{equation}
  \vt_r^{(-l)} =    \vt_r^{(-l-1)}  + t_{l+1} \cdot \vt_{r-1}^{(-l-1)} 
  =\sum_{i=0}^{r} (-t_l)^i  \vt_{r-i}^{(-l+1)}.
  \label{eq:vt+ttv}
\end{equation}
\end{lemma}
\begin{proof}
The first equation is obtained by extracting 
the coefficient of $u^r$ in the equation 
\[
 f_l(u)= f_{l-1}(u)\cdot (1-t_l u)
\] 
which is obvious from the definition of $ \vt_r^{(l)}.$ The second identities are the consequence of the following  equations 
\[
 f_{-l}(u)= f_{-l-1}(u)\cdot (1+t_{l+1} u) = f_{-l+1}(u)(1+t_{l+1}u)^{-1}.
\]
\end{proof}
\begin{lemma}\label{action-of-s_i}
We have 
\begin{eqnarray}
s_i^t( \vt_r^{(l)})&=& \vt_r^{(l)} \ \ \ \ \ \ (l \not= \pm i),\label{eq:inv_vt}\\
s_i^t( \vt_r^{(i)})&=&  \vt_r^{(i-1)} - t_{i+1}\cdot \vt^{(i-1)}_{r-1}   \ \ \ \ (i\geq 0),\label{eq:si_vti} \\
s_i^t( \vt_r^{(-i)})&=&  \vt_r^{(-i-1)} + t_i \cdot \vt_{r-1}^{(-i-1)}  \ \ \ \ \ (i>0).\label{eq:si_vt-i}
\end{eqnarray}
\end{lemma}
\begin{proof}
Since $ \vt^{(l)}_r$ is a polynomial symmetric in $t_1,\ldots,t_{|l|}$, the identity (\ref{eq:inv_vt}) for $i\geq 1$ is obvious. The case when $i=0$, \textit{i.e.}, the invariance of $ \vt_r^{(l)}\;(l\neq 0)$ under $s_0^t$, can be shown in the same manner as in the proof of Proposition \ref{prop:theta_is_invariant}.

For $i\geq 1$, we have 
\[
s_i^t( f_{i}(u))
=  f_0(u) (1-t_1u)\cdots (1-t_{i-1}u) (1-t_{i+1}u)
= f_{i-1}(u)(1-t_{i+1}u).
\]
Thus  the equation (\ref{eq:si_vti}) for $i\geq 1$ is obtained by comparing the coefficients of $u^r.$  The case when $i=0$ is derived from the following equation
\[
s_0^t( f_0(u))
=\frac{1-t_1u}{1+t_1u} \cdot  f_0(u)
=(1-t_1u)\cdot  f_{-1}(u).
\] 
The equation (\ref{eq:si_vt-i}) follows from
\begin{eqnarray*}
s_i^t( f_{-i}(u))
&=& f_0(u)(1+t_1u)^{-1}\cdots(1+t_{i-1}u)^{-1} (1+t_{i+1}u)^{-1}\\
&=& f_0(u)(1+t_1u)^{-1}\cdots(1+t_{i-1}u)^{-1} (1+t_{i}u)^{-1} (1+t_{i+1}u)^{-1}(1+t_iu)\\
&=& f_{-i-1}(u)(1+t_iu).
\end{eqnarray*}
\end{proof}
\begin{lemma}\label{action-of-delta}
For all $i\geq 0$, we have
\begin{eqnarray*}
\delta_i \vt_r^{(l)} = \begin{cases}
0 & \mbox{if}\; l\not=\pm i,\\
 \vt_{r-1}^{(l-1)} & \mbox{if}\; l=\pm i.
\end{cases}
\end{eqnarray*}
\end{lemma}
\begin{proof} 
The case when $l\neq \pm i$ is obvious
from the invariance result (\ref{eq:inv_vt}). Let $i\geq 1.$ The cases when $l=\pm i$ follow from the following equations:
\begin{eqnarray*}
 f_i(u)-s_i^t( f_i(u))
&=& f_{i-1}(u) \left((1-t_{i}u)-(1-t_{i+1}u)\right)\\
&=& f_{i-1}(u)(t_{i+1}-t_i)u,\\
 f_{-i}(u)-s_i^t ( f_{-i}(u))
&=& f_{-i+1}(u)\left((1+t_{i}u)^{-1}-(1+t_{i+1}u)^{-1}\right)\\
&=& f_{-i+1}(u)(1+t_{i}u)^{-1} (1+t_{i+1}u)^{-1}(t_{i+1}-t_i)\\
&=& f_{-i-1}(u)(t_{i+1}-t_i)u.
\end{eqnarray*}
Finally we show $\delta_0 ( \vt_r^{(0)})= \vt_{r-1}^{(-1)}$. This is a consequence of the following equation:
\begin{eqnarray*}
 f_0(u)-s_0^t( f_0(u))= f_0(u)\left(1-\frac{1-t_1u}{1+t_1u}\right)= f_0(u)\frac{2t_1u}{1+t_1u}= f_{-1}(u){2t_1u}.
\end{eqnarray*}
\end{proof}
\begin{lemma}\label{delta > sum} For $i>0$, we have
\[
\delta_i( \vt_{r}^{(i)}\cdot \vt_s^{(-i)})  
=  \vt_{r-1}^{(i-1)}\cdot \vt_s^{(-i-1)} +  \vt_r^{(i-1)}\cdot \vt_{s-1}^{(-i-1)}.
\]
\end{lemma}
\begin{proof}
By the Leipnitz rule, Lemma \ref{action-of-delta}, 
Equation (\ref{eq:si_vti}), and Equation (\ref{eq:vt+ttv}), we have 
\begin{eqnarray*}
\delta_i( \vt_{r}^{(i)}\cdot \vt_s^{(-i)}) 
&=& \delta_i( \vt_{r}^{(i)}) \vt_s^{(-i)}
+s_i( \vt_{r}^{(i)})\delta_i( \vt_{s}^{(-i)})
\\
&=&  \vt_{r-1}^{(i-1)}\cdot \vt_s^{(-i)} 
+ (  \vt_r^{(i-1)} - t_{i+1}\cdot \vt^{(i-1)}_{r-1} ) \vt_{s-1}^{(-i-1)} \\
&=&  \vt_{r-1}^{(i-1)}( \vt_s^{(-i)}- t_{i+1}\cdot \vt_{s-1}^{(-i-1)}) +  \vt_r^{(i-1)}\cdot \vt_{s-1}^{(-i-1)} \\
&=&  \vt_{r-1}^{(i-1)}\cdot \vt_s^{(-i-1)} +  \vt_r^{(i-1)}\cdot \vt_{s-1}^{(-i-1)}.
\end{eqnarray*}
\end{proof}
We use the following notation in the rest of the paper.
\begin{definition}
For all $(r_1,\dots, r_m)$, $(l_1,\dots, l_m) \in \ZZ^m$, let
\[
\Pf\left[\vt_{r_1}^{(l_1)} \vt_{r_2}^{(l_2)}\cdots  \vt_{r_m}^{(l_m)}\right] 
:= \left.\Pf\left[c_{r_1}^{(1)}c_{r_2}^{(2)}\cdots c_{r_m}^{(m)}\right]\right|_{c= \!\vt^{(l)}},
\]
where $|_{c=\vt^{(l)}}$ means that we substitute $\vt_{s}^{(l_i)}$ to $c_{s}^{(i)}$ for all $i\in \{1,\dots, m\}$ and $s \in \ZZ$. 
\end{definition}

We emphasize that we substitute theta polynomials after we write the Pfaffian as polynomials in the formal variable $c_{s}^{(i)}$'s. For example, 
\[
\Pf[ \vt_{-1}^{(l_1)} \vt_{1}^{(l_2)}] = \left.   \Pf[c_{-1}^{(1)}c_{1}^{(2)}]    \right|_{c= \vt^{(l)}} =  \left. (c_{-1}^{(1)}c_{1}^{(2)}-2c_{0}^{(1)}c_{0}^{(2)})    \right|_{c= \vt^{(l)}} = -2 \vt_{0}^{(l_1)} \vt_{0}^{(l_2)} =-2.
\]
\begin{remark}\label{rem: n-ind theta}
The following are clear from the definition of Pfaffian and the facts that $\vt_r^{(l)}=0$ for all $r<0$ and $\vt_0^{(l)}=1$.
\begin{eqnarray*}
&&\Pf[\vt_{r_1}^{(l_1)} \cdots \vt_{r_m}^{(l_m)} \vt_{0}^{(l_{m+1})}] = \Pf[\vt_{r_1}^{(1)} \cdots \vt_{r_m}^{(m)}], \\ 
&&\Pf[\vt_{r_1}^{(l_1)} \cdots \vt_{r_m}^{(l_m)} \vt_{r_{m+1}}^{(l_{m+1})}] = 0 \ \ \ \mbox{ if } \ r_{m+1} < 0.
\end{eqnarray*}
\end{remark}
We use the following two propositions from Lemma \ref{square-relation0}, \ref{action-of-delta}, and \ref{delta > sum} and use them in the proofs of our main results.
\begin{proposition}\label{prop square-relation0}
Suppose $l \geq 0$ and $r > k+l$. Then
\[
\Pf[ \vt_{r_1}^{(l_1)}  \cdots \vt_{r}^{(l)}\vt_{r}^{(l)} \cdots  \vt_{r_m}^{(l_m)}] = 0.
\]
\end{proposition}
\begin{proof}
Lemma \ref{square-relation0} states that $\Pf[\vt_{r}^{(l)}\vt_{r}^{(l)}]=0$ if $l \geq 0$ and $r > k+l$. Thus the claim follows from Proposition \ref{SQinPf} (1).
\end{proof}
\begin{proposition}\label{delta-i-on-Pf}\
\begin{itemize}
\item[(a)] Let $i\geq 0$. If $l_p\not=\pm i$ for all $p \in \{1,\dots,m\}$, then  $\delta_i\Pf[\vt_{r_1}^{(l_1)} \cdots \vt_{r_m}^{(l_m)}]=0$.
\item[(b)] Let $i\geq  0$. Suppose that $l_p \in \{\pm i\}$ for some $p \in \{1,\dots,m\}$ and that $l_q\not\in \{\pm i\}$ for all $q \not=p$. 
\begin{eqnarray*}
&&\delta_i\Pf[ \vt_{r_1}^{(l_1)} \cdots \vt_{r_{\!p}}^{(l_p)}\cdots  \vt_{r_m}^{(l_m)}] = \Pf[ \vt_{r_1}^{(l_1)} \cdots  \vt_{r_{\!p}-1}^{(l_p-1 )}\cdots  \vt_{r_m}^{(l_m)}].
\end{eqnarray*}
\item[(c)] Let $i>0$. Suppose that $l_p=i$ and $l_q=-i$ for some $p<q$ and that $l_s\not=\pm i$ for all $s\not\in\{p,q\}$. Then we have
\begin{eqnarray*}
&&\delta_i\Pf[ \vt_{r_1}^{(l_1)} \cdots  \vt_{r_p}^{(i)}\cdots  \vt_{r_{\!q}}^{(-i)}\cdots  \vt_{r_m}^{(l_m)}]\\
&=&\Pf[ \vt_{r_1}^{(l_1)} \cdots  \vt_{r_{\!p}-1}^{(i-1)}\cdots  \vt_{r_{\!q}}^{(-i-1)}\cdots  \vt_{r_{m}}^{(l_{m})}]
+\Pf[ \vt_{r_1}^{(l_1)} \cdots  \vt_{r_{\!p}}^{(i-1)}\cdots  \vt_{r_{\!q}-1}^{(-i-1)}\cdots  \vt_{r_{m}}^{(l_{m})}].
\end{eqnarray*}
\end{itemize}
\end{proposition}
\begin{proof}
We can prove (a) and (b) from Lemma \ref{action-of-delta} by induction on $m$ in the axioms of Pfaffian. We show how to prove (c). If $m=2$, the claim follows by applying Lemma \ref{delta > sum} to each monomial in the definition of Pfaffian. 
Suppose that $m$ is odd $>2$, and that the claim holds for all $m'<m$. By the definition of Pfaffian, Leibnitz rule, and Lemma \ref{action-of-s_i}, together with (a) and (b) above, we can compute
\begin{eqnarray*}
&&\delta_i\Pf[ \vt_{r_1}^{(l_1)} \cdots  \vt_{r_m}^{(l_m)}]\\
&=& \sum_{s=1}^{m} (-1)^{s-1} \delta_i\left( \vt_{r_s}^{(l_s)} \right)\Pf[ \vt_{r_1}^{(l_1)} \cdots\widehat{ \vt_{r_s}^{(l_s)}}\cdots  \vt_{r_m}^{(l_m)}]
 + (-1)^{s-1} s_i\left( \vt_{r_s}^{(l_s)} \right) \delta_i\Pf[ \vt_{r_1}^{(l_1)} \cdots\widehat{ \vt_{r_s}^{(l_s)}}\cdots  \vt_{r_m}^{(l_m)}]\\
&=& (-1)^{p-1} \vt_{r_p-1}^{(i-1)} \Pf[ \vt_{r_1}^{(l_1)} \cdots\widehat{ \vt_{r_p}^{(i)}}\cdots \vt_{r_q}^{(l_q)}\cdots \vt_{r_m}^{(l_m)}]\\
&& + (-1)^{p-1} \left(\vt_{r_p}^{(i-1)} - t_{i+1}\vt_{r_p-1}^{(i-1)}  \right) \Pf[ \vt_{r_1}^{(l_1)} \cdots\widehat{ \vt_{r_p}^{(i)}}\cdots  \vt_{r_{\!q}-1}^{(-i-1)}\cdots  \vt_{r_m}^{(l_m)}]\\
&& +  (-1)^{q-1} \vt_{r_q-1}^{(-i-1)} \Pf[ \vt_{r_1}^{(l_1)} \cdots \vt_{r_p}^{(i)} \cdots\widehat{ \vt_{r_q}^{(-i)}}\cdots  \vt_{r_m}^{(l_m)}]\\
&& + (-1)^{q-1} \left(\vt_{r_q}^{(-i-1)} + t_{i} \vt_{r_q-1}^{(-i-1)}\right) \Pf[ \vt_{r_1}^{(l_1)} \cdots  \vt_{r_{\!p}-1}^{(i-1)}\cdots\widehat{ \vt_{r_q}^{(-i)}}\cdots  \vt_{r_m}^{(l_m)}]\\
&& + \sum_{s\in \{1,\dots,m\}\backslash \{p,q\}}^{m} (-1)^{s-1}  \vt_{r_s}^{(l_s)} \Pf[ \vt_{r_1}^{(l_1)} \cdots\widehat{ \vt_{r_s}^{(l_s)}}\cdots  \vt_{r_{\!p}-1}^{(i-1)}\cdots  \vt_{r_{\!q}}^{(-i-1)}\cdots  \vt_{r_{m}}^{(l_{m})}]\\
&& + \sum_{s\in \{1,\dots,m\}\backslash \{p,q\}}^{m} (-1)^{s-1}  \vt_{r_s}^{(l_s)} \Pf[ \vt_{r_1}^{(l_1)} \cdots \widehat{ \vt_{r_s}^{(l_s)}}\cdots \vt_{r_{\!p}}^{(i-1)}\cdots  \vt_{r_{\!q}-1}^{(-i-1)}\cdots  \vt_{r_{m}}^{(l_{m})}].
\end{eqnarray*}
By Lemma \ref{theta-relation} and the natural  multilinearity of Pfaffian, 
the sum of the first and second terms is
\begin{eqnarray*}
&&(-1)^{p-1} \vt_{r_p-1}^{(i-1)} \Pf[ \vt_{r_1}^{(l_1)} \cdots\widehat{ \vt_{r_p}^{(i)}}\cdots \vt_{r_q}^{(-i-1)}\cdots \vt_{r_m}^{(l_m)}]
+(-1)^{p-1} \vt_{r_p}^{(i-1)} \Pf[ \vt_{r_1}^{(l_1)} \cdots\widehat{ \vt_{r_p}^{(i)}}\cdots  \vt_{r_{\!q}-1}^{(-i-1)}\cdots  \vt_{r_m}^{(l_m)}].
\end{eqnarray*}
Similarly, the sum of the third and fourth terms is
\begin{eqnarray*}
&& (-1)^{q-1} \vt_{r_q-1}^{(-i-1)} \Pf[ \vt_{r_1}^{(l_1)} \cdots \vt_{r_p}^{(i-1)} \cdots\widehat{ \vt_{r_q}^{(-i)}}\cdots  \vt_{r_m}^{(l_m)}]
 + (-1)^{q-1} \vt_{r_q}^{(-i-1)} \Pf[ \vt_{r_1}^{(l_1)} \cdots  \vt_{r_{\!p}-1}^{(i-1)}\cdots\widehat{ \vt_{r_q}^{(-i)}}\cdots  \vt_{r_m}^{(l_m)}].
\end{eqnarray*}
Thus by the definition of Pfaffian again, we obtain the desired formula. The case when $m$ is even can be proved similarly.
\end{proof}
\subsection{Supplementary results on $\vt$-functions}
In this section, we generalize Proposition \ref{prop square-relation0} and Proposition \ref{delta-i-on-Pf}, which hold in $\calR_{\infty}$. We will use them only in Section \ref{sec: beyond grassmannian}.
\begin{definition}
For each $(r_1,\dots, r_m)$, $(l_1,\dots, l_m) \in \ZZ^m$ and $(k_1,\dots, k_m) \in (\ZZ_{\geq0})^m$, let
\[
\Pf\left[{}_{k_1}\!\vt_{r_1}^{(l_1)}{}_{k_2}\!\vt_{r_2}^{(l_2)}\cdots {}_{k_m}\!\vt_{r_m}^{(l_m)}\right] 
:= \left.\Pf\left[c_{r_1}^{(1)}c_{r_2}^{(2)}\cdots c_{r_m}^{(m)}\right]\right|_{c={}_k\!\vt^{(l)}},
\]
where $|_{c={}_k\!\vt^{(l)}}$ means that we substitute ${}_{k_i}\!\vt_{s}^{(l_i)}$ to $c_{s}^{(i)}$ for all $i \in \{1,\dots,m\}$ and $s \in \ZZ$. 
\end{definition}
\begin{remark}\label{rem: n-ind theta general}
By the definition of Pfaffian and the fact that ${}_k\!\vt_r^{(l)}=0$ for all $r<0$ and ${}_{k}\!\vt_0^{(l)}=1$, we have
\begin{eqnarray*}
&&\Pf[{}_{k_1}\!\vt_{r_1}^{(l_1)} \cdots {}_{k_m}\!\vt_{r_m}^{(l_m)} {}_{k_{m+1}}\!\vt_{0}^{(l_{m+1})}] = \Pf[{}_{k_1}\!\vt_{r_1}^{(1)} \cdots {}_{k_m}\!\vt_{r_m}^{(m)}], \\ 
&&\Pf[{}_{k_1}\!\vt_{r_1}^{(l_1)} \cdots {}_{k_m}\!\vt_{r_m}^{(l_m)} {}_{k_{m+1}}\!\vt_{r_{m+1}}^{(l_{m+1})}] = 0 \ \ \ \mbox{ if } \ r_{m+1} < 0.
\end{eqnarray*}
\end{remark}
Now by Lemma \ref{square-relation0} with the help of Proposition \ref{SQinPf} (1), we have
the following.
\begin{proposition}\label{square-relation0general}
Suppose $l \geq 0$ and $r > k+l$. Then
\[
\Pf[{}_{k_1}\!\vt_{r_1}^{(l_1)}\cdots {}_{k}\!\vt_{r}^{(l)} {}_{k}\!\vt_{r}^{(l)} \cdots {}_{k_m}\!\vt_{r_m}^{(l_m)}] = 0.
\]
\end{proposition}
The proofs of the following lemma and proposition are identical to the ones of Lemma \ref{delta > sum} and Propsition \ref{delta-i-on-Pf}.
\begin{lemma}\label{delta > sum general} For $i>0$, we have
\[
\delta_i({}_{k_1}\!\vt_{r}^{(i)}\cdot{}_{k_2}\!\vt_s^{(-i)}) 
= {}_{k_1}\!\vt_{r-1}^{(i-1)}\cdot{}_{k_2}\!\vt_s^{(-i-1)} + {}_{k_1}\!\vt_r^{(i-1)}\cdot{}_{k_2}\!\vt_{s-1}^{(-i-1)}.
\]
\end{lemma}

\begin{proposition}\label{delta-i-on-Pf general} Let $i\geq 0$.
\begin{itemize}
\item[(a)] If $l_p\not=\pm i$ for all $p$, then  $\delta_i\Pf[{}_{k_1\!}\!\vt_{r_1}^{(l_1)} \cdots {}_{k_m\!}\!\vt_{r_m}^{(l_m)}]=0$.
\item[(b)] Suppose that $l_p\in \{\pm i\}$ for some $p$ and that $l_q\notin \{\pm i\}$ for all $q\not=p$. Then we have
\begin{eqnarray*}
&&\delta_i\Pf[{}_{k_1}\!\vt_{r_1}^{(l_1)} \cdots {}_{k_p}\!\vt_{r_{\!p}}^{(\pm i)}\cdots {}_{k_m}\!\vt_{r_m}^{(l_m)}] 
= \Pf[{}_{k_1}\!\vt_{r_1}^{(l_1)} \cdots {}_{k_p}\!\vt_{r_{\!p}-1}^{(\pm i -1 )}\cdots {}_{k_m}\!\vt_{r_m}^{(l_m)}].
\end{eqnarray*}
\item[(c)] Suppose that $i\not=0$ and that $l_p=i$ and $l_q=-i$ for some $p<q$ and that $l_s\in \{\pm i\}$ for all $s\not\in\{p,q\}$. Then we have
\begin{eqnarray*}
&&\delta_i\Pf[{}_{k_1}\!\vt_{r_1}^{(l_1)} \cdots {}_{k_p}\!\vt_{r_p}^{(i)}\cdots {}_{k_q}\!\vt_{r_{\!q}}^{(-i)}\cdots {}_{k_m}\!\vt_{r_m}^{(l_m)}]\\
&=&\Pf[{}_{k_1}\!\vt_{r_1}^{(l_1)} \cdots {}_{k_p}\!\vt_{r_{\!p}-1}^{(i-1)}\cdots {}_{k_q}\!\vt_{r_{\!q}}^{(-i-1)}\cdots {}_{k_m}\!\vt_{r_{m}}^{(l_{m})}]
+\Pf[{}_{k_1}\!\vt_{r_1}^{(l_1)} \cdots {}_{k_p}\!\vt_{r_{\!p}}^{(i-1)}\cdots {}_{k_q}\!\vt_{r_{\!q}-1}^{(-i-1)}\cdots {}_{k_m}\!\vt_{r_{m}}^{(l_{m})}].
\end{eqnarray*}
\end{itemize}
\end{proposition}
\subsection{Double theta polynomials as equivariant Chern classes}\label{sec:Chern}
In this section, we show that the double theta polynomials ${}_k\vt_r^{(l)}$ correspond to the Chern classes of vector bundles. The result is not used in the proof of 
the main theorem.

Let $\mathcal{E}$ be the trivial vector bundle of rank $2n$ over $\mathcal{F}l_n$, and $\mathcal{L}_i, \mathcal{L}_i^* \subset \mathcal{E}$ the subbundles whose fibers are $\Span(\ee_i), \Span(\ee_i^*)$ respectively. Let $T=(\CC^{\times})^{n}$ be the $n$-dimensional torus and let $t_1,\cdots, t_n$ be the standard basis of $\frakt^*_{\ZZ}$. Then $T$ acts on $\mathcal{L}_i$ with the weight $-t_i$ and $\mathcal{L}_i^*$ with the weight $t_i$. Note that 
$t_i=-c_1^T(\mathcal{L}_i).$ Let
\[
\mathcal{L} = \bigoplus_{i=1}^n \mathcal{L}_i,\quad
\mathcal{L}^* = \bigoplus_{i=1}^n \mathcal{L}_i^*,
\]
and hence $\mathcal{E}=\mathcal{L}\oplus \mathcal{L}^*$. Let 
$
\mathcal{V}_n \subset \cdots \subset \mathcal{V}_1=\mathcal{V} \subset \mathcal{E}
$
be the tautological flag of vector bundles over the complete flag variety $\mathcal{F}l_n$ of isotropic subspaces of $V$ where $\mathrm{rank} \,\mathcal{V}_{i} = n-i+1$. 
Let $
\pmb{z}_i=c_1^T(\mathcal{V}_i/\mathcal{V}_{i+1}).
$
Note that $\mathcal{V}_{k+1}$ is the pullback of the tautological subbundle of rank $n-k$ on ${SG}_n^k$ along the natural projection $\sfp: \mathcal{F}l_n \to {SG}_n^k$.
Note that $\mathcal{Q}:= \mathcal{E}/\mathcal{V}_{k+1}$ is the 
universal quotient bundle of ${SG}_n^k.$ 

From the geometric construction of $\pi_n:\mathcal{R}_\infty
\rightarrow H_T^*(\mathcal{F}l_n)$ (\cite[\S 10]{DSP}), we have 
 $\pi_n(Q_r(x)) = c_r(\mathcal{L}^* - \mathcal{V}_1)= c_r(\mathcal{V}^*_1 - \mathcal{L}) $. In other words,
\begin{equation}
\pi_n: \prod_{i=0}^{\infty}  \frac{1+x_iu}{1-x_iu}  \mapsto \prod_{i=1}^{n}  \frac{1+ t_iu  }{1 + \pmb{z}_i u }
=\prod_{i=1}^{n}  \frac{1-\pmb{z}_i u }{1-t_iu  }. \label{eq:pin}
\end{equation}
Define
\[
\mathcal{U}_l:= \bigoplus_{i=l}^n \mathcal{L}_i \quad
\mbox{ if}\; l > 0, \quad \ \ \ 
\mathcal{U}_{-l}:= \mathcal{L} \oplus \bigoplus_{i=1}^{l+1} \mathcal{L}_i^* \ \ \ \ \  \mbox{if }l \geq  0.
\]

\begin{proposition}\label{theta geometry}
For all $-n \leq l \leq n-1$, we have 
\[\pi_n\left(
_k\vt_r^{(l)}(x,z\, |\, t)\right)= c^T_r(\mathcal{E} - \mathcal{V}_{k+1} - \mathcal{U}_{l+1}).
\]In particular, we have 
\[
c^T_r(\mathcal{E} - \mathcal{V}_{k+1} - \mathcal{U}_{r-k}) =
[\Omega_r]_T\quad (1\leq r\leq n+k). 
\]
\end{proposition}
\begin{proof}
In view of the relation (\ref{eq:pin}), the proposition can be
shown by the following formal calculations. 
For $l \geq 0$, we have
\[
c^T(\mathcal{E} - \mathcal{V}_{k+1} - \mathcal{U}_{l+1}) = 
\frac{\prod_{i=1}^n(1 - t_i^2u^2)}{\prod_{i=k+1}^{n}(1+\pmb{z}_iu) \prod_{i=l+1}^n(1-t_iu) } = \prod_{i=1}^{n}  \frac{1+ t_i u }{1 + \pmb{z}_i u }  \prod_{i=1}^k (1+\pmb{z}_i u) \prod_{i=1}^{l}(1-t_i u),
\]
and for $l > 0$,  we have 
\begin{eqnarray*}
c^T(\mathcal{E} - \mathcal{V}_{k+1} - \mathcal{U}_{-l+1})
&=&\frac{\prod_{i=1}^n
(1 - t_i^2u^2)}{\prod_{i=k+1}^{n}(1+\pmb{z}_iu) \prod_{i=1}^n(1-t_iu) \prod_{i=1}^{l} (1 + t_iu) } \\
&=&\prod_{i=1}^{n}  \frac{1+ t_iu  }{1 + \pmb{z}_iu  }  \prod_{i=1}^k (1+\pmb{z}_iu) \prod_{i=1}^{l}\frac{1}{1+t_iu}.
\end{eqnarray*}
The second statement follows from 
the result (\ref{eq:vtr}) due to Wilson.
\end{proof}

\section{Proof of the main theorem}\label{sec:proof}
Fix $k\geq 0$. We omit $k$ in this section and the next, \textit{i.e.} we denote $\vt_r^{(l)}={}_k\vt_r^{(l)}$.

Recall that for $w\in W_n^{(k)}$ we defined $w^\vee \in W_n^{(k)}$ in \S \ref{sec:dual}. 
\begin{definition}
Let
\[
\Theta_{max}^{(n,k)}(x,z\, |\, t):=\Pf[\vt_{n+k}^{(n-1)}\vt_{n+k-1}^{(n-2)}\cdots \vt_{2k+1}^{(k)}].
\]
For each $\lambda \in \mathcal{P}_n^{(k)}$, define 
\begin{equation}
\Theta_\lambda^{(n,k)}(x,z\, |\, t):=\delta_{(w_\lambda^{(k)})^\vee} \Theta_{max}^{(n,k)}(x,z\, |\, t).
\end{equation}
\end{definition}
\begin{theorem}[Pfaffian sum formula for $\Theta_\lambda$]\label{MainProposition}
Let $\lambda\in\mathcal{P}_n^{(k)}.$ We have 
\begin{equation}\label{main pf sum}
\Theta_{\lambda}^{(n,k)}(x,z|t)=\sum_{I\subset {D}(\lmd)}\Pf\left[\vt_{\lambda_1+a^I_1}^{(\chi_1)} \cdots  \vt_{\lambda_{n-k}+a^I_{n-k}}^{(\chi_{n-k})}\right],
\end{equation}
where $I$ runs over all subsets  of $D(\lambda)$ and  $a_s^I=\#\{j\;|\;(s,j)\in I\}-\#\{i\;|\;(i,s)\in I\}.$
\end{theorem}
\begin{proof}
For simplicity of notation, we drop the superscript $(k)$ from $w_\lambda^{(k)}.$
We proceed by induction on  $\ell(w_\lambda^\vee).$ If $\ell(w_\lambda^\vee)=0$, then $w_\lambda=w_{max}$, and the result is obvious from the definition. Suppose that $\ell(w_\lambda^\vee)>0$. There is a strict $k$-partition $\lambda'$ and $i \in \{0,1,\dots, n-1\}$ such that $w_{\lambda'}\in W^{(k)}_n$, $s_iw_{\lambda'}=w_{\lambda}$, and $\ell(w_\lambda)=\ell(w_{\lambda'}) - 1$. By Lemma \ref{LGLEMMA} and \ref{LGLEMMA0}, $w_{\lambda'}$ is in one of the cases $L1, L2, L3$ and $L0$. Let $\chi_{\lambda}=(\chi_1,\dots,\chi_{n-k})$ and $\chi_{\lambda'}=(\chi_1',\ldots,\chi_{n-k}')$ be the characteristic indices of $\lambda$ and $\lambda'$ respectively. 
By the induction hypothesis we have
\[
\delta_{w_{\lambda'}^\vee}\Pf[\vt_{n+k}^{(n-1)}\vt_{n+k-1}^{(n-2)}\cdots \vt_{2k+1}^{(k)}]=\sum_{I\subset {D}(\lmd')}\Pf\left[\vt_{\lambda_1'+a^I_1}^{(\chi_1')} \cdots  \vt_{\lambda_{n-k}'+a^I_{n-k}}^{(\chi_{n-k}')}\right].
\]

In the cases $L2$, $L3$, or $L0$, we have $D(\lambda')=D(\lambda)$. Furthermore, for some $p$,  $\chi_p'=\pm i=\chi_p+1$ and $\lambda_p'=\lambda_p+1$; $\chi'_q=\chi_q$ and $\lambda_q=\lambda_q'$ for all $q\not=p$. Thus by Proposition \ref{delta-i-on-Pf}, we can compute
\begin{eqnarray*}
&&\delta_{w_{\lambda}^\vee}\Pf[\vt_{n+k}^{(n-1)}\vt_{n+k-1}^{(n-2)}\cdots \vt_{2k+1}^{(k)}]\\
&=&\delta_i\delta_{w_{\lambda'}^\vee}\Pf[\vt_{n+k}^{(n-1)}\vt_{n+k-1}^{(n-2)}\cdots \vt_{2k+1}^{(k)}]\\
&=&\sum_{I\subset {D}(\lmd')}\delta_i\Pf\left[\vt_{\lambda_1'+a^I_1}^{(\chi_1')} \cdots\vt_{\lambda_a'+a^I_p}^{(\pm i)}\cdots \cdots \vt_{\lambda_{n-k}'+a^I_{n-k}}^{(\chi_{n-k}')}\right]\\
&=&\sum_{I\subset {D}(\lmd)}\Pf\left[\vt_{\lambda_1'+a^I_1}^{(\chi_1')} \cdots\vt_{\lambda_p'-1+a^I_p}^{(\pm i-1)}\cdots  \vt_{\lambda_{n-k}'+a^I_{n-k}}^{(\chi_{n-k}')}\right]\\
&=&\sum_{I\subset {D}(\lmd)}\Pf\left[\vt_{\lambda_1+a^I_1}^{(\chi_1)} \cdots\vt_{\lambda_p+a^I_p}^{(\chi_p)}\cdots  \vt_{\lambda_{n-k}+a^I_{n-k}}^{(\chi_{n-k})}\right],
\end{eqnarray*}
where the first equality follows from Proposition \ref{prop:delta_w2}, the second is the induction hypothesis, and the third follows by Proposition \ref{delta-i-on-Pf}. 

In the case $L1$, we have $D(\lambda)= D(\lambda') \sqcup \{(p,q)\}$. Furthermore $\chi'_p=i=\chi_p+1$, $\chi'_q=-i=\chi_q+1$, $\lambda_p=\lambda_p'-1$ and $\lambda_q'=\lambda_q$ for some $p$ and $q$;  $\chi_r'=\chi_r$ and $\lambda_r=\lambda_r'$ for all $r\not=p,q$. Here note that $\lambda_p'\not=0$ so that $\lambda_p\geq 0$. The claim now follows from the computation:
\begin{eqnarray*}
&&\delta_{w_{\lambda}^\vee}\Pf[\vt_{n+k}^{(n-1)}\vt_{n+k-1}^{(n-2)}\cdots \vt_{2k+1}^{(k)}]\\
&=&\delta_i\delta_{w_{\lambda'}^\vee}\Pf[\vt_{n+k}^{(n-1)}\vt_{n+k-1}^{(n-2)}\cdots \vt_{2k+1}^{(k)}]\\
&=&\sum_{I\subset {D}(\lmd')}
\delta_i\Pf\left[\vt_{\lambda_1'+a^I_1}^{(\chi_1')} \cdots\vt_{\lambda_p'+a^I_p}^{(i)}\cdots\vt_{\lambda_q'+a^I_q}^{(-i)}\cdots \cdots \vt_{\lambda_{n-k}'+a^I_{n-k}}^{(\chi_{n-k}')}\right]\\
&=&\sum_{I\subset {D}(\lmd')}\left(
\Pf\left[\vt_{\lambda_1'+a^I_1}^{(\chi_1')} \cdots\vt_{\lambda_p'+a^I_p-1}^{(i-1)}\cdots\vt_{\lambda_q'+a^I_q}^{(-i-1)}\cdots \cdots \vt_{\lambda_{n-k}'+a^I_{n-k}}^{(\chi_{n-k}')}\right]\right.\\
&&\ \ \ \ \ \ \ \ \ \ \ \ \ \ \ \ + \left.
\Pf\left[\vt_{\lambda_1'+a^I_1}^{(\chi_1')} \cdots\vt_{\lambda_p'+a^I_p}^{(i-1)}\cdots\vt_{\lambda_q'+a^I_q-1}^{(-i-1)}\cdots \cdots \vt_{\lambda_{n-k}'+a^I_{n-k}}^{(\chi_{n-k}')}\right]\right)\\
&=&\sum_{I\subset {D}(\lmd')}
\Pf\left[\vt_{\lambda_1+a^I_1}^{(\chi_1)} \cdots\vt_{\lambda_p+a^I_p}^{(\chi_p)}\cdots\vt_{\lambda_q+a^I_q}^{(\chi_q)}\cdots \cdots \vt_{\lambda_{n-k}+a^I_{n-k}}^{(\chi_{n-k})}\right]\\
&&\ \ \ \ \ \ \ \ \ \ \ \ \ \ \ \ + \sum_{I\subset {D}(\lmd')}
\Pf\left[\vt_{\lambda_1+a^I_1}^{(\chi_1)} \cdots\vt_{\lambda_p+a^I_p+1}^{(\chi_p)}\cdots\vt_{\lambda_q+a^I_q-1}^{(\chi_q)}\cdots \cdots \vt_{\lambda_{n-k}+a^I_{n-k}}^{(\chi_{n-k})}\right]\\
&=&\sum_{I\subset {D}(\lmd')}
\Pf\left[\vt_{\lambda_1+a^I_1}^{(\chi_1)} \cdots\vt_{\lambda_p+a^I_p}^{(\chi_p)}\cdots\vt_{\lambda_q+a^I_q}^{(\chi_q)}\cdots \cdots \vt_{\lambda_{n-k}+a^I_{n-k}}^{(\chi_{n-k})}\right]\\
&&\ \ \ \ \ \ \ \ \ \ \ \ \ \ \ \ + \sum_{(p,q)\in I\subset {D}(\lmd)}
\Pf\left[\vt_{\lambda_1+a^I_1}^{(\chi_1)} \cdots\vt_{\lambda_p+a^I_p}^{(\chi_p)}\cdots\vt_{\lambda_q+a^I_q}^{(\chi_q)}\cdots \cdots \vt_{\lambda_{n-k}+a^I_{n-k}}^{(\chi_{n-k})}\right]\\
&=&\sum_{I\subset {D}(\lmd)}
\Pf\left[\vt_{\lambda_1+a^I_1}^{(\chi_1)} \cdots\vt_{\lambda_p+a^I_p}^{(\chi_p)}\cdots\vt_{\lambda_q+a^I_q}^{(\chi_q)}\cdots \cdots \vt_{\lambda_{n-k}+a^I_{n-k}}^{(\chi_{n-k})}\right],
\end{eqnarray*}
where the first equality follows from Proposition \ref{prop:delta_w2}, the second is the induction hypothesis, the third follows by Proposition \ref{delta-i-on-Pf}, and the second last equality holds, since, for each $I \in D(\lambda')$ and $J:=I \cup \{(p,q)\} \in D(\lambda)$, $a^I$ and $a^J$ are related by
\[
a_p^I + 1=a_p^{J} , \ \ \  a_q^I - 1 = a_q^J, \ \ \ \mbox{ and } \ \ \ a_r^I = a_r^J \ \forall r \not=p,q.
\]
\end{proof}
\begin{remark}
We owe H. Naruse for pointing out, in the early stage of this work, that $\mathfrak{C}_{w_{max}}$ has a Pfaffian expression.
\end{remark}
\begin{remark}\label{rem:n-indep}
The expression of the right hand side of Theorem \ref{MainProposition} is essentially independent of $n.$ More precisely,  if $\lambda\in \mathcal{P}_n^{(k)}$, then we have obviously $\lambda\in \mathcal{P}_{n+1}^{(k)}$ and all nonzero Pfaffians appearing in the formulas for $\Theta_{\lambda}^{(n,k)}$ and $\Theta_{\lambda}^{(n+1,k)}$ naturally coincide. We have $\Theta_{\lambda}^{(n,k)}=\Theta_{\lambda}^{(n+1,k)}$, in particular. This fact can be checked by using Remark \ref{rem: n-ind theta}. In fact, one can check that the lower indexes (degree) of the right end $\vt$ in the Pfaffians appearing in the formula of $ \Theta_{\lambda}^{(n+1,k)}$ are less than or equal to zero.
\end{remark}
\begin{proposition}[Stability of $\Theta_\lambda$] 
Let $\lambda\in\mathcal{P}_n^{(k)}.$ For all $m\geq n$, we have  \[\Theta_\lambda^{(m,k)}(x,z\, |\, t) =\Theta_\lambda^{(n,k)}(x,z\, |\, t).\]
\end{proposition}
\begin{proof}
This is a consequence of the Pfaffian sum formula  in Theorem \ref{MainProposition} (see Remark \ref{rem:n-indep}).
\end{proof}
By the above proposition and Proposition \ref{prop:theta_is_invariant},
 for each $\lambda\in \mathcal{P}^{(k)}_\infty,$ we can define  $\Theta^{(k)}_\lambda(x,z|t)$ to be 
 the element of $\calR_{\infty}^{(k)}$
 such that 
  $\Theta_{\lambda}^{(k)}(x,z|t)
  =\Theta^{(n,k)}_\lambda(x,z|t)$  for any $n$ such that $\lambda\in \mathcal{P}_n^{(k)}.$
\begin{lemma}\label{lem:VanishDelta} 
We have 
$\delta_j\Theta_{max}^{(n,k)}=0$ for $j\neq k.$
\end{lemma}
\begin{proof}
If $j\neq k$, then we are in the situation of (a) or (b) in Proposition \ref{delta-i-on-Pf}. If (b) is the case, then the claim follows from Proposition \ref{SQinPf} (together with Lemma \ref{square-relation0}).
\end{proof}
\begin{proposition}\label{prop:divdif}
Let $\lambda \in \calP^{(k)}_\infty$ and $w_{\lambda}^{(k)}$ the corresponding element in $W_\infty^{(k)}.$ 
We have
\[
\delta_i\Theta_{\lambda}^{(k)}=\begin{cases}
\Theta_{\mu}^{(k)}& \mbox{if}\;  
s_iw_\lambda^{(k)}=w_\mu^{(k)}\;(\mu\in\calP^{(k)}_\infty,\;\mu\subset\lambda,\;|\mu|=|\lambda|-1),
\\
0 & \mbox{otherwise}.
\end{cases}
\]
\end{proposition}
\begin{proof} 
By the definition of $\Theta_{\lambda}^{(n,k)}$
and the fact $\delta_i\Theta_{max}^{(n,k)}=0$ for $i\neq k$ (Lemma \ref{lem:VanishDelta}), 
the result follows from Proposition \ref{prop:delta_w2}
immediately.
\end{proof}
\begin{lemma}\label{localization at empty} 
We have
\[
\Theta_{\lambda}^{(k)}|_{\emptyset} = \delta_{\lambda, \emptyset},
\]
where the notation $|_{\mu}$ is defined in Definition \ref{loc def}, and $\emptyset$ denotes the empty partition.
\end{lemma}
\begin{proof}
We have $\Theta_{\emptyset}^{(k)}=1$ since ${}_k\vt_{0}^{(\ell)} = 1$ and ${}_k\vt_{m}^{(\ell)} = 0$ for all $m<0$, and hence $\left.\Theta_{\emptyset}^{(k)}\right|_{\emptyset}=1$. 

Now assume $\lambda\not=\emptyset$. For each $g\in \calR_{\infty}^{(k)}$, the polynomial $g|_{\emptyset} \in \ZZ[t]$ is given by the specializations
\begin{eqnarray*}
(z_1,\dots,z_k) \mapsto (t_1,\dots, t_k)\ \ \ \mbox{ and } \ \ \  (x_1,x_2,\dots)\mapsto (0,0,\dots ).
\end{eqnarray*}
The generating function of $\vt$-functions in Definition \ref{def-theta} becomes a polynomial in $u$ of degree $k+l$ after we specialize it as above. Thus, by the degree reason, we have
\begin{equation}\label{van k+l<m>0}
\left.{}_k\vt_m^{(l)}\right|_{\emptyset} = 0 \ \ \mbox{ if $\ell + k < m$ and $0 < m$}. 
\end{equation}
We can expand the right hand side of (\ref{main pf sum})  as a polynomial in terms of the $\vt$-functions ${}_k\vt_m^{(l)}$ by using the definition of $\Pf$ in Section \ref{sec:Pf}, in such a way that each monomial contains ${}_k\vt_{\lambda_1+a_1^I + j}^{(\chi_1)}$ (the first factor) for some $j\geq 0$.  Note that $a_1^I\geq 0$ since $\{i \ |\ (i,1) \in I\}$ is always empty. Since $0< \lambda_1\leq \lambda_1+a_1^I + j$ and $\chi_1 + k = \lambda_1 - 1< \lambda_1$ (see (\ref{lambda_i+s})),  we have $\left.{}_k\vt_{\lambda_1+a_1^I + j}^{(\chi_1)}\right|_{\emptyset} = 0$ by (\ref{van k+l<m>0}). Therefore $\left.\Theta_{\lambda}^{(k)}\right|_{\emptyset}=0$ for all $\lambda \in \calP_{\infty}^{(k)}$.
\end{proof}
\begin{theorem}\label{prop:Theta=C} 
Let $\lambda\in \mathcal{P}^{(k)}_\infty.$ We have $$\mathfrak{C}_{w_{\lambda}^{(k)}}(z,t;x)= \Theta_\lambda^{(k)}(x,z\, |\, t).$$
\end{theorem}
\begin{proof}
By Proposition \ref{prop:theta_is_invariant} and Proposition \ref{MainProposition}, we have $\Theta_{\lambda}^{(k)} \in \calR_{\infty}^{(k)}$. By Proposition \ref{prop:divdif} and Lemma \ref{localization at empty}, the family $\Theta_{\lambda}^{(k)}, \lambda \in \calP_{\infty}^{(k)}$ satisfies the equations (\ref{lddeq}), (\ref{vanishingempty}) in Lemma \ref{uniqueness}. Therefore, by the uniqueness, $\frakC_{w_\lambda^{(k)}}$ must coincide with $\Theta_{\lambda}^{(k)}$ for all $\lambda \in \calP_{\infty}^{(k)}$. \end{proof}
\begin{corollary}\label{rem:generated by theta}
The ring $\mathcal{R}_\infty^{(k)}$ is generated by  ${}_k\vt_r^{(0)}\;(r\geq 1)$ as an algebra over $\ZZ[t]$.
\end{corollary}
\begin{proof}
By Theorem \ref{prop:Theta=C}, each $\frakC_{w_{\lambda}^{(k)}}$ is an element of the ring generated by ${}_k\vt_r^{(l)} \ (r, l\in \ZZ)$. Since each ${}_k\vt_r^{(l)}$ is in the algebra generated by ${}_k\vt_r^{(0)}\;(r\geq 1)$ over $\ZZ[t]$, the claim follows from Proposition \ref{prop:basisR^k}. 
\end{proof}
\section{Raising operators and Wilson's conjecture}\label{sec: IM-TW}
\subsection{Basics on raising operators}
Let $R_{ij}, 1\leq i<j \leq m$ be the operator that act on $\ZZ^m$ by
\[
R_{ij}: (a_1,\dots,a_i, \dots, a_j, \dots, a_m ) \mapsto (a_1,\dots,a_i+1, \dots, a_j-1, \dots, a_m ).
\]
Let $(c^{(1)}, c^{(2)}, \dots ,c^{(m)})$ be an $m$-tuple such that each $c^{(i)}$
is an infinite sequence of variables $c_r^{(i)}\;(r\in \ZZ)$.
The action of $R_{ij}$ on a degree $m$ monomial  $c_{r_1}^{(1)}\cdots c_{r_m}^{(m)}$ is defined by
\[
R_{ij}(c_{r_1}^{(1)}\cdots c_{r_i}^{(i)}\cdots c_{r_j}^{(j)} \cdots c_{r_m}^{(m)}) = c_{r_1}^{(1)}\cdots c_{r_i+1}^{(i)}\cdots c_{r_j-1}^{(j)} \cdots c_{r_m}^{(m)}.
\]
Let  $\calA$ be  the set of all $\ZZ$-linear combinations of monomials $c_{r_1}^{(1)}\cdots c_{r_m}^{(m)}, (r_1,\dots, r_m) \in \ZZ^m$. The action of any polynomial in $R_{ij}$ on $\calA$ is naturally defined. For example,  
\[
(1+R_{ij})(c_{r_1}^{(1)} \cdots c_{r_m}^{(m)}) = c_{r_1}^{(1)}\cdots c_{r_i}^{(i)}\cdots c_{r_j}^{(j)} \cdots c_{r_m}^{(m)}+c_{r_1}^{(1)}\cdots c_{r_i+1}^{(i)}\cdots c_{r_j-1}^{(j)} \cdots c_{r_m}^{(m)}.
\]
Since the actions of the operators $R_{ij}, i<j$ commute, they are extended to the action of the polynomial ring $\ZZ[R_{ij}, 1\leq i<j\leq m]$. For example, take $(1 + R_{ij})(1-R_{i'j'}) \in \ZZ[R_{ij}, 1\leq i<j\leq m]$
\begin{eqnarray*}
(1 + R_{ij})(1-R_{i'j'})(c_{r_1}^{(1)}\cdots c_{r_m}^{(m)}) &=& (1-R_{i'j'})(1 + R_{ij})(c_{r_1}^{(1)}\cdots c_{r_m}^{(m)}) \\
&=& (1-R_{i'j'} + R_{ij} - R_{i'j'}R_{ij})(c_{r_1}^{(1)}\cdots c_{r_m}^{(m)}) .
\end{eqnarray*}

Consider a formal power series $F = \sum_{s=0}^{\infty} F_s$ where each $F_{s}$ is a homogeneous polynomial in $R_{ij}$ of degree $s$, regarding $R_{ij}$'s as formal variables of degree one. Each $F_s$ acts on $\calA$ and  so $F_s(c_{r_1}^{(1)} \cdots c_{r_m}^{(m)})$ is in $\calA$. Thus we obtain the following formal series of $c_{r_1}^{(1)}\dots, c_{r_m}^{(m)}, (r_1,\dots, r_m) \in \ZZ^m$
\[
F(c_{r_1}^{(1)} \cdots c_{r_m}^{(m)}):= \sum_{s=0}^{\infty} F_s(c_{r_1}^{(1)} \cdots c_{r_m}^{(m)}).
\]
It is well-defined since the coefficient of each $c_{s_1}^{(1)}\cdots c_{s_m}^{(m)}$ in the sum is finite. Indeed, the degree $s$ of the operator $F_s$ that creates a particular monomial $c_{s_1}^{(1)}\cdots c_{s_m}^{(m)}$ is bounded. It also has the property that the only appearing terms are such that $s_1+\cdots + s_m = r_1 + \cdots + r_m$. Considering those properties, we can conclude that 
\[
\left.F(c_{r_1}^{(1)} \cdots c_{r_m}^{(m)})\right|_{\geq 0}
\]
is a polynomial where $|_{\geq 0}$ denotes the substitution $c_r^{(i)}=0$ for all $r<0$ and all $i$.  If $F_1$ and $F_2$ are two formal power series as above, then the product $F_1F_2$ is also such a formal power series and therefore we have
\[
(F_1F_2)(c_{r_1}^{(1)} \cdots c_{r_m}^{(m)}) = F_1(F_2(c_{r_1}^{(1)} \cdots c_{r_m}^{(m)})).
\]
The RHS of this identity is well-defined, \textit{i.e.} it is a formal power series such that the coefficient of each $c_{s_1}^{(1)}\cdots c_{s_m}^{(m)}, (s_1,\dots, s_m) \in \ZZ^m$ is finite.  
\begin{example}\label{exm:R-Pf}
The following formal power series is important for our purpose:
\[
F:=\frac{1-R_{12}}{1+R_{12}} = 1 - 2R_{12} + 2R_{12}^2 - \cdots = 1 + \sum_{s=1}^{\infty} (-1)^s 2 R_{12}^s.
\]
For example, we have 
$F(c_{-2}^{(1)}c_1^{(2)}) = c_{-2}^{(1)}c_1^{(2)} -2c_{-1}^{(1)}c_0^{(2)} +2 c_{0}^{(1)}c_{-1}^{(2)} -2c_{1}^{(1)}c_{-2}^{(2)} + \cdots$, and hence 
we have
$\left.F(c_{-2}^{(1)}c_1^{(2)})\right|_{\geq 0} = 0,$
while
$\left.F(c_{-1}^{(1)}c_3^{(2)})\right|_{\geq 0} = -2c_{0}^{(1)}c_2^{(2)} +2 c_{1}^{(1)}c_1^{(2)} -2 c_{2}^{(1)}c_0^{(2)}.
$
\end{example}
The following lemma is obvious from the definition.
\begin{lemma}\label{product operator}
Let $I(F)$ be the set of $i$'s such that $R_{ij}$ or $R_{ji}$ appear in $F$. If $I(F_1) \cap I(F_2) = \varnothing$, then we have the following well-defined identity of formal power series
\[
(F_1F_2)(c_{r_1}^{(1)} \cdots c_{r_m}^{(m)}) = \left(\prod_{i \not\in I(F_1) \cup I(F_2)} c_{r_i}^{(i)}\right)\cdot F_1\left( \prod_{i\in I(F_1)} c_{r_i}^{(i)}\right)\cdot F_2\left( \prod_{i\in I(F_2)} c_{r_i}^{(i)}\right).
\]
\end{lemma}
\subsection{Pfaffians in terms of raising operators}
We have the following description of Pfaffians.
Let $\Delta_m:=\{(i,j)\in \ZZ^2\;|\;1\leq i<j\leq m\}.$
\begin{proposition}\label{prop:R-Pf}We have
\[
\left.\Pf[c_{r_1}^{(1)}\cdots c_{r_m}^{(m)}] \right|_{\geq 0}= \left. \left(\prod_{(i,j)\in \Delta_m}\frac{1-R_{ij}}{1+R_{ij}} \;c_{r_1}^{(1)} \cdots c_{r_m}^{(m)}  \right)\right|_{\geq 0}.
\]
\end{proposition}
\begin{proof}
We proceed by induction on $m$.
The cases $m=1$ is obvious.
For $m=2$, the identity
\[
\left.\Pf[c_{r_1}^{(1)}c_{r_2}^{(2)}]\right|_{\geq 0}= \left.\left(\frac{1-R_{12}}{1+R_{12}}  (c_{r_1}^{(1)}c_{r_2}^{(2)})  \right)\right|_{\geq 0}
\]
follows clearly from the definition (\textit{cf.} Example \ref{exm:R-Pf}).
The general case can be deduced from the following identity of formal series: for $m$  even,  
\[
\prod_{(i,j)\in \Delta_m} \frac{1-R_{ij}}{1+R_{ij}} 
= \sum_{s=2}^{m} (-1)^s \frac{1-R_{1s}}{1+R_{1s}} \prod_{(i,j) \in \Delta_m \atop{i, j \in \{1,\dots,m\}\backslash\{1, s\}}} \frac{1-R_{ij}}{1+R_{ij}}, 
\]
and, for $m$ odd,
\[
\prod_{(i,j)\in \Delta_m}\frac{1-R_{ij}}{1+R_{ij}} 
= \sum_{s=1}^{m} (-1)^{s-1} \prod_{(i,j) \in \Delta_m \atop{i,j\in \{1,\dots,m\} \backslash \{s\}}} \frac{1-R_{ij}}{1+R_{ij}}.
\]
Since $R_{ij}R_{js}= R_{is}$, the proof of 
these equations 
can be reduced to showing the equations for the  
rational functions obtained from replacing 
$R_{ij}$ with $y_i/y_j.$
Such equations goes back to Schur (\cite[p.226]{Schur}).
\end{proof}

\subsection{Pfaffian sum formula and Wilson's conjecture}
\begin{definition}
Let $\lambda$ be a $k$-strict partition contained in the $(n-k)\times (n+k)$ rectangle and $\chi \in \ZZ^{n-k}$ the corresponding characteristic index. Let $\left(\cdot\right)|_{c=\vt^{(\chi)}}$ be the substitution of $\vt_{r_i}^{(\chi_i)}$ to $c_{r_i}^{(i)}$ for each $r_i \in \ZZ$ and $i=1,\dots, n-k$. Define
\[
R_{\lambda}[\vt_{\lambda_1}^{(\chi_1)}\cdots \vt_{\lambda_{n-k}}^{(\chi_{n-k})}]:=\left.
\left(  \prod_{(i,j)\in D(\lambda)^c} \frac{1-R_{ij}}{1+R_{ij}} \prod_{(i,j)\in D(\lambda)}(1-R_{ij})(c_{\lambda_1}^{(1)}\cdots c_{\lambda_{n-k}}^{(n-k)})\right)\right|_{c=\vt^{(\chi)}}.
\]
\end{definition}
By Lemma \ref{lem:chi-lambda} and the remark below, this function coincides with the one defined in Wilson's thesis \cite[Definition 10]{W}. Note that if $D(\lambda)=\Delta_{n-k}$, it is a single determinant.
In fact, the argument in 
\cite[\S 1]{T} shows
\[
\left.
\left(  \prod_{(i,j)\in \Delta_{n-k}}(1-R_{ij})(c_{\lambda_1}^{(1)}\cdots c_{\lambda_{n-k}}^{(n-k)})\right)\right|_{c=\vt^{(\chi)}}
=\mathrm{Det}[\vt_{\lambda_1}^{(\chi_1)}\cdots \vt_{\lambda_{n-k}}^{(\chi_{n-k})}].
\]
If $D(\lambda)=\emptyset$,  then the definition gives a single Pfaffian
$\mathrm{Pf}[\vt_{\lambda_1}^{(\chi_1)}\cdots \vt_{\lambda_{n-k}}^{(\chi_{n-k})}]$ by Proposition \ref{prop:R-Pf}.

\begin{remark}\label{theta of TW and IM}
Let $\theta_p^r[j]$ be the function defined at Definition 7 in \cite{W}, then
\[
\theta_p^r[j] = \begin{cases}
{}_k\vt_{p+j}^{r} &\mbox{if}\;  k<p \mbox{ and } r \leq p-k-1,\\
{}_k\vt_{p+j}^{-r} & \mbox{if}\; k\geq p.\\
\end{cases}
\]
\end{remark}

Finally the following proposition shows that Theorem \ref{thm:Theta=C} is equivalent to Conjecture 1 in \cite{W}, and therefore Corollary \ref{Pf-Det} follows.
\begin{proposition} Let $\lambda$ be a $k$-strict partition in $\mathcal{P}_n^{(k)}$ and $\chi$ the corresponding characteristic index. We have
\[
\sum_{I\subset {D}(\lmd)}\Pf\left[\vt_{\lambda_1+a^I_1}^{(\chi_1)} \cdots  \vt_{\lambda_{n-k}+a^I_{n-k}}^{(\chi_{n-k})}\right] = R_{\lambda}[\vt_{\lambda_1}^{(\chi_1)}\cdots \vt_{\lambda_{n-k}}^{(\chi_{n-k})}].
\]
\end{proposition}
\begin{proof} We have
\begin{eqnarray*}
&&\sum_{I\subset {D}(\lmd)}\Pf\left[\vt_{\lambda_1+a^I_1}^{(\chi_1)} \cdots  \vt_{\lambda_{n-k}+a^I_{n-k}}^{(\chi_{n-k})}\right] \\
&=&\left. \left( \prod_{(i,j) \in \Delta_{n-k}} \frac{1-R_{ij}}{1+R_{ij}} \sum_{I\subset {D}(\lambda)}c_{\lambda_1+a^I_1}^{(1)} \cdots  c_{\lambda_{n-k}+a^I_{n-k}}^{({n-k})}\right)\right|_{c=\vt^{(\chi)}}\\
&=&\left.\left( \prod_{(i,j) \in \Delta_{n-k}} \frac{1-R_{ij}}{1+R_{ij}} \cdot \prod_{(i,j)\in D(\lambda)}(1 + R_{ij}) (c_{\lambda_1}^{(1)}\cdots c_{\lambda_{n-k}}^{(n-k)}) \right)\right|_{c=\vt^{(\chi)}}\\
&=& \left.\left(   \prod_{(i,j) \in D(\lambda)^c} \frac{1-R_{ij}}{1+R_{ij}}     \cdot \prod_{(i,j) \in D(\lambda)} (1 - R_{ij}) (c_{\lambda_1}^{(1)}\cdots c_{\lambda_{n-k}}^{(n-k)}) \right)\right|_{c=\vt^{(\chi)}}\\
&=& R_{\lambda}[\vt_{\lambda_1}^{(\chi_1)}\cdots \vt_{\lambda_{n-k}}^{(\chi_{n-k})}],
\end{eqnarray*}
where the first equality follows from the linearity of $(\cdot)|_{c=\vt^{(\chi)}}$ and the operators, the second follows from the definition of $a^I$, the third follows from Lemma \ref{product operator}, and the last is the definition of $R_{\lambda}$.
\end{proof}
\section{Pfaffian sum formula beyond Grassmannians}\label{sec: beyond grassmannian}
In this section, we show that our technique of deriving the Pfaffian sum formula can be applied beyond the $k$-Grassmannian elements. Let $G={Sp}_{2n}(\CC)$ as before. First we derive a single Pfaffian formula for the polynomial corresponding to the top class of a symplectic partial flag variety. Then we introduce a certain group of signed permutations, called \emph{pseudo $k$-Grassmannian elements}, and show the Pfaffian sum formula for each polynomial corresponding to those. We conclude by remarking the possibility of extending our computation further with the example of \emph{all} signed permutations when $n=3$.
\subsection{The longest elements for partial flag varieties}\label{subsec: pseudo}
For $J \subset \{0,\dots, n-1\}$, let $W_{J}$ be the subgroup of $W_{\infty}$ generated by $s_i, i\not\in J$. The set of the minimum-length coset representatives of $W_{\infty}/W_{J}$ is 
\[
W^{J} := \{  s \in W_{\infty} \ |\ \ell(w) > \ell(ws_i) \ \ \forall i\geq 0, i\not\in J\}.
\]
The Schubert varieties of the generalized flag variety $G/P_{J}$ are indexed by $W_n^{J}:=W_n \cap W^{J}$ where $G={Sp}_{2n}(\CC)$ as before and $P_{J}$ is the parabolic subgroup associated to $J$.
\begin{lemma}
Let $J=\{k_1<\dots<k_p\}$. The elements of $W_n^{J}$ are the signed permutations $w=(w_1\cdots w_n)$ such that
\begin{equation}\label{p step flag perm}
\begin{cases}
w_1<\cdots <w_{k_2}, w_{k_2+1}<\cdots < w_{k_3}, \cdots, w_{k_p+1} < \cdots < w_n &\mbox{if}\;   k_1=0,\\
0<w_1<\cdots <w_{k_1}, w_{k_1+1}<\cdots < w_{k_2}, \cdots, w_{k_p+1} < \cdots < w_n    &\mbox{if}\;  k_1\not=0.
\end{cases}
\end{equation}
Furthermore, the longest element $w$ in $W_n^{J}$ is given by
\[\begin{cases}
w=\overline{k_2}\,\overline{k_2-1}\,\dots\, \overline{1}\,|\, \overline{k_3}\,\overline{k_3-1}\cdots \overline{k_2+1}\,|\,\cdots\,|\,\overline{n}\,\overline{n-1}\,\cdots\,\overline{k_p+1}  &\mbox{if}\;  k_1=0,\\
w= 12\cdots k_1\,|\,\overline{k_2}\,\overline{k_2-1}\,\dots\, \overline{k_1+1}\,|\,\overline{k_3}\,\overline{k_3-1}\cdots \overline{k_2+1}\,|\,\cdots\,|\,\overline{n}\,\overline{n-1}\,\cdots\,\overline{k_p+1} & \mbox{if}\;   k_1\not=0.
\end{cases}\]
\end{lemma}
Here the vertical lines in one line notation are just only to emphasize the descents. For example, the longest element of $W_7^{\{0,3,5\}}$ is $w=\bar3\bar2\bar1|\bar5\bar4|\bar7\bar6$, and the longest element of $W_9^{\{2,6,8\}}$ is $w = 12|\bar6\bar5\bar4\bar3|\bar8\bar7|\bar9$.
\begin{proof}
Since $W_{n,J}:=W_{J} \cap W_n$ is generated by $s_i, i\in J$, it is obvious that those permutations are coset representatives for both cases $0\in J$ and $0\not\in J$. Recall that the length of a signed permutation is the ``number of inversions" (Proposition 8.1.1 \cite{BjBr})
\[
\inv(w) := \inv(w_1,\dots,w_n) + \mbox{neg}(w_1,\dots,w_n) + \mbox{nsp}(w_1,\dots,w_n),
\]
where $\mbox{neg}$ is the number of negative numbers and $\mbox{nsp}$ is the number of pairs whose sums are negative. Let $w$ be the one at (\ref{p step flag perm}) and let $v$ be an arbitrary element in the coset $w W_{J}^n$. The first part of the claim follows, just by observing that each term in $\inv(v)$ is greater than or equal to the corresponding term of $\inv(w)$. For the second part, it is clear that such $w$ realizes the largest number of inversions, which can be actually computed from the definition of inversions:
\begin{equation}\label{dim of partial flag}
\inv(w)=
\begin{cases}
\sum_{i=2}^p k_i(n-k_i) + n + \frac{n(n+1)}{2}&\mbox{if}\;  k_1=0,\\
\sum_{i=1}^p k_i(n-k_i) + n-k_1 + \frac{(n-k_1)(n-k_1+1)}{2} + \frac{k_1(n-k_1)}{2} &\mbox{if}\;  k_1\not=0.
\end{cases}
\end{equation}
\end{proof}
\begin{corollary}\label{red ex}
Let $w_0$ be the longest element in $W_n$.  Let $J=\{0=k_1<\dots<k_p\}$. Let $k_{p+1}=n$. For each $i=1,\dots, p$, let
\[
\underline{v_i}:= (k_{i+1}-1, k_{i+1}-2, \dots, k_i+2, k_i+1, k_{i+1}-1, k_{i+1}-2, \dots, k_i+2, k_{i+1}-1, \dots,  k_{i+1}-1, k_{i+1}-2, k_{i+1}-1),
\]
where we set $\underline{v_i} = \emptyset$ if $k_{i+1}=k_i+1$.
\begin{itemize} 
\item[(a)] Let $w$ be the longest element  in $W_n^{J}$. Let $(i_1,\dots, i_r)$ be a reduced word of $w$, \textit{i.e.} $r:=\ell(w)$ and $w=s_{i_1}\cdots s_{i_r}$. Then
\[
(\underline{v_1},\dots, \underline{v_p},  i_1,\dots, i_r)
\]
is a reduced word for $w_0$. 
\item[(b)] Let $w'$ be the longest element  in $W_n^{J\backslash\{k_1\}}$ and $(i_1',\dots, i_r')$ a reduced word of $w'$. Let $v$ be the longest element of $W_{k_2}^{(0)}$ and $(j_1,\dots, j_{s})$ a reduced word for $v$. Then
\[
(\underline{v_1},\dots, \underline{v_p}, j_1,\dots, j_{s},  i_1',\dots, i_r')
\]
is a reduced word for $w_0$.
\end{itemize}
\end{corollary}
\begin{example}
Let $J=\{0\}$ and $n=4$. We have $\underline{v_1}=(3,2,1,3,2,3)$, $w_0 = s_3s_2s_1s_3s_2s_3 w$ and $\ell(w_0) = 6 + \ell(w)$.

Let $J=\{0,3,6,7\}$ and $n=9$. We have $\underline{v_1}= (2,1,2)$, $\underline{v_2}= (5,4,5)$, $\underline{v_3}= \emptyset$, and $\underline{v_4}= (8)$. Thus $w_0=(s_2s_1s_2)(s_5s_4s_5)s_8 w$ and $\ell(w_0) = \ell(w) + 7$.

Let $J=\{3,6,7\}$ and $n=9$. Then $w_0 = (s_2s_1s_2)(s_5s_4s_5)s_8  (s_0s_1s_2s_0s_1s_0)w'$ and $\ell(w_0) = \ell(w') + 13$. where $w'$ is the longest element of $W_9^{\{3,6,7\}}$ and $(0,1,2,0,1,0)$ is a reduced word of the longest element of $W_3^{(0)}$.
\end{example}
\subsection{Pfaffian formula for the longest elements for partial flag varieties}
The key fact in this section is that the top function $\frakC_{w_0}$ is written in a Pfaffian form. Then we can operate $\delta_i$'s on $\frakC_{w_0}$ while keeping the Pfaffian form to some extent. Recall from Remark \ref{rem: rel to Ik and Ka} that, for a strict partition $\lambda=(\lambda_1,\dots,\lambda_n) \in  \calP_{\infty}^{(0)}$,  we have
\[
Q_{\lambda}(x|t) = \Pf[{}_0\vt_{\lambda_1}^{\lambda_1-1}{}_0\vt_{\lambda_2}^{\lambda_2-1}\cdots {}_0\vt_{\lambda_n}^{\lambda_n-1}].
\]
Let $\rho = (2n-1,2n-3, \dots, 3,1)$. Ikeda-Mihalcea-Naruse \cite[Theorem 1.2]{DSP} proved that for the longest element $w_0$ in $W_n$ we have
\[
\frakC_{w_{0}} = Q_{\rho}(x|t)\big|_{(t_1,t_2,t_3,t_4,\dots)=(t_1,-z_1,t_2,-z_2,\dots)}.
\]
By observing $Q_{\rho}(x|t)=\Pf[{}_0\vt_{2n-1}^{2n-2}\,{}_0\vt_{2n-3}^{2n-4} \cdots {}_0\vt_{3}^{2}\,{}_0\vt_{1}^{0}]$ and 
\[
{}_0\vt_{r}^{2i}(x, z|t)\big|_{(t_1,t_2,t_3,t_4,\dots)=(t_1,-z_1,t_2,-z_2,\dots)} = {}_i\vt_{r}^{i}(x,z|t),
\]
we have the following result. 
\begin{theorem}[Ikeda-Mihalcea-Naruse]\label{topDSP}
Let $w_0\in W_n$ be the longest element. Then
\[
\frakC_{w_0} = \Pf\left[{}_{n-1}\vt_{2n-1}^{(n-1)}{}_{n-2}\vt_{2n-3}^{(n-2)} \cdots {}_1\vt_{3}^{(1)}{}_0\vt_{1}^{(0)}\right].
\] 
\end{theorem}
The $\vt$-functions in the Pfaffian above can be regarded as equivariantly modified special Schubert classes from various isotropic Grassmannians.

To apply $\delta_i$'s systematically to the top function $\frakC_{w_0}$, we need Proposition \ref{square-relation0general}, Proposition \ref{delta-i-on-Pf general}, and the following lemma. 
\begin{lemma}\label{k-l updown}
If $l\geq 0$ and $k>0$, \
\begin{eqnarray*}
{}_k\vt_{r}^{(l)} &=& {}_{k-1}\vt_{r}^{(l+1)} + (t_{l+1}+z_k) \cdot {}_{k-1}\vt_{r-1}^{(l)}.
\end{eqnarray*}
Moreover, if $l_p\geq 0$ and $k_p >0$, then
\begin{eqnarray*}
&&\Pf[{}_{k_1}\!\vt_{r_1}^{(l_1)}\cdots {}_{k_p}\!\vt_{r_p}^{(l_p)} \cdots {}_{k_m}\!\vt_{r_m}^{(l_m)}]\\
&=& 					\Pf[{}_{k_1}\!\vt_{r_1}^{(l_1)} \cdots {}_{k_p-1}\vt_{r_p}^{(l_p+1)} \cdots {}_{k_m}\!\vt_{r_m}^{(l_m)}]
+ (t_{l_p+1} + z_{k_p})	\Pf[{}_{k_1}\!\vt_{r_1}^{(l_1)} \cdots {}_{k_p-1}\vt_{r_p-1}^{(l_p)} \cdots {}_{k_m}\!\vt_{r_m}^{(l_m)}].
\end{eqnarray*}
\end{lemma}
\begin{proof}
The equation ${}_kf_{l+1}(u) = {}_kf_{l}(u)\cdot (1-t_{l+1}u) = {}_{k-1}f_{l+1}(u)\cdot (1+z_ku)$ implies
\[
{}_k\vt_{r}^{(l+1)} = {}_k\vt_{r}^{(l)} - t_{l+1}\cdot {}_k\vt_{r-1}^{(l)} = {}_{k-1}\vt_{r}^{(l+1)} + z_k \cdot {}_{k-1}\vt_{r-1}^{(l+1)}.
\]
The first claim follows from using this formula twice. Indeed, we have
\begin{eqnarray*}
&& {}_k\vt_{r}^{(l)} - t_{l+1}\cdot ({}_{k-1}\vt_{r-1}^{(l)} + z_{k}\cdot{}_{k-1}\vt_{r-2}^{(l)})  = {}_{k-1}\vt_{r}^{(l+1)} + z_k \cdot ({}_{k-1}\vt_{r-1}^{(l)} - t_{l+1}\cdot{}_{k-1}\vt_{r-2}^{(l)} ).
\end{eqnarray*}
This gives the desired equality. The second claim follows from the multilinearity of Pfaffian.
\end{proof}

We obtain the next lemma by the consecutive application of Proposition \ref{delta-i-on-Pf general} (b), Lemma \ref{k-l updown}, and Proposition \ref{square-relation0general} in this order.
\begin{lemma}\label{updown trick}
Suppose $l_{p+1}+1=l_p=i>0$ and $l_q\not\in\{\pm i\}$ for all $q\not=p$. If $(k_{p+1}, r_{p+1})=(k_p-1, r_p-2)$, and $r_{p}>l_{p} + k_{p}$, then
\begin{eqnarray*}
\delta_i \Pf[{}_{k_1}\!\vt_{r_1}^{(l_1)}\cdots {}_{k_p}\!\vt_{r_p}^{(i)}{}_{k_{p+1}}\!\vt_{r_{p+1}}^{(i-1)} \cdots {}_{k_m}\!\vt_{r_m}^{(l_m)}]
&=& \Pf[{}_{k_1}\!\vt_{r_1}^{(l_1)}\cdots {}_{k_p-1}\vt_{r_p-1}^{(i)}{}_{k_{p+1}}\!\vt_{r_{p+1}}^{(i-1)} \cdots {}_{k_m}\!\vt_{r_m}^{(l_m)}].
\end{eqnarray*}
\end{lemma}
\begin{proof}
We have
\begin{eqnarray*}
&&\delta_i \Pf[{}_{k_1}\!\vt_{r_1}^{(l_1)}\cdots {}_{k_p}\!\vt_{r_p}^{(i)}{}_{k_{p+1}}\!\vt_{r_{p+1}}^{(i-1)} \cdots {}_{k_m}\!\vt_{r_m}^{(l_m)}]\\
&=& \Pf[{}_{k_1}\!\vt_{r_1}^{(l_1)}\cdots {}_{k_p-1}\vt_{r_p-1}^{(i)}{}_{k_{p+1}}\!\vt_{r_{p+1}}^{(i-1)} \cdots {}_{k_m}\!\vt_{r_m}^{(l_m)}] + (t_{i+1} +z_{k_p}) \Pf[{}_{k_1}\!\vt_{r_1}^{(l_1)}\cdots {}_{k_p-1}\vt_{r_p-2}^{(i-1)}{}_{k_{p+1}}\!\vt_{r_{p+1}}^{(i-1)} \cdots {}_{k_m}\!\vt_{r_m}^{(l_m)}]\\
&=& \Pf[{}_{k_1}\!\vt_{r_1}^{(l_1)}\cdots {}_{k_p-1}\vt_{r_p-1}^{(i)}{}_{k_{p+1}}\!\vt_{r_{p+1}}^{(i-1)} \cdots {}_{k_m}\!\vt_{r_m}^{(l_m)}].
\end{eqnarray*}
\end{proof}
\begin{example}\label{exm0} The following computation demonstrates the content of the above lemma.
\begin{eqnarray*}
\delta_3\Pf[{}_3\vt_{7}^{(3)}{}_2\vt_{5}^{(2)}]
&=&\Pf[{}_3\vt_{6}^{(2)}{}_2\vt_{5}^{(2)}]\\
&=&\Pf[{}_2\vt_{6}^{(3)}{}_2\vt_{5}^{(2)}]+ (t_3+z_3) \Pf[{}_2\vt_{5}^{(2)}{}_2\vt_{5}^{(2)}]\\
&=&\Pf[{}_2\vt_{6}^{(3)}{}_2\vt_{5}^{(2)}],
\end{eqnarray*}
where the first equality follows from Proposition \ref{delta-i-on-Pf general} (b), the second equality follows from Lemma \ref{k-l updown}, the third is by Proposition \ref{square-relation0general}.
\end{example}
\begin{example}\label{longest W^0}
The above lemma allows us to find the Pfaffian formula for the longest element $w$ of $W_n^{(0)}$. By Corollary \ref{red ex} where $J=\{0\}$, we have
\[
\frakC_w=(\delta_{n-1})(\delta_{n-2}\delta_{n-1}) \cdots(\delta_2\delta_3\cdots \delta_{n-2}\delta_{n-1})(\delta_{1}\delta_2 \cdots \delta_{n-2}\delta_{n-1})\frakC_{w_0}.
\]
Observe that Lemma \ref{updown trick} applies to the action of each $\delta_i$, and we have the known formula
\[
\frakC_w = \Pf[ {}_0\vt_n^{(n-1)}{}_0\vt_{n-1}^{(n-2)} \cdots {}_0\vt_2^{(1)}{}_0\vt_1^{(0)}].
\]
For example, we compute $\frakC_{\bar4\bar3\bar2\bar1}$ (the longest element in $W_4^{(0)}$) from the top function $\frakC_{\bar1\bar2\bar3\bar4}$. 
\begin{eqnarray*}
\frakC_{\bar1\bar2\bar3\bar4}&=&\Pf[{}_3\vt_{7}^{(3)}{}_2\vt_{5}^{(2)}{}_1\vt_{3}^{(1)}{}_0\vt_{1}^{(0)}]\\
\stackrel{\delta_3}{\longrightarrow} \ \ 
\frakC_{\bar1\bar2\bar4\bar3}&=&\Pf[{}_2\vt_{6}^{(3)}{}_2\vt_{5}^{(2)}{}_1\vt_{3}^{(1)}{}_0\vt_{1}^{(0)}]\\
\stackrel{\delta_2}{\longrightarrow} \ \ 
\frakC_{\bar1\bar3\bar4\bar2}&=&\Pf[{}_2\vt_{6}^{(3)}{}_1\vt_{4}^{(2)}{}_1\vt_{3}^{(1)}{}_0\vt_{1}^{(0)}]\\
\stackrel{\delta_1}{\longrightarrow} \ \ 
\frakC_{\bar2\bar3\bar4\bar1}&=&\Pf[{}_2\vt_{6}^{(3)}{}_1\vt_{4}^{(2)}{}_0\vt_{2}^{(1)}{}_0\vt_{1}^{(0)}]\\
\stackrel{\delta_3}{\longrightarrow} \ \ 
\frakC_{\bar2\bar4\bar3\bar1}&=&\Pf[{}_1\vt_{5}^{(3)}{}_1\vt_{4}^{(2)}{}_0\vt_{2}^{(1)}{}_0\vt_{1}^{(0)}]\\
\stackrel{\delta_2}{\longrightarrow} \ \ 
\frakC_{\bar3\bar4\bar2\bar1}&=&\Pf[{}_1\vt_{5}^{(3)}{}_0\vt_{3}^{(2)}{}_0\vt_{2}^{(1)}{}_0\vt_{1}^{(0)}]\\
\stackrel{\delta_3}{\longrightarrow} \ \ 
\frakC_{\bar4\bar3\bar2\bar1}&=&\Pf[{}_0\vt_{4}^{(3)}{}_0\vt_{3}^{(2)}{}_0\vt_{2}^{(1)}{}_0\vt_{1}^{(0)}].
\end{eqnarray*}
\end{example}

In general, we can find the Pfaffian formula for the longest element of $W_n^{J}$. First we introduce the following notation for simplicity.
\begin{definition}
For each $B=(\kappa_1,\dots, \kappa_b; r_1,\dots, r_b; l_1,\dots,l_b) \in \ZZ^{3b}$, we formally denote
\[
\vt_{B}:={}_{\kappa_b}\vt_{r_b}^{(l_b)} \cdots {}_{\kappa_1}\vt_{r_1}^{(l_1)}.
\]
\end{definition}
By choosing any integer sequence $0\leq k_1<k_2<\cdots < k_p<n=:k_{p+1}$, we can write 
\[
\frakC_{w_0} =\Pf[\vt_{B_p} \cdots  \vt_{B_1}],
\]
where, for each $i=1,\dots,p$, 
\[
B_i:= (k_i,k_i+1, \dots,k_{i+1}-1\ ; \  2k_i+1,2k_i+3,\dots, 2k_{i+1}-1\ ; \  k_i,k_i+1, \dots,k_{i+1}-1) \in \ZZ^{3(k_{i+1}-k_i)}.
\]
\begin{theorem}\label{top block k-Pf}
Let $w$ be the longest element in $W_n^{J}$ where $J=\{k_1<\cdots<k_p\}$. Let $k_{p+1}:=n$. Then 
\begin{equation}\label{block}
\frakC_{w} =\Pf[\vt_{\widetilde{B}_p} \cdots  \vt_{\widetilde{B}_2}\vt_{\widetilde{B}_1}],
\end{equation}
where for each $i=1,\dots,p$, 
\[
\widetilde{B}_i:=(k_i, \dots,k_i\ ;\   2k_i+1,2k_i+2,\dots, k_i + k_{i+1}\ ;\  k_i, k_i+1, \dots, k_{i+1}-1).
\]
\end{theorem}
\begin{proof}
Suppose $k_1=0$. By Corollary \ref{red ex}, we have
\[
\frakC_w= \delta_{v_1^c} \cdots \delta_{v_p^c} \frakC_{w_0} ,
\]
where 
\[
\delta_{v_i^c} = \delta_{k_{i+1}-1}  (\delta_{k_{i+1}-2}  \delta_{k_{i+1}-1})  \cdots (\delta_{k_i+2}\cdots \delta_{k_{i+1}-2}\delta_{k_{i+1}-1})(\delta_{k_i+1}\delta_{k_i+2}\cdots \delta_{k_{i+1}-2}\delta_{k_{i+1}-1}).
\]
Lemma \ref{updown trick} applies to the action of each $\delta_i$ above, and we obtain the claim. 

Instead of considering the case when $k_1\not=0$, we can assume $J=\{k_2<\cdots <k_p\}$ where $k_2\not=0$ and $w'$ is the longest element in $W_n^{J}$. We need to show that $\frakC_{w'} =\Pf[\vt_{\widetilde{B}_p} \cdots  \vt_{\widetilde{B}_2}]$. By Corollary \ref{red ex}, 
\begin{equation}\label{ldd on w'}
\frakC_{w'} = \delta_{j_s}\cdots\delta_{j_1} \frakC_{w},
\end{equation}
where $(j_1,\dots, j_s)$ is a reduced word for the longest element of $W_{k_2}^{(0)}$ and $w$ is the longest element in $W_n^{J \cup \{k_1=0\}}$. From the previous case, we have
\[
\frakC_{w}=\Pf[\vt_{\widetilde{B}_p} \cdots  \vt_{\widetilde{B}_2}\vt_{\widetilde{B}_1}].
\]
The divided difference operators  in (\ref{ldd on w'}) act only on the part $\vt_{\widetilde{B}_1}$ through Proposition \ref{delta-i-on-Pf general}. Together with Remark \ref{rem: n-ind theta general}, we obtain $\frakC_{w'} =\Pf[\vt_{\widetilde{B}_p} \cdots  \vt_{\widetilde{B}_2}]$.
\end{proof}
\begin{example}\label{exm3}  In the above proof, we operate $\delta_i$'s as in Example \ref{longest W^0} to each $\vt_{B_i}$ of the top function $\frakC_{w_0}$ to obtain the Pfaffian formula for the longest element of $W_n^{J}$ for  $0 \in J$. For example, we compute $\frakC_{\bar3\bar2\bar1|\bar6\bar5\bar4|\bar7\bar8\bar9}$ (the longest element in $W_9^{\{0,3,6,7\}}$) from the top function $\frakC_{\bar1\bar2\bar3\bar4|\bar5\bar6\bar7\bar8\bar9}$. 
\begin{eqnarray*}
\frakC_{\bar1\bar2\bar3\bar4\bar5\bar6\bar7\bar8\bar9}&=&\Pf[{}_8\vt_{17}^{(8)}{}_7\vt_{15}^{(7)}{}_6\vt_{13}^{(6)}{}_5\vt_{11}^{(5)}{}_4\vt_{9}^{(4)}{}_3\vt_{7}^{(3)}{}_2\vt_{5}^{(2)}{}_1\vt_{3}^{(1)}{}_0\vt_{1}^{(0)}]\\
\xrightarrow{\delta_2\delta_1\delta_2} \ \ \ 
\frakC_{\bar3\bar2\bar1|\bar4\bar5\bar6\bar7\bar8\bar9}&=&\Pf[{}_8\vt_{17}^{(8)}{}_7\vt_{15}^{(7)}{}_6\vt_{13}^{(6)}{}_5\vt_{11}^{(5)}{}_4\vt_{9}^{(4)}{}_3\vt_{7}^{(3)}|	{}_0\vt_{3}^{(2)}{}_0\vt_{2}^{(1)}{}_0\vt_{1}^{(0)}]\\
\xrightarrow{\delta_5\delta_4\delta_5} \ \ \ 
\frakC_{\bar3\bar2\bar1|\bar6\bar5\bar4|\bar7\bar8\bar9}&=&\Pf[{}_8\vt_{17}^{(8)}{}_7\vt_{15}^{(7)}		{}_6\vt_{13}^{(6)}|	{}_3\vt_{9}^{(5)}{}_3\vt_{8}^{(4)}{}_3\vt_{7}^{(3)}	|	{}_0\vt_{3}^{(2)}{}_0\vt_{2}^{(1)}{}_0\vt_{1}^{(0)}]\\
\stackrel{\delta_8}{\longrightarrow} \ \ \ 
\frakC_{\bar3\bar2\bar1|\bar6\bar5\bar4|\bar7|\bar9\bar8}&=&\Pf[{}_7\vt_{16}^{(8)}{}_7\vt_{15}^{(7)}	|	{}_6\vt_{13}^{(6)}|		{}_3\vt_{9}^{(5)}{}_3\vt_{8}^{(4)}{}_3\vt_{7}^{(3)}	|	{}_0\vt_{3}^{(2)}{}_0\vt_{2}^{(1)}{}_0\vt_{1}^{(0)}].
\end{eqnarray*}
From this we can obtain the formula for the longest element of $W_n^{J\backslash\{0\}}$, by further applying $\delta_i$'s, using Proposition \ref{delta-i-on-Pf general} and Remark \ref{rem: n-ind theta general}. For example, we obtain $\frakC_{123|\bar6\bar5\bar4|\bar7|\bar9\bar8}$ (the longest element in $W_9^{\{3,6,7\}}$) from $\frakC_{\bar3\bar2\bar1|\bar6\bar5\bar4|\bar7|\bar9\bar8}$:
\begin{eqnarray*}
\frakC_{\bar3\bar2\bar1|\bar6\bar5\bar4|\bar7|\bar9\bar8}&=&\Pf[{}_7\vt_{16}^{(8)}{}_7\vt_{15}^{(7)}	|	{}_6\vt_{13}^{(6)}|		{}_3\vt_{9}^{(5)}{}_3\vt_{8}^{(4)}{}_3\vt_{7}^{(3)}	|	{}_0\vt_{3}^{(2)}{}_0\vt_{2}^{(1)}{}_0\vt_{1}^{(0)}]\\
\xrightarrow{\delta_0\delta_1\delta_0\delta_2\delta_1\delta_0}  \ \ \ 
\frakC_{123|\bar6\bar5\bar4|\bar7|\bar9\bar8}&=&\Pf[{}_7\vt_{16}^{(8)}{}_7\vt_{15}^{(7)}	|	{}_6\vt_{13}^{(6)}|		{}_3\vt_{9}^{(5)}{}_3\vt_{8}^{(4)}{}_3\vt_{7}^{(3)}	|	{}_0\vt_{0}^{(-1)}{}_0\vt_{0}^{(-2)}{}_0\vt_{0}^{(-3)}]\\
&=&\Pf[{}_7\vt_{16}^{(8)}{}_7\vt_{15}^{(7)}	|	{}_6\vt_{13}^{(6)}|		{}_3\vt_{9}^{(5)}{}_3\vt_{8}^{(4)}{}_3\vt_{7}^{(3)}	].
\end{eqnarray*}
The last equality follows from Remark \ref{rem: n-ind theta general}.
\end{example}
\subsection{Pseudo $k$-Grassmannian elements}
\begin{definition}\label{def: pseudo Gr}
Let $k$ be a non-negative integer and $J=\{k=k_1<\cdots <k_p\}$.  A \emph{pseudo $k$-Grassmannian  element} is a signed permutation $w$ in $W_n^{J}$ such that its one line notation is of the form
\[
w=w_1w_2\cdots w_{k_2}\,|\, \overline{k_3}\,\overline{k_3-1}\cdots \overline{k_2+1}\,|\,\cdots\,|\,\overline{n}\,\overline{n-1}\,\cdots\,\overline{k_p+1}.
\]
\end{definition}
Note that in this case the first $k_2$ letters $w_1,\cdots, w_{k_2}$ form a $k$-Grassmannian permutation. Also if we replace them by  $\overline{k_2}\,\overline{k_2-1}\,\dots\, \overline{k+1}$ ($k=0$) or $12\cdots k | \overline{k_2}\,\overline{k_2-1}\,\dots\, \overline{k+1}$ $(k\not=0)$, then we obtain the longest element in $W_n^{J}$. For example, $(\bar2 1|\bar4\bar3)$ is pseudo $0$-Grassmannian in $W_4^{\{0,2\}}$ and $(2|\bar3\bar1|\bar4)$ is pseudo $1$-Grassmannian in $W_4^{\{1,3\}}$. On the other hand, $(\bar31|\bar4\bar2)$ and $(3|\bar4\bar1|\bar2)$ are in $W_4^{\{0,2\}}$ and $W_4^{\{1,3\}}$ respectively but they are not pseudo $k$-Grassmannian.

\vspace{0.2in}

Now the following theorem follows from  the same argument in proof of Theorem \ref{MainProposition}, just by operating on the part $\vt_{\widetilde{B}_1}$ in Theorem \ref{top block k-Pf} (use Proposition \ref{delta-i-on-Pf general}). 
\begin{theorem}\label{NEW THEOREM INTRO}
Let $w$ be a pseudo $k$-Grassmannian element in $W_n^{J}$ where $J=\{k_1<\cdots < k_p\}$. Let $k=k_1$ and $m=k_2$. Let $\tilde{B}_i$ be as in Theorem \ref{top block k-Pf}.
We have
\begin{equation}\label{new formula}
\frakC_{w} = \sum_{I\subset {D}(\lambda)} \Pf\left[ \vt_{\widetilde{B}_p} \cdots  \vt_{\widetilde{B}_2} {}_k\vt_{\lambda_1+a^I_1}^{(\chi_1)} \cdots  {}_k\vt_{\lambda_{m-k}+a^I_{m-k}}^{(\chi_{m-k})}\right],
\end{equation}
where $\lambda$ is the $k$-strict partition in $\mathcal{P}_m^{(k)}$ associated to the signed permutation $(w_1\cdots w_m)$ consisting of the first $m$ letters of $w$, $\chi$ is the associated characteristic index, $I$ runs over all subsets  of $D(\lambda)$, and $a_s^I= \#\{j\;|\;(s,j)\in I\} -\#\{i\;|\;(i,s)\in I\}.$
\end{theorem}

\subsection{All $n=3$ signed permutations}
We can go further from Theorem \ref{NEW THEOREM INTRO}. For example, let's look at
\[
\frakC_{\bar2\bar1|\bar4\bar3}=\Pf[{}_2\vt_{6}^{(3)}{}_2\vt_{5}^{(2)}|{}_0\vt_{2}^{(1)}{}_0\vt_{1}^{(0)}]
\ \ \stackrel{\delta_2}{\longrightarrow}\ \ 
\frakC_{\bar3\bar1|\bar4\bar2}=\Pf[{}_2\vt_{6}^{(3)}{}_2\vt_{4}^{(1)}  |  {}_0\vt_{2}^{(1)}{}_0\vt_{1}^{(0)}].
\]
The upper indices collide at $1$, and we don't have a systematic technique to apply $\delta_1$ without breaking the Pfaffian form. However this shows that the Pfaffian sum formula exists beyond the pseudo $k$-Grassmannian permutations, since $\bar3\bar1|\bar4\bar2$ is not a pseudo $k$-Grassmannian.  We can further apply $\delta_3$ and $\delta_0$ to $\frakC_{\bar3\bar1|\bar4\bar2}$, keeping the Pfaffian form. For example, we can easily compute all 48 signed permutation of $W_3$.  Among them, there are 16 permutations that are not pseudo $k$-Grassmannian permutations.  It turns out that all of them are written as (sum of) Pfaffians, \emph{except} 
$\frakC_{\bar32\bar1},  \frakC_{\bar321},  \frakC_{\bar 231}$ and $\frakC_{\bar132}$.
\newpage
\section{Example: $(n,k)=(5,2), (5,3)$}\label{sec:exm}
{\bf Case $(n,k)=(5,3)$}

{\tiny
\[
123|45
\]
\[
(00) \ \ \Det[\vt_0^{-4}\vt_0^{-5}]
\]

\[
124|35
\]
\[
(10) \ \ \Det[\vt_1^{-3}\vt_0^{-5}]
\]

\[
134|25\ \ \ \ 	\ \ \ \  \ \ \ \  \ \ \ \ 	125|34
\]
\[
(20)\ \ \Det[\vt_2^{-2}\vt_0^{-5}]\ \ \ \ \ \ 	(11)\ \  \Det[\vt_1^{-3}\vt_1^{-4}]
\]

\[
234|15\ \ 	\	\ \ \ \ \  \ \ \ \  \ \ \ \ \ \ 	135|24
\]
\[
(30)\ \ \Det[\vt_3^{-1}\vt_0^{-5}]\ \ \ \ 	\ \ 	(21) \ \  \Det[\vt_2^{-2}\vt_1^{-4}]
\]

\[
234|\bar15 \ \ 	\ \ \ \ \ \ \ \ 	 \ \ \ \  \ \ \ \ 235|14\ \ 		\ \ \ \ \ \  \ \ \ \  \ \ \ \ \ \ 	145|23
\]
\[
(40)\ \ \Det[\vt_4^0\vt_0^{-5}]\ \ \ \ \ 	(31)\ \ \Det[\vt_3^{-1}\vt_1^{-4}]\ \ \ \ \ \ 	(22) \ \  \Det[\vt_2^{-2}\vt_2^{-3}]
\]

\[
134|\bar25\ \  \ \ \ \  \ \ \ \ \ \ \ 	\ \ \ \ \ \ \  	235|\bar14\ \ 	\ \ \  \ \ \ \  \ \ \ \ \ \ 	\ \ \ \ \ 245|13
\]
\[
(50)\ \ \Det[\vt_5^1\vt_0^{-5}]\ \ \ \ \ \ 	(41)\ \ \Det[\vt_4^0\vt_1^{-4}]\ \ \ \ \ \  	(32)\ \  \Det[\vt_3^{-1}\vt_2^{-3}]
\]

\[
124|\bar35\ \ 	\ \ 	 \ \ \ \  \ \ \ \  \ \ \ \  \ \ \ \ 135|\bar24\ \ 	\ \  \ \ \ \  \ \ \ \  \ \ \ \  \ \ \ \ 	245|\bar13\ \  \ \ \ \  \ \ \ \  \ \ \ \ 	\ \  \ \ \ \ 	345|12
\]
\[
(60)\ \ \Det[\vt_6^2\vt_0^{-5}]\ \ \ \ \   \ \ \ 	(51)\ \ \Det[\vt_5^1\vt_1^{-4}]\ \ \ \ \  \ \ \ 	(42)\ \ \Det[\vt_4^0\vt_2^{-3}]\ \ \	\ \ \ \ \ \ 	(33) \ \ \Det[\vt_3^{-1}\vt_3^{-2}]
\]

\[
123|\bar45\ \ 	\ \  \ \ \ \ 	 \ \ \ \ \ \ \ \   \ \ \ \ 125|\bar34\ \ 	\ \  \ \ \ \ 	 \ \ \ \  \ \ \ \  \ \ \ \ 145|\bar23\ \ 	\ \  \ \ \ \ \ \ \ \  \ \ \ \   \ \ \ \ 	345|\bar12
\]
\[
(70) \ \ \Det[\vt_7^3\vt_0^{-5}]\ \ \ \ \ \ 	\ \ 	(61) \ \ \Det[\vt_6^2\vt_1^{-4}]\ \ \ \ \ \ 		(52)\ \ \Det[\vt_5^1\vt_2^{-3}]\ \ \ \ \ \ 	\ \ 	(43)\ \ \Det[\vt_4^0\vt_3^{-2}]
\]

\[
123|\bar54\ \ 	\ \  \ \ \ \ 	 \ \ \ \  \ \ \ \  \ \ \ \ 125|\bar43\ \ \ \  \ \ \ \ 	 \ \ \ \  \ \ \ \  \ \ \ \  \ \ \ \ 	145|\bar32\ \ \ \ 	 \ \ \ \  \ \ \ \  \ \ \ \  \ \ \ \ 	345|\bar21
\]
\[
(80) \ \ \Pf[\vt_8^4\vt_0^{-4}] \ \ \ \ \ \ 	\ \ 	(71) \ \ \Pf[\vt_7^3\vt_1^{-3}]\ \ \ \ \ \ 	\ \ 	(62)\ \ \Pf[\vt_6^2\vt_2^{-2}]\ \ \ \ \ \ 	\ \ 	(53)\ \ \Pf[\vt_5^1\vt_3^{-1}]
\]

\[
124|\bar53\ \ 	\ \  \ \ \ \ 	 \ \ \ \  \ \ \ \  \ \ \ \ 135|\bar42\ \ 	\ \ 	 \ \ \ \  \ \ \ \  \ \ \ \  \ \ \ \ 245|\bar31\ \  \ \ \ \ 	\ \  \ \ \ \  \ \ \ \  \ \ \ \ 	345|\bar2\bar1
\]
\[
(81) \ \ \Pf[\vt_8^4\vt_1^{-3}]\ \ \ \ \ \ 	\ \ 	(72)\ \ \Pf[\vt_7^3\vt_2^{-2}]\ \ \ \ \ \ 	\ \ 	(63)\ \ \Pf[\vt_6^2\vt_3^{-1}]\ \ \ \ \ \ 	\ \ 	(54)\ \ \Pf[\vt_5^1\vt_4^0]
\]

\[
134|\bar52\ \ 	\ \  \ \ \ \ 	 \ \ \ \  \ \ \ \ 235|\bar41\ \ 	\ \  \ \ \ \  \ \ \ \  \ \ \ \ 	245|\bar3\bar1
\]
\[
(82) \ \ \Pf[\vt_8^4\vt_2^{-2}]\ \ \ \ \ \ 	\ \ 	(73)\ \ \Pf[\vt_7^3\vt_3^{-1}]\ \ \ \ \ \ 	\ \ 	(64)\ \ \Pf[\vt_6^2\vt_4^0]
\]

\[
234|\bar51 \ \ 	\ \  \ \ \ \  \ \ \ \  \ \ \ \ 	235|\bar4\bar1\ \ \ \  \ \ \ \ 	 \ \ \ \  \ \ \ \ 145|\bar3\bar2
\]
\[
(83)\ \ \Pf[\vt_8^4\vt_3^{-1}]\ \ \ \ \ \ 	\ \ 	(74)\ \ \Pf[\vt_7^3\vt_4^0]\ \ \ \ \ \ \ \ 		(65)\ \ \Pf[\vt_6^2\vt_5^1]
\]

\[
234|\bar5\bar1\ \ \ \ 	 \ \ \ \  \ \ \ \  \ \ \ \ 135|\bar4\bar2
\]
\[
(84)\ \ \Pf[\vt_8^4\vt_4^0]\ \ \ \ \ \ 	\ \ 	(75)	\ \ \Pf[\vt_7^3\vt_5^1]
\]

\[
134|\bar5\bar2 \ \  \ \ \ \  \ \ \ \  \ \ \ \ 	\ \ 125|\bar4\bar3
\]
\[
(85)	\ \ \Pf[\vt_8^4\vt_5^1] \ \ \ \ 	\ \ 	\ \ 	(76)	\ \ \Pf[\vt_7^3\vt_6^2] 
\]

\[
124|\bar5\bar3 
\]
\[
(86)  \ \ \Pf[\vt_8^4\vt_6^2] 
\]

\[
123|\bar5\bar4
\]
\[
(87)	\ \ \Pf[\vt_8^4\vt_7^3]
\]
}

\newpage
{\small
{\bf Case $(n,k)=(5,2)$.}

{\tiny
\[
12|345
\]
\[
(000)
\]
\[
\Det[\vt_0^{-3}\vt_0^{-4}\vt_0^{-5}]
\]

\[
13|245
\]
\[
(100)
\]
\[
\Det[\vt_1^{-2}\vt_0^{-4}\vt_0^{-5}]
\]

\[
23|145\ \ \ \ \ \ \ \ \ \ \ \ \   14|235
\]
\[
(200)\ \ \ \ \ \ \ \ \ \ \ \ \ \ \   (110)
\]
\[
\Det[\vt_2^{-1}\vt_0^{-4}\vt_0^{-5}]\ \ \ \ \ \ 
\Det[\vt_1^{-2}\vt_1^{-3}\vt_0^{-5}]
\]

\[
23|\bar145\ \ \ \ \ \ \ \  \ \ \ \ \  \ \ \ \ \  24|135\ \ \ \ \ \ \ \ \ \ \ \ \  \ \ \ \ \   15|234 
\]
\[
(300)\ \ \ \ \  \ \ \ \ \     \ \ \ \ \  \ \ \ \ \    (210)\ \ \ \ \ \ \ \ \ \  \ \ \ \ \  \ \ \ \ \  (111)
\]
\[
\Det[\vt_3^0\vt_0^{-4}\vt_0^{-5}]\ \ \ \ \ \ 
\Pf[\vt_2^{-1}\vt_1^{-3}\vt_0^{-5}]\ \ \ \ \ \ 
\Det[\vt_1^{-2}\vt_1^{-3}\vt_1^{-4}]
\]

\[
13|\bar245\ \ \ \ \ \ \ \  \ \ \ \ \  \ \  24|\bar135\ \ \ \ \ \ \ \  \ \ \ \ \ \ \  \ \ \ \ \  34|125\ \ \  \ \ \ \ \  \ \ \ \ \ \ \ \ \  \ \   25|134
\]
\[
(400)\ \ \ \ \   \ \ \ \ \  \ \ \ \ \   \ \   (310)\ \ \ \ \   \ \ \ \ \  \ \ \ \ \   \ \ \ \  \ \ \   (220)\ \ \ \ \ \ \ \ \ \  \ \ \ \   \ \ \ \ \ \ \   (211)
\]
\[
\Pf[\vt_4^1\vt_0^{-4}\vt_0^{-5}]\ \ \ \ \ \ 
\Det[\vt_3^0\vt_1^{-3}\vt_0^{-5}]\ \ \ \ \ \ 
\Det[\vt_2^{-1}\vt_2^{-2}\vt_0^{-5}] \ \ \ \ \ \ 
\Det[\vt_2^{-1}\vt_1^{-3}\vt_1^{-4}]
\]

\[
\ \ 12|\bar345\ \ \ \ \ \ \ \  \ \ \ \ \  \ \ \   14|\bar235\ \ \ \ \ \ \ \  \ \ \ \ \  \ \ \ \  34|\bar125\ \ \ \ \ \ \ \  \ \ \ \ \  \ \ \ \ \    \ \ 25|\bar134\ \ \ \ \ \ \ \  \ \ \ \ \  \ \ \ \ \ 35|124\ \ \ \ \  
\]
\[
\ \ (500)\ \ \ \ \  \ \ \ \ \  \ \ \ \ \  \  \ \    (410)\ \ \ \ \  \ \ \ \ \  \ \ \ \ \  \ \ \ \     (320)\ \ \ \ \  \ \ \ \ \  \ \ \ \ \  \ \ \ \ \      \ \ (311)\ \ \ \ \ \ \ \ \ \  \ \ \ \ \  \ \ \ \ \     (221) \ \ \ \ \  
\]
\[
\Det[\vt_5^2\vt_0^{-4}\vt_0^{-5}]\ \ \ \ \ \ 
{\Det[\vt_4^1\vt_1^{-2}\vt_0^{-5}]}\ \ \ \ \ \ 
{\Det[\vt_3^0\vt_2^{-2}\vt_0^{-5}]}\ \ \ \ \ \ \ 
{\Pf[\vt_3^0\vt_1^{-3}\vt_1^{-4}]\atop{+\Pf[\vt_4^0\vt_0^{-3}\vt_1^{-4}] \atop{+ \Det[\vt_4^0\vt_1^{-3}\vt_0^{-4}]\atop{+\Det[\vt_5^0\vt_0^{-3}\vt_0^{-4}]}}}}\ \ \ \ \ \ 
\Det[\vt_2^{-1}\vt_2^{-2}\vt_1^{-4}]
\]

\[
12|\bar435\ \ \ \ \ \ \ \  \ \ \ \ \  \  \   \ 14|\bar325\ \ \ \ \ \ \ \  \ \ \ \ \ \ \  34|\bar215\ \ \ \ \ \ \ \   \ \ \ \ \  \ \ \ \ \    15|\bar234\ \ \ \ \ \ \ \  \ \ \ \ \ \    \ \ \ \ \     35|\bar124\ \ \ \ \  \ \ \ \ \ \ \ \ \  \ \ \ \ \ \       45|123 
\]
\[
(600)\ \ \ \ \  \ \ \ \ \  \ \ \ \ \  \       \  (510)\ \ \ \ \  \ \ \ \ \  \ \ \ \ \   \         (420)\ \ \ \ \  \ \ \ \ \   \ \ \ \ \  \ \ \ \ \      (411)\ \ \ \ \  \ \ \ \ \  \ \ \ \ \    \ \ \ \ \        (321)\ \ \ \ \   \ \ \ \ \   \ \ \ \ \  \ \ \ \ \   \     (222)   
\]
\[
\Det[\vt_6^3\vt_0^{-3}\vt_0^{-5}]\ \ \ \ \ \ 
\Pf[\vt_5^2\vt_1^{-2}\vt_0^{-5}]\ \ \ \ \ \ 
\Pf[\vt_4^1\vt_2^{-1}\vt_0^{-5}]\ \ \ \ \ \ 
{\Pf[\vt_4^1\vt_1^{-3}\vt_1^{-4}]\atop{+\Pf[\vt_5^1\vt_0^{-3}\vt_1^{-4}] \atop{+ \Det[\vt_5^1\vt_1^{-3}\vt_0^{-4}]\atop{+\Det[\vt_6^1\vt_0^{-3}\vt_0^{-4}]}}}}\ \ \ \ \ \ 
{\Pf[\vt_3^0\vt_2^{-2}\vt_1^{-4}]\atop{+\Pf[\vt_4^0\vt_1^{-2}\vt_1^{-4}]\atop{+ \Det[\vt_4^0\vt_2^{-2}\vt_0^{-4}]\atop{+\Det[\vt_5^0\vt_1^{-2}\vt_0^{-4}]}}}}\ \ \ \ \ \ \ 
\Det[\vt_2^{-1}\vt_2^{-2}\vt_2^{-3}]
\]

\[
\ \ \ \ \ \ 12|\bar534\ \ \ \ \ \ \ \  \ \ \ \ \ \ \ \  \ \ \ \ \ \ 13|\bar425\ \ \ \ \ \ \ \  \ \ \ \ \  \ \ \ \ \  \ \ \ \ \ 24|\bar315\ \ \ \ \ \ \ \  \ \ \ \ \ \ \ \ \ \ \ \ \ \   34|\bar2\bar15\ \ \ \ \ \ \ \  \ \ \ \ \  \ \ \ \ \ \ \ \ \ \  \  15|\bar324\ \ \ \ \  \ \ \ \ \ \ \ \  \ \ \ \ \  \ \ \ \ \  \  35|\bar214\ \ \ \ \ \ \ \  \ \ \ \ \  \ \ \ \ \  \ \ \ \ \  \ \ \ \ \  45|\bar123\ \ \ \ \ \  \ \ \ \ \  \ \ \ \ 
\]
\[
\ \ \ \ \ \ (700)\ \ \ \ \  \ \ \ \ \  \ \ \ \ \     \ \ \ \ \ \ \ \ \      (610)\ \ \ \ \ \ \ \ \ \  \ \ \ \ \  \ \ \ \ \      \ \ \ \ \     (520)\ \ \ \ \    \ \ \ \ \  \ \ \ \ \  \ \ \ \ \  \ \ \ \   \    (430)\ \ \ \ \ \ \ \ \ \  \ \ \ \ \  \ \ \ \ \   \ \ \ \ \      \ \     (511)\ \ \ \ \ \ \ \ \ \  \ \ \ \ \ \  \ \ \ \ \   \  \ \ \ \ \     (421)\ \ \ \ \  \ \ \ \ \  \ \ \ \ \        \ \ \ \ \  \ \ \ \ \   \ \ \ \ \          \  (322)\ \ \ \ \  \ \ \ \ \   \ \ \ \ \  
\]
\[
\Det[\vt_7^4\vt_0^{-3}\vt_0^{-4}]\ \ \ \ \ \ 
\Pf[\vt_6^3\vt_1^{-2}\vt_0^{-5}]\ \ \ \ \ \ 
\Pf[\vt_5^2\vt_2^{-1}\vt_0^{-5}]\ \ \ \ \ \ 
\Pf[\vt_4^1\vt_3^0\vt_0^{-5}]\ \ \ \ \ \ 
{\Pf[\vt_5^2\vt_1^{-2}\vt_1^{-4}] \atop{+ \Det[\vt_6^2\vt_1^{-2}\vt_0^{-4}]}}\ \ \ \ \ \ 
{\Pf[\vt_4^1\vt_2^{-1}\vt_1^{-4}] \atop{+ \Det[\vt_5^1\vt_2^{-1}\vt_0^{-4}]}}\ \ \ \ \ \ \  
\Det[\vt_3^0\vt_2^{-2}\vt_2^{-3}]\ \ \ \ \ \ \ \ \ \ \ \ \ \ \ \ \ \ 
\]
\

\[
\ \ \ \ \ 13|\bar524\ \ \ \ \ \ \ \  \ \ \ \ \ \ \  23|\bar415\ \ \ \ \ \ \ \  \ \ \ \ \  \   24|\bar3\bar15\ \ \ \ \ \ \ \  \ \ \ \ \ \   \  15|\bar423\ \ \ \ \ \ \ \ \ \  \ \ \ \ \  \ 25|\bar314\ \ \ \ \  \ \ \ \ \ \ \ \  \ \ \ \ \   35|\bar2\bar14\ \ \ \ \ \ \ \  \ \ \ \ \  \ \ \ \ \     45|\bar213
\]
\[
\ \ \ \ \ \ (710)\ \ \ \ \  \ \ \ \ \  \ \ \ \ \    \      (620)\ \ \ \ \ \ \ \ \ \  \ \ \ \ \            (530)\ \ \ \ \    \ \ \ \ \  \ \ \ \ \  \    \   (611)\ \ \ \ \ \ \  \ \ \ \ \  \ \ \ \ \   \   (521)\ \ \ \ \ \ \ \ \ \  \ \ \ \ \  \ \ \ \ \       (431)\ \ \ \ \  \ \ \ \ \  \ \ \ \ \        \ \ \ \ \       (422)
\]
\[
\Pf[\vt_7^4\vt_1^{-2}\vt_0^{-4}]\ \ \ \ \ \ 
\Pf[\vt_6^3\vt_2^{-1}\vt_0^{-5}]\ \ \ \ \ \ 
\Pf[\vt_5^2\vt_3^0\vt_0^{-5}]\ \ \ \ \ \ 
{\Pf[\vt_6^3\vt_1^{-2}\vt_1^{-3}] \atop{+\Pf[\vt_6^3\vt_2^{-2}\vt_0^{-3}]}}\ \ \ \ \ \ 
{\Pf[\vt_5^2\vt_2^{-1}\vt_1^{-4}] \atop{+ \Det[\vt_6^2\vt_2^{-1}\vt_0^{-4}]}}\ \ \ \ \ \ 
{\Pf[\vt_4^1\vt_3^0\vt_1^{-4}] \atop{+  \Det[\vt_5^1\vt_3^0\vt_0^{-4}] }}\ \ \ \ \ \ 
{\Pf[\vt_4^1\vt_2^{-1}\vt_2^{-3}] \atop{+ \Pf[\vt_5^1\vt_2^{-1}\vt_1^{-3}]  \atop{+\Pf[\vt_4^1\vt_3^{-1}\vt_1^{-3}] \atop{+ \Pf[\vt_5^1\vt_3^{-1}\vt_0^{-3}] }}}}
\]

\[
\ \ \ \ \ \ 23|\bar514\ \ \ \ \ \ \ \ \ \ \ \ \ \ \   \ \ \ \ 23|\bar4\bar15\ \ \ \ \ \ \ \ \ \ \ \ \ \ \ \ \ 14|\bar3\bar25\ \ \ \ \    \  \ \ \ \ \ \ \ \ \ \ \ \ \  \ 14|\bar523\ \ \ \ \  \ \  \ \ \ \ \ \ \ \ \ \ \ \  \ \ \ \  \ 25|\bar413\ \ \ \ \  \ \ \ \ \ \ \ \ \ \   \ \   \ \ \ \     \ \ 25|\bar3\bar14\ \ \ \ \ \ \ \ \   \ \ \ \ \  \ \ \ \ \ \ \ \  45|\bar312 \ \ \ \ \ \ \ \ \ \ \ \ \ \ \ \ \ \ \ \ \ \ \ \ \ \ \  45|\bar2\bar13  \ \ \ \ \ \ \ \ \ \ \ \ \ \ \ \ \ \ \ \ \
\]
\[
\ \ \ \ \ \ \ (720)\ \ \ \ \ \ \ \ \ \ \ \ \ \ \   \ \ \ \  \ \  (630)\ \ \ \ \ \ \ \ \ \ \ \ \ \ \ \  \   \ \   (540)\ \ \ \ \      \ \ \ \ \ \ \ \ \ \   \ \    \   \ \ \ \  (711)\ \ \ \ \  \ \  \ \ \ \ \ \ \ \ \ \ \ \  \ \ \ \ \  \ \   (621)\ \ \ \ \  \ \ \ \ \ \ \ \ \ \   \ \   \ \ \ \    \ \ \ \ \   (531)    \ \ \ \ \ \ \ \ \   \ \ \ \ \  \ \ \ \ \    \ \ \ \ \       (522)\ \ \ \ \ \ \ \ \ \ \ \ \ \ \ \ \ \ \ \ \ \ \ \ \ \ \ \ \ \  (432) \ \ \ \ \ \ \ \ \ \ \ \ \ \ \ \ \ \ \ \ 
\]
\[
\Pf[\vt_7^4\vt_2^{-1}\vt_0^{-4}]\ \ \ \  
\Pf[\vt_6^3\vt_3^0\vt_0^{-5}]\ \ \ \   
\Pf[\vt_5^2\vt_4^1\vt_0^{-5}]\ \ \ \   
{\Pf[\vt_7^4\vt_1^{-2}\vt_1^{-3}] \atop{+\Pf[\vt_7^4\vt_2^{-2}\vt_0^{-3}]}}\ \ \ \  
{\Pf[\vt_6^3\vt_2^{-1}\vt_1^{-3}] \atop{+\Pf[\vt_6^3\vt_3^{-1}\vt_0^{-3}]}}\ \ \ \   
{\Pf[\vt_5^2\vt_3^0\vt_1^{-4}] \atop{+  \Det[\vt_6^2\vt_3^0\vt_0^{-4}] }}\ \ \ \   
{\Pf[\vt_5^2\vt_2^{-1}\vt_2^{-2}] \atop{+\Pf[\vt_5^2\vt_3^{-1}\vt_1^{-2}]}} \ \ \ \  
{\Pf[\vt_4^1\vt_3^0\vt_2^{-3}] \atop{+ \Pf[\vt_5^1\vt_3^0\vt_1^{-3}] \atop{+\Pf[\vt_4^1\vt_4^0\vt_1^{-3}] \atop+ \Pf[\vt_5^1\vt_4^0\vt_0^{-3}] }}}
\]

\[
23|\bar5\bar14\ \ \ \ \ \ \ \ \ \ \ \ \  \ \ \ 13|\bar4\bar25\ \ \ \ \ \ \ \ \ \ \ \ \ \ \ \ \ \ \ \ \ 24|\bar513\ \ \ \ \ \ \ \ \ \ \ \ \ \ \ \ \ \ \ 25|\bar4\bar13\ \ \ \ \ \ \ \ \ \ \ \ \ \ \ \ \ \ \ 15|\bar3\bar24\ \ \ \ \ \ \ \ \ \ \ \ \ \ \ \ \ \ \ \ \ 35|\bar412\ \ \ \ \ \ \ \ \ \ \ \ \ \ \ \ 45|\bar3\bar12
\]
\[
(730)\ \ \ \ \ \ \ \ \ \ \ \ \ \ \    \ \ \            (640)\ \ \ \ \    \ \ \ \ \ \ \ \ \ \ \ \ \ \ \     \ \ \       (721)\ \ \ \ \ \ \ \ \ \ \ \ \ \ \   \ \ \  \ \ \   (631)\ \ \ \ \  \ \ \   \ \ \ \ \ \ \ \ \ \      \ \ \         (541)\ \ \ \ \   \ \ \ \ \ \ \ \ \ \ \ \ \ \ \     \ \ \     (622)\ \ \ \ \ \ \ \ \ \ \ \ \ \ \    \ \ \     (532)
\]
\[
\Pf[\vt_7^4\vt_3^0\vt_0^{-4}]\ \ \ \ \ \ \ \ \ 
\Pf[\vt_6^3\vt_4^1\vt_0^{-5}]\ \ \ \ \ \ \
{\Pf[\vt_7^4\vt_2^{-1}\vt_1^{-3}] \atop{+\Pf[\vt_7^4\vt_3^{-1}\vt_0^{-3}]}}\ \ \ \ \ \ \
{\Pf[\vt_6^3\vt_3^0\vt_1^{-3}] \atop{+\Pf[\vt_6^3\vt_4^0\vt_0^{-3}]}} \ \ \ \ \ \ \
{\Pf[\vt_5^2\vt_4^1\vt_1^{-4}] \atop{+  \Det[\vt_6^2\vt_4^1\vt_0^{-4}] }}\ \ \ \ \ \ \
{\Pf[\vt_6^3\vt_2^{-1}\vt_2^{-2}] \atop{+\Pf[\vt_6^3\vt_3^{-1}\vt_1^{-2}]}} \ \ \ \ \ \ \ \ \ 
{\Pf[\vt_5^2\vt_3^0\vt_2^{-2}] \atop{+\Pf[\vt_5^2\vt_4^0\vt_1^{-2}]}}
\]

\newpage

\[
\ \ \ \ 13|\bar5\bar24\ \ \ \ \ \ \ \ \ \ \ \ \ \ \  12|\bar4\bar35\ \ \ \ \ \ \ \ \ \ \ \ \ \ \ \ \ \ \ \ 24|\bar5\bar13\ \ \ \ \ \ \ \ \ \ \ \ \ \ \ \  \ \ 15|\bar4\bar23\ \ \ \ \ \ \ \ \ \ \ \ \ \ \ \ \  \ 34|\bar512\ \ \ \ \  \ \ \ \ \ \ \ \ \ \ \ \ \ \ \ 35|\bar4\bar12\ \ \ \ \ \  \ \ \ \ \ \ \ \ \ 45|\bar3\bar21     \ \ \ \ 
\]
\[
\ \ \ \ (740)\ \ \ \ \ \ \ \ \ \ \ \ \ \ \    \ \           (650)\ \ \ \ \   \ \ \ \ \ \ \ \ \ \ \ \ \ \      \ \ \         (731)\ \ \ \ \  \ \ \ \ \ \ \ \ \ \        \  \  \ \ \      (641)\ \ \ \ \    \ \ \ \ \ \ \ \ \ \   \ \    \ \ \    (722)\ \ \ \ \     \ \ \ \ \ \ \ \ \ \ \ \ \ \   \ \ \   (632)\ \ \ \ \   \ \  \ \    \ \ \ \ \   \ \ \       (542)\ \ \ \  
\]
\[
\Pf[\vt_7^4\vt_4^1\vt_0^{-4}]\ \ \ \ \ \ \ \ 
\Pf[\vt_6^3\vt_5^2\vt_0^{-5}]\ \ \ \ \ \ 
{\Pf[\vt_7^4\vt_3^0\vt_1^{-3}] \atop{+\Pf[\vt_7^4\vt_4^0\vt_0^{-3}]}}\ \ \ \ \ \ 
{\Pf[\vt_6^3\vt_4^1\vt_1^{-3}] \atop{+\Pf[\vt_6^3\vt_5^1\vt_0^{-3}]}}\ \ \ \ \ \ 
{\Pf[\vt_7^4\vt_2^{-1}\vt_2^{-2}] \atop{+\Pf[\vt_7^4\vt_3^{-1}\vt_1^{-2}]}}\ \ \ \ \ \  
{\Pf[\vt_6^3\vt_3^0\vt_2^{-2}] \atop{+\Pf[\vt_6^3\vt_4^0\vt_1^{-2}]}} \ \ \ \ \ \ \ \ 
\Pf[\vt_5^2\vt_4^1\vt_2^{-1}]
\]

\[
12|\bar5\bar34\ \ \ \ \ \ \ \ \ \ \ \ \ 14|\bar5\bar23\ \ \ \ \ \ \ \ \ \ \ \ \ \ \ \ \ \ 15|\bar4\bar32\ \ \ \ \ \ \ \ \ \ \ \ \ 34|\bar5\bar12\ \ \ \ \ \ \ \ \ \ \ \ \ \ \ \ \ \ 35|\bar4\bar21\ \ \ \ \ \ \ \ \ \ \ \ \ 45|\bar3\bar2\bar1
\]
\[
(750)\ \ \ \ \ \ \ \ \ \ \ \ \ \ \                 (741)\ \ \ \ \ \ \ \ \ \ \ \ \ \ \ \ \ \ \ \              (651)\ \ \ \ \ \ \ \ \ \ \ \ \ \ \            (732)\ \ \ \ \  \ \ \ \ \ \ \ \ \ \  \ \ \ \ \           (642)\ \ \ \ \ \ \ \ \ \ \ \ \ \ \              (543)
\]
\[
\Pf[\vt_7^4\vt_5^2\vt_0^{-4}]\ \ \ \ \ \ 
{\Pf[\vt_7^4\vt_4^1\vt_1^{-3}] \atop{+\Pf[\vt_7^4\vt_5^1\vt_0^{-3}]}}\ \ \ \ \ \ 
\Pf[\vt_6^3\vt_5^2\vt_1^{-2}]\ \ \ \ \ \ 
{\Pf[\vt_7^4\vt_3^0\vt_2^{-2}] \atop{+\Pf[\vt_7^4\vt_4^0\vt_1^{-2}]}}\ \ \ \ \ \ 
\Pf[\vt_6^3\vt_4^1\vt_2^{-1}]\ \ \ \ \ \ 
\Pf[\vt_5^2\vt_4^1\vt_3^0]
\]

\[
12|\bar5\bar43\ \ \ \ \ \ \ \  \ \ \ \ \ \ \ \ \ \    14|\bar5\bar32\ \ \ \ \ \ \ \  \ \ \ \ \  34|\bar5\bar21\ \ \ \ \ \ \ \  \ \ \ \ \  25|\bar4\bar31\ \ \ \ \ \ \ \  \ \ \ \ \  \ \ \ \ \  35|\bar4\bar2\bar1
\]
\[
(760)\ \ \ \ \ \ \ \ \ \  \ \ \ \ \ \ \ \ \ \            (751)\ \ \ \ \  \ \ \ \ \  \ \ \ \ \                     (742)\ \ \ \ \  \ \ \ \ \  \ \ \ \ \             (652)\ \ \ \ \   \ \ \ \ \  \ \ \ \ \  \ \ \ \ \                (643)                  
\]
\[
\Pf[\vt_7^4\vt_6^3\vt_0^{-3}]\ \ \ \ \ \ \ \ \ \ 
\Pf[\vt_7^4\vt_5^2\vt_1^{-2}]\ \ \ \ \ \ 
\Pf[\vt_7^4\vt_4^1\vt_2^{-1}]\ \ \ \ \ \ 
\Pf[\vt_6^3\vt_5^2\vt_2^{-1}]\ \ \ \ \ \ \ \ \ \ 
\Pf[\vt_6^3\vt_4^1\vt_3^0]
\]

\[
13|\bar5\bar42\ \ \ \ \ \ \ \ \ \ \ \ \ 24|\bar5\bar31\ \ \ \ \ \ \ \ \ \ \ \ \ 34|\bar5\bar2\bar1\ \ \ \ \ \ \ \ \ \ \ \ \ 25|\bar4\bar3\bar1
\]
\[
(761)\ \ \ \ \   \ \ \ \ \ \ \ \ \ \            (752)\ \ \ \ \    \ \ \ \ \ \ \ \ \ \            (743)\ \ \ \ \     \ \ \ \ \ \ \ \ \ \               (653)                    
\]
\[
\Pf[\vt_7^4\vt_6^3\vt_1^{-2}]\ \ \ \ \ \ \Pf[\vt_7^4\vt_5^2\vt_2^{-1}]\ \ \ \ \ \ \Pf[\vt_7^4\vt_4^1\vt_3^0]\ \ \ \ \ \ \Pf[\vt_6^3\vt_5^2\vt_3^0]
\]

\[
23|\bar5\bar41\ \ \ \ \ \ \ \        \ \  \ \ \ \ \ \ \ \ \   24|\bar5\bar3\bar1\ \ \ \ \ \ \ \     \ \ \ \ \ \ \ \ \ \ \ \     15|\bar4\bar3\bar2
\]
\[
(762)\ \ \ \ \  \ \ \ \ \        \ \ \    \ \ \ \ \ \ \ \             (753)\ \ \ \ \ \ \ \ \ \            \ \ \ \ \ \ \ \ \ \ \ \ \                  (654)               
\]
\[
\Pf[\vt_7^4\vt_6^3\vt_2^{-1}]\ \ \ \ \ \ \ \  \ \ \ \ \ \    \Pf[\vt_7^4\vt_5^2\vt_3^0]\ \ \ \ \  \ \ \ \ \  \ \ \  \ \Pf[\vt_6^3\vt_5^2\vt_4^1]
\]

\[
23|\bar5\bar4\bar1\ \ \ \ \ \ \ \      \ \ \ \ \ \ \ \     14|\bar5\bar3\bar2
\]
\[
(763)\ \ \ \ \  \ \ \ \ \             \ \ \ \ \ \ \ \               (754)           
\]
\[
\Pf[\vt_7^4\vt_6^3\vt_3^0]\ \ \  \ \ \ \ \ \ \Pf[\vt_7^4\vt_5^2\vt_4^1]
\]

\[
13|\bar5\bar4\bar2
\]
\[
(764)               
\]
\[
\Pf[\vt_7^4\vt_6^3\vt_4^1]
\]

\[
12|\bar5\bar4\bar3
\]
\[
(765)           
\]
\[
\Pf[\vt_7^4\vt_6^3\vt_5^2]
\]
}
}



\end{document}